\documentclass[11pt, oneside]{article}

\usepackage{epsfig}
\usepackage{tikz}
\usepackage{url}
\usetikzlibrary{fadings,shapes,arrows.meta,positioning}
\usepackage{hhline}
\usepackage[title]{appendix}
\usepackage{amsmath,amstext,amssymb,amsfonts,amsthm}
\usepackage{mathrsfs}
\usepackage{xcolor}
\usepackage{mathtools}
\usepackage{graphicx}
\usepackage{booktabs}
\usepackage{multicol}
\usepackage{multirow}
\usepackage[labelfont=bf]{caption}
\usepackage{subcaption}
\usepackage{siunitx}

\newcommand{\x}{\mathbf{x}}
\newcommand{\y}{\mathbf{y}}

\newtheorem{lem}{Lemma}[section]
\newtheorem{thm}{Theorem}[section]

\newtheorem{rem}{Remark}[section]

\newtheorem{exmp}{Example}

\numberwithin{equation}{section}
\numberwithin{figure}{section}

\newcommand{\jl}{[\![}
\newcommand{\jr}{]\!]}

\tikzfading[name=fade right, left color=transparent!0, right color=transparent!100]
\tikzfading[name=fade left, right color=transparent!0, left color=transparent!100]
\tikzfading[name=fade down, bottom color=transparent!0, top color=transparent!100]
\tikzfading[name=fade up, top color=transparent!0, bottom color=transparent!100]
\allowdisplaybreaks[4]

\title{
An extension of the box method discrete fracture model (Box-DFM) to include low-permeable barriers with minimal additional degrees of freedom}
\author{
Ziyao Xu\footnotemark[1] \and Dennis Gläser\footnotemark[2]
}
\date{}

\begin{document}

\maketitle
\renewcommand{\thefootnote}{\fnsymbol{footnote}}
\footnotetext[1]{Department of Applied and Computational Mathematics and Statistics,
University of Notre Dame, Notre Dame, IN 46556, USA. E-mail: zxu25@nd.edu}
\footnotetext[2]{Institute for Modelling Hydraulic and Environmental Systems, University of Stuttgart, Pfaffenwaldring 61, 70569 Stuttgart, Germany. E-mail: dennis.glaeser@iws.uni-stuttgart.de}

\begin{center}
\small
\begin{minipage}{0.9\textwidth}
\textbf{Abstract.}
The box method discrete fracture model (Box-DFM) is {\color{black} an important finite volume-based discrete fracture model (DFM)} 
to simulate flows in fractured porous media.
In this paper, we investigate a simple but effective extension of the box method discrete fracture model to include low-permeable barriers.
The method remains identical to the traditional Box-DFM \cite{monteagudo2004control, reichenberger2006mixed} in the absence of barriers.
The inclusion of barriers requires only minimal additional degrees of freedom to accommodate pressure {\color{black}discontinuities}
and necessitates minor modifications to the original coding framework of the Box-DFM.
We use extensive numerical tests on published benchmark problems and comparison with existing finite volume DFMs to demonstrate the validity and
performance of the method.


\medskip

\end{minipage}
\end{center}
\setlength{\parindent}{2em}

\pagenumbering{arabic}

\section{Introduction}\label{Sect:intro}

As a result of geological processes and human activities, cracks are widely distributed in subsurface rocks. 
Depending on the degree of precipitation, cracks can act either as highly permeable, flow-preferential channels or as low-permeability, flow-blocking barriers; both have a significant impact on the hydraulic properties of rocks. 
Understanding fluid flow in fractured rocks is vital for many engineering fields, such as hydrocarbon reservoir engineering, hydraulic fracturing, subsurface contaminant transport and geothermal energy extraction. 
Hence, this area has attracted great interest from researchers.

Due to the extreme contrast in permeability and lateral scale between fractures and their surrounding rocks, accurately and efficiently modeling fluid flow in fractured porous media is highly nontrivial. In recent decades, many effective approaches have been developed to address this challenge. These approaches can be divided into two categories: continuum models and discrete fracture models. In continuum models, fractures are either homogenized into the surrounding matrix, as seen in the effective permeability approach \cite{oda1985permeability, chen2015new}, or integrated into a separate continuum that exchanges flux with the matrix, as in the dual-porosity dual-permeability models \cite{barenblatt1960basic, warren1963behavior} and multi-continuum models \cite{wu1988multiple, leme1998multiple}. Continuum models are preferable when fractures are small, dense, and fully connected within rocks. However, model error becomes significant when large, individual fractures dominate the flow fields. As a compromise, discrete fracture models (DFMs) explicitly account for the effects of individual fractures as lower-dimensional geometries, thus providing a more accurate description of flows in fractured media.

Various DFMs have been developed based on different numerical methods.
Early works on DFMs \cite{noorishad1982upstream, baca1984modelling} were established based on the Galerkin finite element methods (FEM-DFM) on {\color{black} conforming grids where fractures are aligned with the faces of grid cells.}
In these models, the {\color{black} stiffness matrices of flows in fractures} are superposed on that of the porous matrix {\color{black}(see also \cite{kim2000finite, karimi2001numerical})}.
Another approach for simulating subsurface flows is the class of control volume methods. 
Based on the cell-centered control volume methods with two-point flux approximation (TPFA), the TPFA-DFM was developed \cite{karimi2004efficient}. 
This method treats fracture elements as geometrically low-dimensional but conceptually equi-dimensional control volumes and computes their flux exchanges with other control volumes based on transmissibility coefficients. 
It is further extended to MPFA-DFM \cite{sandve2012efficient, ahmed2015control, glaser2017discrete}, based on the multi-point flux approximation (MPFA) \cite{aavatsmark1998discretizationI, aavatsmark1998discretizationII}, to allow for full tensor permeabilities
on grids that are not K-orthogonal. 
The Box methods \cite{bank1987some, hackbusch1989first, helmig1997multiphase}, where the control volumes are vertex-centered boxes and the function space is the {\color{black} linear} Lagrange finite element, 
represent another widely used class of control volume methods. 
Based on the Box methods, the Box-DFM has been developed \cite{monteagudo2004control, reichenberger2006mixed}, where the grids are conforming, and pressure continuity across fractures is assumed, as also studied in \cite{monteagudo2007control, monteagudo2007comparison}. 
Models based on other numerical methods on conforming meshes are also actively studied; for example, the mixed finite element methods (MFEM) \cite{martin2005modeling, hoteit2005multicomponent, hoteit2008efficient, zidane2014efficient}, discontinuous/enriched Galerkin methods (DG/EG) \cite{antonietti2019discontinuous, mozolevski2021high, chen2023discontinuous, kadeethum2020flow}, vertex approximate gradient methods \cite{brenner2016gradient, brenner2017gradient}, and mortar approaches \cite{frih2012modeling, boon2018robust}. 
To alleviate the grid constraints when fracture networks are complex, the non-conforming discrete fracture models are also studied in the literature, e.g., the embedded discrete fracture model (EDFM) \cite{li2008efficient, moinfar2014development, fumagalli2016upscaling, ctene2017projection}, Lagrange multiplier methods \cite{koppel2019lagrange, schadle20193d}, XFEM-DFM \cite{d2012mixed, schwenck2015dimensionally, flemisch2016review}, reinterpreted discrete fracture model \cite{xu2020hybrid, xu2023hybrid, karimi2001numerical}, and immersed finite element DFM \cite{zhao2023discrete}. 

Among the aforementioned discrete fracture models, control volume DFMs are especially popular due to their local mass conservation, flexibility for complex domains, and ease of implementation. 
Cell-centered control volume DFMs have advantages in coupling fracture elements of arbitrary permeability with matrix elements, as additional degrees of freedom (DoFs) are assigned to fracture elements. 
However, cell-centered schemes require more DoFs on simplex grids than vertex-centered schemes. The number of cell elements is roughly twice that of vertices on triangular meshes, and the ratio is {\color{black} approximately six} on tetrahedral meshes. Moreover, MPFA methods typically produce more nonzero entries in the resulting stiffness matrix than box methods \cite{glaser2022comparison}. 
In \cite{haegland2009comparison}, a comparison between MPFA-DFM and Box-DFM was presented. 
The Box method is capable of handling full tensor permeability on unstructured meshes as well as MPFA methods, but with fewer DoFs. Furthermore, the stiffness matrix of Box-DFM is symmetric positive-definite, which allows for the use of many fast linear solvers.
However, the Box-DFM can only be used for high-permeable fractures, as discussed in \cite{haegland2009comparison}.
Extended Box-DFM to account for wider type of fractures in a uniform way has been developed in \cite{glaser2022comparison}, where the equations of fracture flows are established and coupled with the matrix equations based on the interface model established in \cite{martin2005modeling}.
A treatment of assigning individual equations for fracture variables was also developed in \cite{bogdanov2003effective}.
Both methods are not identical to the Box-DFM \cite{monteagudo2004control, reichenberger2006mixed} when only high-permeable fractures present.

In this paper, we propose a simple yet effective extension of the Box-DFM to include low-permeable barriers. The model ignores the flow in the tangential direction along barriers and adopts a broken Lagrange finite element space as the trial functions to accommodate pressure discontinuity. The low-dimensional Darcy's law is adopted to compute the flux across cell interfaces aligned with low-permeable barriers. Since no individual equations and variables are assigned to fractures and barriers, the inclusion of barriers requires only minimal additional degrees of freedom, and the model is identical to the original Box-DFM in the absence of barriers. Moreover, we prove that the resulting stiffness matrix remains symmetric positive-definite, as in the original Box-DFM. In the numerical section, we test the model on published benchmarks to demonstrate that such a simple treatment for low-permeable barriers achieves high accuracy.
The range of validity is also studied in this section.

The rest of the paper is organized as follows. 
In Section \ref{Sect:model}, we describe the model problem of steady-state single-phase flow through a fractured porous medium. 
In Section \ref{Sect:algorithm}, we first provide a review of the classic Box-DFM. 
Subsequently, we introduce the algorithm of the extended Box-DFM to include barriers and explore its properties.
{\color{black}In Section \ref{Sect:tests}, the validity and performance of the extended Box-DFM are tested by published benchmarks before a conclusion and outlook in Section \ref{Sect:summary}.}
The source code to all examples implemented in DuMux can be found at \url{https://git.iws.uni-stuttgart.de/dumux-pub/xu2024}.

\section{Problem description}\label{Sect:model}
We consider the steady-state single-phase flow through a porous medium governed by the Darcy's law.

In {\color{black} $n$ dimensional space, $ n\in\{2, 3\}$}, 
the equi-dimensional model is formulated as
\begin{equation}\label{eq:EquiDmodel}
-\nabla\cdot\left(\mathbf{K}\nabla p\right)=q, \quad {\color{black} \text{in}}~\Omega,
\end{equation}
subject to the boundary conditions
\begin{equation}\label{eq:BdrCond}
p=g_D ~\text{on}~ \Gamma_D, \quad -(\mathbf{K}\nabla p)\cdot\mathbf{n}= g_N ~\text{on}~\Gamma_N:=\partial\Omega\setminus\Gamma_D,
\end{equation}
where {\color{black}$\nabla$} is the gradient operator, {\color{black}$\Omega\subset\mathbb{R}^n$} is an open bounded domain, $\mathbf{K}$ is the permeability tensor of porous media, $p$ is the pressure of fluid, $q$ is the source term, $\Gamma_D$ and $\Gamma_N$ are the Dirichlet and Neumann boundaries, respectively, and $\mathbf{n}$ is the unit outer normal on $\partial\Omega$.

In the absence of fractures and barriers, the problem \eqref{eq:EquiDmodel}, \eqref{eq:BdrCond} is generally an inhomogeneous anisotropic elliptic equation, which can be solved by traditional numerical methods on regular meshes, e.g., the Box-methods \cite{hackbusch1989first, helmig1997multiphase} or the cell-centered finite volume methods \cite{aavatsmark1998discretizationI, aavatsmark1998discretizationII}.
However, in the presence of fractures and barriers, numerical methods based on the equi-dimensional setting may not be well-suited due to the extreme contrast in geometry and permeability between fractures/barriers compared to the porous matrix.

\begin{figure}[!htbp]
  \centering
  \begin{tikzpicture}[scale=1.0]
    \draw[thick] (-4,-2) -- (4,-2);
    \draw[thick] (-4,2) -- (4,2);
    \draw[thick] (-4,-2) -- (-4,2);
    \draw[thick] (4,-2) -- (4,2);
    \draw[thick] (-2,-2) -- (2,2);

    \draw[-Latex] (-1,-1) -- ++(0.7,-0.7) node[above right] {$\mathbf{n}^-$};
    \draw[-Latex] (1,1) -- ++(-0.7,0.7) node[below left] {$\mathbf{n}^+$};    

    \node at (0.3,0) {$\gamma$};

    \coordinate [label=${\Omega}_-$] (b) at(-2,-0.3) ;
    \coordinate [label=${\Omega}_+$] (b) at(2,-0.3) ;
  \end{tikzpicture}
 \caption{The geometry of an interface model {\color{black} in $\mathbb{R}^2$}. $\gamma$ represents the interface of a fracture or barrier. $\Omega^{\pm}$ and $\mathbf{n}^{\pm}$ are the bulk matrix regions and unit outer normal from each region, respectively.}
 \label{fig:hybridDomain}
\end{figure}
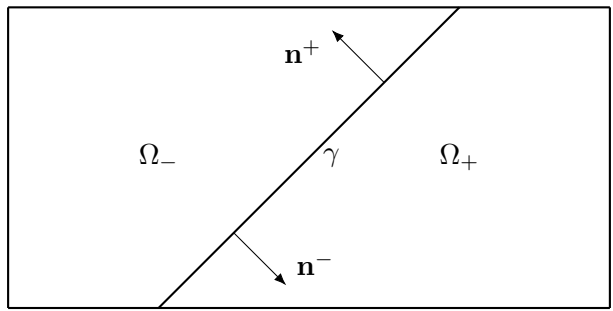

The hybrid-dimensional model, which treats the fractures and barriers as co-dimension one manifolds, provides an efficient way to explicitly take into account the effects of fractures and barriers.
Consider a fractured porous medium $\Omega$ subdivided into $\Omega^{\pm}$ by the $(n-1)$-dimensional
fracture/barrier $\gamma$.
For an illustration, see Figure \ref{fig:hybridDomain}.
{\color{black}In this paper, we consider the governing equation} 
\begin{equation}\label{eq:HybridDmodel1}
-\nabla\cdot\left(\mathbf{K}_m^{\pm}\nabla p^{\pm}\right)=q^{\pm}, \quad \text{in}~\Omega^{\pm},
\end{equation}
in the porous matrix, and the interface condition
\begin{equation}\label{eq:HybridDmodel2}
\begin{cases}
    p_f=p^{\pm},\\
    u_f=-k_f{\color{black}\nabla_{\tau}p_f},\\
    a{\color{black}\nabla_\tau\cdot u_f}=-[\mathbf{K}_m\nabla p\cdot \mathbf{n}],
\end{cases}
{\color{black}\text{on}~}\gamma,
\end{equation}
{\color{black}if $\gamma$ is a high-permeable fracture}, {\color{black}or}
\begin{equation}\label{eq:HybridDmodel3}
-\mathbf{K}^-_m\nabla p^-\cdot\mathbf{n}^-=
\mathbf{K}^+_m\nabla p^+\cdot\mathbf{n}^+=-k_b\frac{p^+-p^-}{a},\quad {\color{black}\text{on}~}\gamma,
\end{equation}
{\color{black}if $\gamma$ is a low-permeable barrier.}
{\color{black} Here, we denote by $\mathbf{K}_{m}^\pm$, $p^{\pm}$ and $q^{\pm}$ the permeability, pressure and source on $\Omega^{\pm}$, respectively.}
{\color{black}Moreover,} $k_f$ is the permeability of high-permeable fractures, $k_b$ is the permeability of low-permeable barriers, 
{\color{black}$\nabla_\tau$ is the gradient operator on the tangent space of $\gamma$,}
$[\mathbf{K}_m\nabla p\cdot \mathbf{n}]=\mathbf{K}_m^-\nabla p^-\cdot \mathbf{n}^- + \mathbf{K}_m^+\nabla p^+\cdot \mathbf{n}^+$ is the jump of normal flux across $\gamma$, $\mathbf{n}^{\pm}$ is the unit outer normal on $\gamma$ from $\Omega^{\pm}$, and $a$ is the aperture of $\gamma$. 

\begin{rem}
The hybrid-dimensional model \eqref{eq:HybridDmodel1}, \eqref{eq:HybridDmodel2}, \eqref{eq:HybridDmodel3} considered in this paper is an important limit case of the general interface model established in \cite{martin2005modeling}.
The pressure continuity in the fractures is assumed, and the tangential flow in the barriers is neglected here.
{\color{black}See discussions about the equations (4.5), (4.6) therein.}

We would like to emphasize that we do not necessarily assume the fracture/barrier $\gamma$ split the whole domain $\Omega$ into two separate pieces.
The geometric setting above is only for the ease of illustration. 
We shall drop the superscripts $\pm$ in {\color{black}the porous matrix equation} \eqref{eq:HybridDmodel1}, and use $\pm$ only for the interface conditions in the next section.
\end{rem}

\section{Algorithm}\label{Sect:algorithm}

{\color{black} In this section, we review the algorithm of the traditional Box-DFM and propose the inclusion of low-permeable barriers. 
The illustration is based on two-dimensional space {\color{black} with triangular grids} for convenience, but the extension to three-dimensional space {\color{black} or quadrilateral grids} is straightforward.
}

\subsection{A review of the traditional Box-DFM}

We denote the computational domain by $\Omega_h$ and consider a Delaunay triangulation $\{T:T\in\mathcal{T}\}$ of the domain, where $T$ are non-overlapping triangular cells such that $\cup_{T\in\mathcal{T}}T=\Omega_h$.
The triangular mesh $\{T:T\in\mathcal{T}\}$ is called the primal mesh.
The corresponding dual mesh $\{B:B\in\mathcal{B}\}$ is obtained through connecting the midpoints of the sides of $T$ to its centroid, see Figure \ref{fig:mesh_a} for illustration, where the boxes $B\in\mathcal{B}$ in the dual mesh are centered at vertices of the primal mesh.

We define the $P^1$-Lagrange finite element space on the mesh $\{T:T\in\mathcal{T}\}$ as
\begin{equation}\label{eq:FEMSpace}
V_h=\{v\in C(\Omega_h): v|_{T}\in P^1(T), T\in\mathcal{T}\},
\end{equation}
where $P^1(T)$ is the linear function space on the cell $T$ and $C(\Omega_h)$ is the space of continuous functions on $\Omega_{h}$.

We first briefly review the box method for the equi-dimensional model \eqref{eq:EquiDmodel}.
We look for the numerical solution $p=\sum_{i=1}^{N}p_{i}\psi_{i}(\mathbf{x})\in V_{h}$, such that
\begin{equation}\label{eq:Box0}
-\int_{\partial B} \mathbf{K}\nabla p \cdot \mathbf{n}ds = \int_{B} q \, dx dy, \quad \forall B\in\mathcal{B},
\end{equation}
where $N$ is the number of boxes in the dual mesh (also the number of vertices in the primal mesh), $p_i$ is the nodal value of the solution at the vertex $\mathbf{x}_i$, and $\psi_i(\mathbf{x})$ is the shape functions serving as the basis of the $P^1$-Lagrange finite element space at $\mathbf{x}_i$.
The scheme \eqref{eq:Box0} is obtained by integrating \eqref{eq:EquiDmodel} over every box $B\in\mathcal{B}$ and applying the Gauss divergence theorem.
For brevity, we omit the discussion of the boundary conditions \eqref{eq:BdrCond} here and henceforth.

We then adopt the traditional Box-DFM as described in \cite{monteagudo2004control, reichenberger2006mixed} to explicitly take into account the flow contributed by the high-permeable fracture.
The fractures are aligned with the edges of the primal mesh and intersect the dual mesh at midpoints of the sides of primal cells, as shown in Figure \ref{fig:mesh_b}.
Under the assumption of pressure continuity and through the superposition of fracture flows, the Box-DFM for the hybrid-dimensional model \eqref{eq:HybridDmodel1}, \eqref{eq:HybridDmodel2} is to find the numerical solution $p\in V_{h}$, such that
\begin{equation}\label{eq:BoxF}
-\int_{\partial B} \mathbf{K}_m\nabla p \cdot \mathbf{n}ds 
-\sum_{j=1}^{k} a k_{f} \frac{\partial p}{\partial \nu_j}(\mathbf{x}^F_j)
= \int_{B} q \, dx dy, \quad \forall B\in\mathcal{B},
\end{equation}
where $\mathbf{x}^F_{j}, j=1,2,\ldots,k$ are points where the fractures intersect $\partial B$, $a$ and $k_{f}$ are aperture and permeability of the fracture, and $\nu_{j}$ is the unit tangential vector of the fracture at $\mathbf{x}^F_{j}$ pointing from inside of $B$ to its outside.

\begin{rem}
There are multiple ways to derive the scheme \eqref{eq:BoxF}. 
For example, one can directly plug the permeability tensor $\mathbf{K}$, established in \cite{xu2020hybrid}, or the effective permeability tensor derived in \cite{oda1985permeability}, into the model \eqref{eq:Box0} to attain \eqref{eq:BoxF}.
{\color{black} The scheme can also be directly derived from the interface model \eqref{eq:HybridDmodel1} and \eqref{eq:HybridDmodel2}, by establishing separate equations for the sub-boxes divided by the fractures, summing up the equations belonging to the same box, substituting the interface condition, and applying the fundamental theorem of calculus on the fracture edge.}
\end{rem}

\begin{figure}[htbp]
 \centering
 \begin{subfigure}[b]{0.6\textwidth}
  \begin{tikzpicture}[scale=1.0]
    \draw[thick] (0,0) -- (0.6,2);
    \draw[thick] (0,0) -- (2.2,0.6);
    \draw[thick] (0,0) -- (1.7,-1.5);
    \draw[thick] (0,0) -- (-0.5,-1.8);
    \draw[thick] (0,0) -- (-2.1,-0.8);
    \draw[thick] (0,0) -- (-1.5,1.5);
    \foreach \x/\y/\num in {0/0/1, 0.6/2/2, 2.2/0.6/3, 1.7/-1.5/4, -0.5/-1.8/5, -2.1/-0.8/6, -1.5/1.5/7}{
        \draw[fill=black] (\x,\y) circle (1pt); 
        \node at (\x,\y) [below left] {\num};   
    }
    \draw[thick] (0.6,2) -- (2.2,0.6);
    \draw[thick] (2.2,0.6) -- (1.7,-1.5);
    \draw[thick] (1.7,-1.5) -- (-0.5,-1.8);
    \draw[thick] (-0.5,-1.8) -- (-2.1,-0.8);
    \draw[thick] (-2.1,-0.8) -- (-1.5,1.5);
    \draw[thick] (-1.5,1.5) -- (0.6,2);
    \draw[dashed] (0.3,1) -- (0.933333333,0.866666667);
    \draw[dashed] (0.933333333,0.866666667) -- (1.1,0.3);
    \draw[dashed] (1.1,0.3) -- (1.3,-0.3);
    \draw[dashed] (1.3,-0.3) -- (0.85,-0.75);
    \draw[dashed] (0.85,-0.75) -- (0.4,-1.1);
    \draw[dashed] (0.4,-1.1) -- (-0.25,-0.9);
    \draw[dashed] (-0.25,-0.9) -- (-0.866666667,-0.866666667);
    \draw[dashed] (-0.866666667,-0.866666667) -- (-1.05,-0.4);
    \draw[dashed] (-1.05,-0.4) -- (-1.2,0.233333333);
    \draw[dashed] (-1.2,0.233333333) -- (-0.75,0.75);
    \draw[dashed] (-0.75,0.75) -- (-0.3,1.166666667);
    \draw[dashed] (-0.3,1.166666667) -- (0.3,1);
    \draw[thick, path fading=fade down] (0.6,2) -- (-0.1,2.7);
    \draw[thick, path fading=fade down] (0.6,2) -- (0.9, 2.8);
    \draw[thick, path fading=fade down] (0.6,2) -- (1.5,2.2);
    \draw[thick, path fading=fade right] (2.2,0.6) -- (2.6,1.5);
    \draw[thick, path fading=fade right] (2.2,0.6) -- (3.3,-0.2);
    \draw[thick, path fading=fade right] (1.7,-1.5) -- (3.0,-1.0);
    \draw[thick, path fading=fade right] (1.7,-1.5) -- (2.9,-2.5);
    \draw[thick, path fading=fade up] (1.7,-1.5) -- (1.3,-2.6);
    \draw[thick, path fading=fade up] (-0.5,-1.8) -- (0.2,-2.7);
    \draw[thick, path fading=fade up] (-0.5,-1.8) -- (-1.8,-2.5);
    \draw[thick, path fading=fade up] (-2.1,-0.8) -- (-2.2,-2.0);
    \draw[thick, path fading=fade up] (-2.1,-0.8) -- (-3.2,-1.0);
    \draw[thick, path fading=fade down] (-2.1,-0.8) -- (-3.3,0.5);
    \draw[thick, path fading=fade left] (-1.5,1.5) -- (-2.8,1.0);
    \draw[thick, path fading=fade left] (-1.5,1.5) -- (-2.8,2.0);
    \draw[thick, path fading=fade down] (-1.5,1.5) -- (-1.0,2.4);            
  \end{tikzpicture}  
  \caption{Unfractured porous media. Solid lines denote the primal mesh, dashed lines denote a box centered at the vertex $1$ in the dual mesh.}\label{fig:mesh_a}
 \end{subfigure}
 \begin{subfigure}[b]{0.6\textwidth}
  \begin{tikzpicture}[scale=1.0]
    \draw[thick] (0,0) -- (0.6,2);
    \draw[thick] (0,0) -- (2.2,0.6);
    \draw[ultra thick, blue] (0,0) -- (1.7,-1.5);
    \draw[thick] (0,0) -- (-0.5,-1.8);
    \draw[thick] (0,0) -- (-2.1,-0.8);
    \draw[ultra thick, blue] (0,0) -- (-1.5,1.5);
    \foreach \x/\y/\num in {0/0/1, 0.6/2/2, 2.2/0.6/3, 1.7/-1.5/4, -0.5/-1.8/5, -2.1/-0.8/6, -1.5/1.5/7}{
        \draw[fill=black] (\x,\y) circle (1pt); 
        \node at (\x,\y) [below left] {\num};   
    }
    \draw[thick] (0.6,2) -- (2.2,0.6);
    \draw[thick] (2.2,0.6) -- (1.7,-1.5);
    \draw[thick] (1.7,-1.5) -- (-0.5,-1.8);
    \draw[thick] (-0.5,-1.8) -- (-2.1,-0.8);
    \draw[thick] (-2.1,-0.8) -- (-1.5,1.5);
    \draw[thick] (-1.5,1.5) -- (0.6,2);
    \draw[dashed] (0.3,1) -- (0.933333333,0.866666667);
    \draw[dashed] (0.933333333,0.866666667) -- (1.1,0.3);
    \draw[dashed] (1.1,0.3) -- (1.3,-0.3);
    \draw[dashed] (1.3,-0.3) -- (0.85,-0.75);
    \draw[dashed] (0.85,-0.75) -- (0.4,-1.1);
    \draw[dashed] (0.4,-1.1) -- (-0.25,-0.9);
    \draw[dashed] (-0.25,-0.9) -- (-0.866666667,-0.866666667);
    \draw[dashed] (-0.866666667,-0.866666667) -- (-1.05,-0.4);
    \draw[dashed] (-1.05,-0.4) -- (-1.2,0.233333333);
    \draw[dashed] (-1.2,0.233333333) -- (-0.75,0.75);
    \draw[dashed] (-0.75,0.75) -- (-0.3,1.166666667);
    \draw[dashed] (-0.3,1.166666667) -- (0.3,1);
    \draw[thick, path fading=fade down] (0.6,2) -- (-0.1,2.7);
    \draw[thick, path fading=fade down] (0.6,2) -- (0.9, 2.8);
    \draw[thick, path fading=fade down] (0.6,2) -- (1.5,2.2);
    \draw[thick, path fading=fade right] (2.2,0.6) -- (2.6,1.5);
    \draw[thick, path fading=fade right] (2.2,0.6) -- (3.3,-0.2);
    \draw[thick, path fading=fade right] (1.7,-1.5) -- (3.0,-1.0);
    \draw[ultra thick, blue, path fading=fade right] (1.7,-1.5) -- (2.9,-2.5);
    \draw[thick, path fading=fade up] (1.7,-1.5) -- (1.3,-2.6);
    \draw[thick, path fading=fade up] (-0.5,-1.8) -- (0.2,-2.7);
    \draw[thick, path fading=fade up] (-0.5,-1.8) -- (-1.8,-2.5);
    \draw[thick, path fading=fade up] (-2.1,-0.8) -- (-2.2,-2.0);
    \draw[thick, path fading=fade up] (-2.1,-0.8) -- (-3.2,-1.0);
    \draw[thick, path fading=fade down] (-2.1,-0.8) -- (-3.3,0.5);
    \draw[thick, path fading=fade left] (-1.5,1.5) -- (-2.8,1.0);
    \draw[ultra thick, blue, path fading=fade left] (-1.5,1.5) -- (-2.8,2.0);
    \draw[thick, path fading=fade down] (-1.5,1.5) -- (-1.0,2.4);
    \node at (0.85,-0.75) [circle,fill,inner sep=1pt,label=right:$\mathbf{x}^F_1$] {};
    \node at (-0.75,0.75) [circle,fill,inner sep=1pt,label=above:$\mathbf{x}^F_2$] {};
  \end{tikzpicture}  
  \caption{Porous media with high-permeable fractures. Blue lines denote fractures, which intersect the box at $\mathbf{x}_{1}^F$ and $\mathbf{x}_{2}^{F}$.}\label{fig:mesh_b}
 \end{subfigure}
  \begin{subfigure}[b]{0.6\textwidth}
  \begin{tikzpicture}[scale=1.0]
    \foreach \x/\y/\num in {2.2/0.6/3, -0.5/-1.8/5, -2.1/-0.8/6, -1.5/1.5/7}{
        \draw[fill=black] (\x,\y) circle (1pt); 
        \node at (\x,\y) [below left] {\num};   
    }
    \node at (-0.07,0.07) [circle,fill,inner sep=1pt,label=above left:$1'$] {};
    \node at (0.07,0.07) [circle,fill,inner sep=1pt,label=above right:$1''$] {};
    \node at (-0.02,-0.02) [circle,fill,inner sep=1pt,label=below:$1^*$] {};
    \node at (0.53,2) [circle,fill,inner sep=1pt,label=left:$2'$] {};
    \node at (0.67,2) [circle,fill,inner sep=1pt,label=right:$2''$] {};
    \node at (1.67,-1.43) [circle,fill,inner sep=1pt,label=above right:$4'$] {};
    \node at (1.58,-1.52) [circle,fill,inner sep=1pt,label=below left:$4''$] {};
    \draw[thick,red] (-0.006666667,0.04) -- (0.6,2); 
    \draw[thick,red] (-0.006666667,0.04) -- (1.625,-1.475); 
    \draw[thick,red] (-0.006666667,0.04) -- (-2.1,-0.8); 
    \draw[thick, red, path fading=fade down] (0.6,2) -- (0.9, 2.8); 
    \draw[thick, red, path fading=fade right] (1.625,-1.475) -- (2.9,-2.55); 
    \draw[thick] (-0.07,0.07) -- (0.53,2); 
    \draw[thick] (0.07,0.07) -- (0.67,2); 
    \draw[thick] (0.07,0.07) -- (2.2,0.6); 
    \draw[thick] (0.07,0.07) -- (1.67,-1.43);
    \draw[thick] (-0.02,-0.02) -- (1.58,-1.52); 
    \draw[thick] (-0.02,-0.02) -- (-0.5,-1.8); 
    \draw[thick] (-0.07,0.07) -- (-2.1,-0.8); 
    \draw[thick] (-0.02,-0.02) -- (-2.1,-0.8); 
    \draw[thick] (-0.07,0.07) -- (-1.5,1.5); 
    \draw[thick] (0.67,2) -- (2.2,0.6); 
    \draw[thick] (2.2,0.6) -- (1.65,-1.45); 
    \draw[thick] (1.58,-1.52) -- (-0.5,-1.8); 
    \draw[thick] (-0.5,-1.8) -- (-2.1,-0.8); 
    \draw[thick] (-2.1,-0.8) -- (-1.5,1.5); 
    \draw[thick] (-1.5,1.5) -- (0.55,2);
    \draw[dashed] (0.3,1) -- (0.933333333,0.866666667);
    \draw[dashed] (0.933333333,0.866666667) -- (1.1,0.3);
    \draw[dashed] (1.1,0.3) -- (1.3,-0.3);
    \draw[dashed] (1.3,-0.3) -- (0.85,-0.75);
    \draw[dashed] (0.85,-0.75) -- (0.4,-1.1);
    \draw[dashed] (0.4,-1.1) -- (-0.25,-0.9);
    \draw[dashed] (-0.25,-0.9) -- (-0.866666667,-0.866666667);
    \draw[dashed] (-0.866666667,-0.866666667) -- (-1.05,-0.4);
    \draw[dashed] (-1.05,-0.4) -- (-1.2,0.233333333);
    \draw[dashed] (-1.2,0.233333333) -- (-0.75,0.75);
    \draw[dashed] (-0.75,0.75) -- (-0.3,1.166666667);
    \draw[dashed] (-0.3,1.166666667) -- (0.3,1);
    \draw[thick, path fading=fade down] (0.53,2) -- (-0.1,2.7); 
    \draw[thick, path fading=fade down] (0.53,2) -- (0.83, 2.8); 
    \draw[thick, path fading=fade down] (0.67,2) -- (0.97, 2.8); 
    \draw[thick, path fading=fade down] (0.67,2) -- (1.5,2.2); 
    \draw[thick, path fading=fade right] (2.2,0.6) -- (2.6,1.5);
    \draw[thick, path fading=fade right] (2.2,0.6) -- (3.3,-0.2);
    \draw[thick, path fading=fade right] (1.67,-1.43) -- (3.0,-1.0); 
    \draw[thick, path fading=fade right] (1.67,-1.43) -- (2.92,-2.5); 
    \draw[thick, path fading=fade right] (1.58,-1.52) -- (2.88,-2.6); 
    \draw[thick, path fading=fade up] (1.58,-1.52) -- (1.3,-2.6);
    \draw[thick, path fading=fade up] (-0.5,-1.8) -- (0.2,-2.7);
    \draw[thick, path fading=fade up] (-0.5,-1.8) -- (-1.8,-2.5);
    \draw[thick, path fading=fade up] (-2.1,-0.8) -- (-2.2,-2.0);
    \draw[thick, path fading=fade up] (-2.1,-0.8) -- (-3.2,-1.0);
    \draw[thick, path fading=fade down] (-2.1,-0.8) -- (-3.3,0.5);
    \draw[thick, path fading=fade left] (-1.5,1.5) -- (-2.8,1.0);
    \draw[thick, path fading=fade left] (-1.5,1.5) -- (-2.8,2.0);
    \draw[thick, path fading=fade down] (-1.5,1.5) -- (-1.0,2.4);
  \end{tikzpicture}
  \caption{Porous media with low-permeable barriers. Red lines denote barriers, which subdivide the box into three sub-boxes. Degrees of freedom of vertices on barriers are duplicated accordingly to account for pressure discontinuity.}\label{fig:mesh_c}
 \end{subfigure}
 \caption{The meshes used in box methods for unfractured porous media, porous media with high-permeable fractures and low-permeable barriers.}
 \label{fig:mesh}
\end{figure}
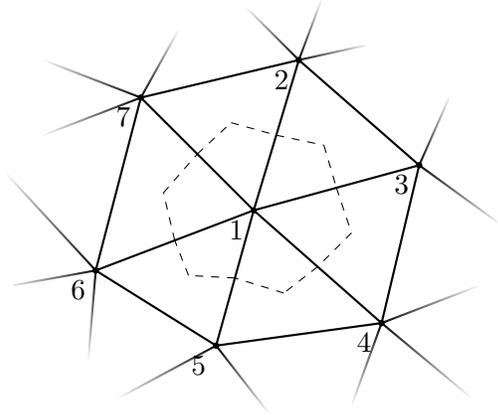
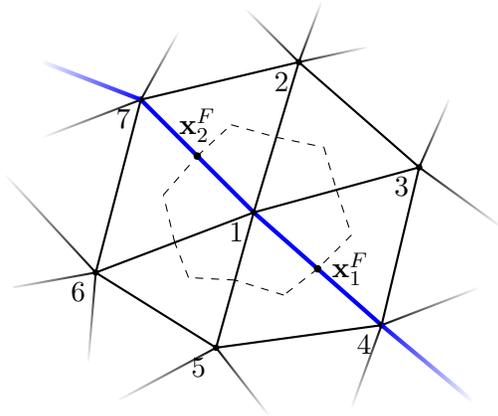
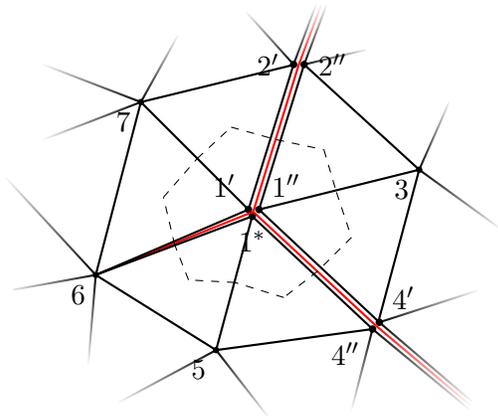

\subsection{The inclusion of low-permeable barriers}
Similar to the case with fractures, we assume that the barriers align with the edges of cells in the primal mesh.

We define the extended $P^1$-Lagrange finite element space on the fitting mesh $\{T:T\in\mathcal{T}\}$ as
\begin{equation}\label{eq:ExFEMSpace}
\overline{V}_h=\{v\in C(\Omega_h\setminus\Gamma): v|_{T}\in P^1(T), T\in\mathcal{T}\},
\end{equation}
where $\Gamma=\cup_{i}\gamma_i$ is the union of all barriers and $C(\Omega_h\setminus\Gamma)$ is the space of continuous functions on $\Omega_h\setminus\Gamma$, namely the functions admit discontinuities only on $\Gamma$.
Compared with the finite element space $V_{h}$ defined in \eqref{eq:FEMSpace}, the extended space $\overline{V}_h$ in \eqref{eq:ExFEMSpace} is endowed with additional degrees of freedom at the vertices on barriers to accommodate the discontinuity of pressure

In the dual mesh, the boxes $B\in\mathcal{B}$ may be subdivided into separate sub-boxes by the barriers $\Gamma$, as illustrated in Figure \ref{fig:mesh_c}.
Consequently, we define the extended collection of boxes $\overline{\mathcal{B}}$ as encompassing both the original boxes $B\in\mathcal{B}$ that remain undivided by $\Gamma$, and the sub-boxes resulting from such subdivision.
As a result, the number of boxes in $\overline{\mathcal{B}}$ is the same as the degrees of freedom of $\overline{V}_h$. 
It is worth making a special note about the situation of the tips of $\Gamma$ to eliminate ambiguity. 
According to the definition of $\overline{V}_h$, the vertices at the tips of $\Gamma$ have only one degree of freedom. 
Correspondingly, the boxes $B\in\mathcal{B}$ that enclose these vertices are regarded as undivided.

The extended Box-DFM for the hybrid-dimensional model \eqref{eq:HybridDmodel1}, \eqref{eq:HybridDmodel3} is to find the numerical solution $p\in\overline{V}_h$, such that
\begin{equation}\label{eq:BoxB}
-\int_{\partial B\setminus\Gamma} \mathbf{K}_m\nabla p \cdot \mathbf{n}ds-\int_{\partial B\cap\Gamma}k_b\frac{p^+-p^-}{a}ds=\int_{B} q\,dxdy,\quad \forall B\in\overline{\mathcal{B}},
\end{equation}
where $p^{-}$ and $p^{+}$ are pressure on $\partial B\cap\Gamma$ from the inside and outside of $B$, respectively.
For those undivided boxes $B\in\mathcal{\overline{B}}\cap\mathcal{B}$, it's clear that $\partial B\setminus\Gamma=B$ and $\partial B\cap\Gamma=\emptyset$, thus the equation is identical to \eqref{eq:Box0}.

\begin{rem}\label{rm:1}
People have long noticed \cite{bank1987some} that the stiffness matrix of the box method \eqref{eq:Box0} is identical to the stiffness matrix of the linear finite element method; see a concise proof in \cite{hackbusch1989first}. 
Following the same approach therein, one can prove that the stiffness matrix of Box-DFM \eqref{eq:BoxF} is identical to that of FEM-DFM \cite{karimi2001numerical}.
Therefore, the stiffness matrices of the box method \eqref{eq:Box0} and Box-DFM \eqref{eq:BoxF} are symmetric positive-definite, as in the finite element method, which makes them suitable for many fast solvers.
Fortunately, the stiffness matrix of the extended Box-DFM \eqref{eq:BoxB} is still symmetric positive-definite.
A proof is attached in Appendix \ref{app:SPD}.

\end{rem}

Finally, we combine the schemes \eqref{eq:BoxF} and \eqref{eq:BoxB} to formulate our extended Box-DFM for fracture and barrier networks: Find $p\in\overline{V}_h$, such that
\begin{equation}\label{eq:BoxFB}
-\int_{\partial B\setminus\Gamma}\mathbf{K}_m\nabla p \cdot \mathbf{n}ds
-\sum_{j=1}^{k} a k_{f} \frac{\partial p}{\partial \nu_j}(\mathbf{x}^F_j)
-\int_{\partial B\cap\Gamma}k_b\frac{p^+-p^-}{a}ds=\int_{B} q\,dxdy,\quad \forall B\in\overline{\mathcal{B}},
\end{equation}
where $\mathbf{x}^F_{j}, j=1,2,\ldots,k$ are points where the fractures intersect $\partial B$.

\begin{rem}\label{rm:2}
If the fracture and barrier intersect at a vertex, the treatment there depends on whether the fracture penetrates the barrier, or the barrier cuts the fracture. In the former case, there is only one degree of freedom, one box, and one equation associated with this vertex. In the latter case, degrees of freedom are assigned to each sub-box cut by the barriers, and equations are established for each sub-box.
{\color{black}
For an illustration of the differences between these two approaches, {\color{black} we refer the reader to}
Example \ref{ex:complex} in the numerical section, where both treatments are implemented and their outcomes demonstrated.}
\end{rem}

\section{Numerical experiments}\label{Sect:tests}

In this section, we test the effectiveness and range of validity of our extended Box-DFM {\color{black} on a set of test cases including published benchmark problems},
see e.g.,  \cite{flemisch2018benchmarks, angot2009asymptotic, zhao2023discrete, glaser2022comparison}.
{\color{black} If not otherwise stated, the units used in the examples adhere to SI, i.e., $\mathbf{K}$ in $\si{m}^2$, $p$ in $\si{Pa}$, $a$ in $\si{m}$, etc.}
The results are compared with those obtained from another extension of Box-DFM developed in \cite{glaser2022comparison}, which established separate equations and assigned individual degrees of freedom for arbitrary types of fractures. 
To distinguish between the two methods in the presentation, we refer to our method as 'box-dfm' and the method developed in \cite{glaser2022comparison} as 'ebox-dfm', following the naming convention the authors already established in the literature.

The problems are presented in roughly increasing order of geometrical complexity. 
For simplicity, we set the source term $q=0$ in all examples except the convergence test.
Since the method is identical to the traditional Box-DFM in the absence of barriers, we focus only on 
benchmarks involving low-permeable barriers.
In all the tests, except for the last one, {\color{black}we assume equal normal and tangential permeabilities in the barriers}.

\begin{exmp}\label{ex:convergence}{Convergence test}

In this experiment, we test the convergence of the method \eqref{eq:BoxB} to the model problem \eqref{eq:HybridDmodel1}, \eqref{eq:HybridDmodel3}
by considering the following setting: $\mathbf{K}_m=\mathbf{I}$,
$\frac{k_b}{a}=1$, $\Omega^{-}=[0,\frac12]\times[0,1]$, $\Omega^{+}=[\frac12,1]\times[0,1]$, $\gamma:\{x=\frac12, 0\leq y\leq 1\}$, and 
\begin{equation}\label{eq:sol_source}
\begin{cases}
p^{-}=\sin(x)\sin(y),~ q^{-}=2\sin(x)\sin(y), &(x,y)\in\Omega^{-},\\
p^{+}=p^{-}+\cos(\frac12)\sin(y),~ q^{+}=q^{-}+\cos(\frac12)\sin(y), &(x,y)\in\Omega^{+},
\end{cases}
\end{equation}
with the corresponding Dirichlet boundary condition.

We test the algorithm on unstructured grids with different level of refinement.
The grid used for the lowest level of refinement, consisting of 256 triangles, is shown in Figure \ref{fig:grid_converge_test}.
The $L^2$-error in porous matrix and the orders of convergence are presented in Table \ref{tab:convergence}, where $p_h$ denotes the numerical solution and $p$ is the exact solution given in \eqref{eq:sol_source}.
From the table, we can clearly observe the {\color{black} expected} 
order of convergence of the box method.

\begin{figure}[!hbpt]
 \centering
 \includegraphics[width=0.3\textwidth]{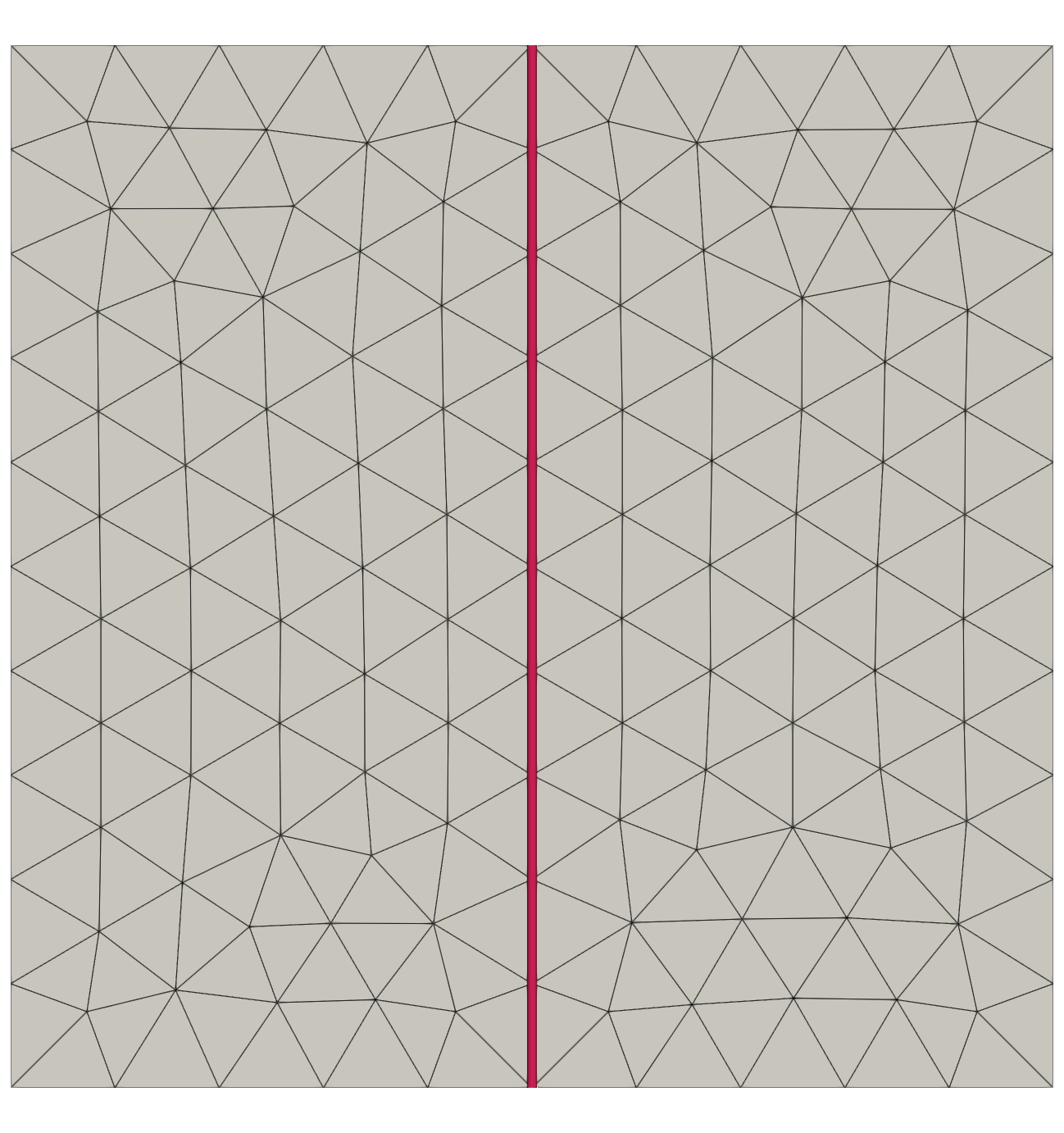}
  \setlength{\abovecaptionskip}{2pt} 
 \caption{\textbf{Example \ref{ex:convergence}: convergence test.} 
 The grid used for the lowest refinement, consisting of $256$ triangles. The red line represents the barrier interface.}
 \label{fig:grid_converge_test}
\end{figure}

\begin{table}[!hbpt]
\centering
\begin{tabular}{c c c}
\hline\hline
$i$ & $||p_h-p||_{2}$ & Order \\ \hline\hline

0 & 5.39e-04 & - \\ \hline
1 & 1.36e-04 & 1.99 \\ \hline
2 & 3.40e-05 & 2.00 \\ \hline
3 & 8.50e-06 & 2.00 \\ \hline
4 & 2.13e-06 & 2.00 \\ \hline
5 & 5.31e-07 & 2.00 \\ \hline
\end{tabular}
\caption{\textbf{Example \ref{ex:convergence}: convergence test.} 
{\color{black} $L^2$ errors evaluated on $\Omega^{\pm}$ on different levels of refinement}
}\label{tab:convergence}
\end{table}
\end{exmp}

\begin{exmp}\label{ex:single}
{Single barrier}

In this experiment, we examine the performance of our method in cases involving a single barrier. 
We consider two scenarios: the first involves a vertical barrier from (0.5, 0.5) to (0.5, 1), see the same test in references \cite{angot2009asymptotic, zhao2023discrete}; the second tests a slanted barrier from (0.25, 0.75) to (0.75, 0.25). 
In both scenarios, the computational domain is defined as $\Omega=[0,1]^2$, and the permeability tensor in the porous matrix is $\mathbf{K}_m=\mathbf{I}$. 
The ratio of permeability to aperture of barriers is given by $\frac{k_b}{a}=10^{-5}$. 
The top and bottom boundaries are impermeable ($g_N=0$), while the left and right boundaries are set with Dirichlet conditions, $g_D=0$ and $g_D=1$, respectively.

We conduct the computation using our extended box-dfm and the ebox-dfm developed in \cite{glaser2022comparison} on identical grids. 
The pressure contours obtained from our method, alongside the background mesh, are presented in Figure \ref{fig:single_contour} (a), (b).
The differences in pressure between the two methods are illustrated in Figure \ref{fig:single_contour} (c), (d). 
The comparison reveals that the discrepancies in the results of the two methods are tiny.
Additionally, we present a comparison of the pressure profiles obtained by the two methods along specific slices in Figure \ref{fig:single_plot}. 
This comparison also exhibits an excellent agreement between two methods.

\begin{figure}[!htbp]
 \centering
 \begin{subfigure}[b]{0.3\textwidth}
  \includegraphics[width=\textwidth]{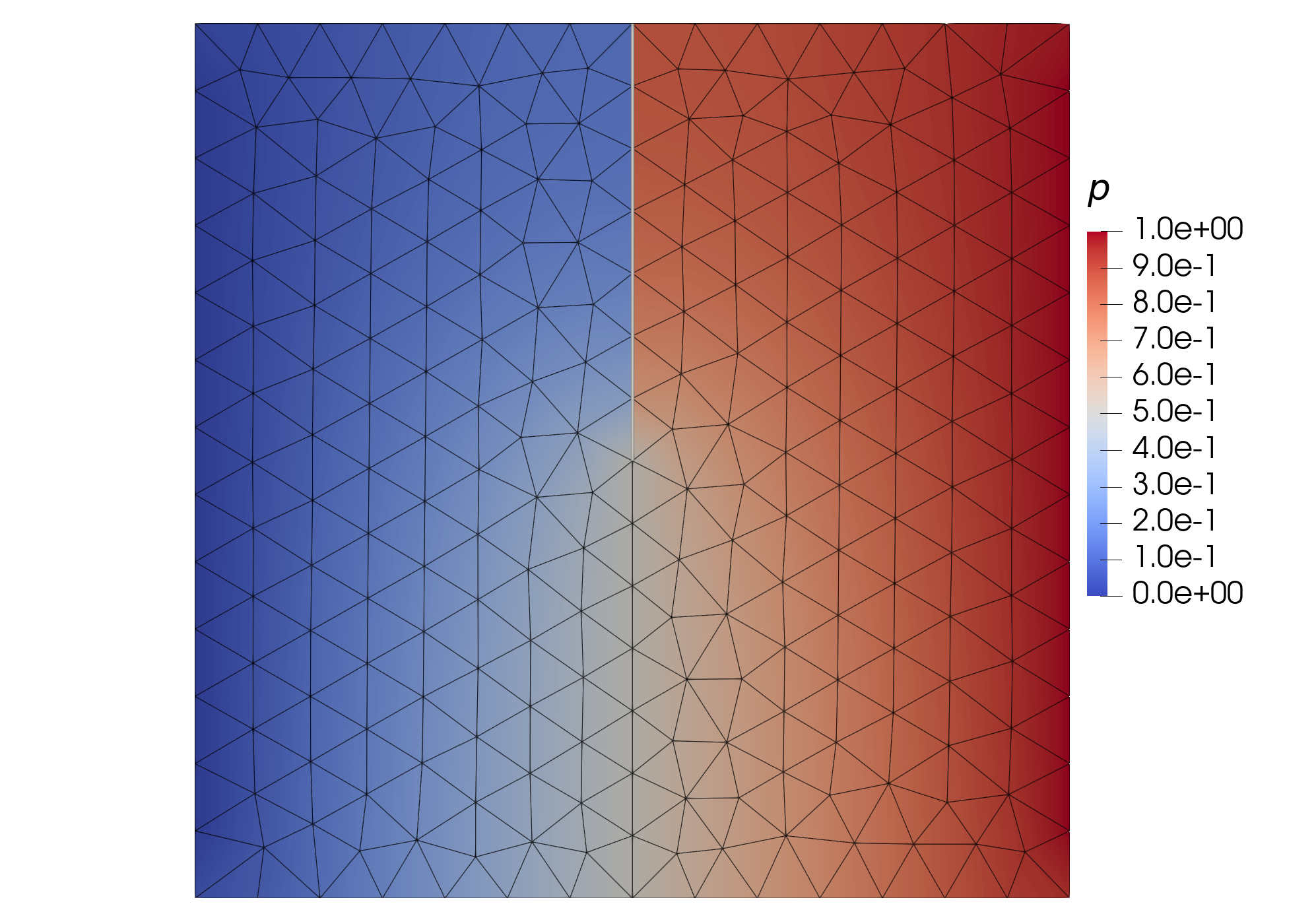}
  \caption{vertical barrier}
 \end{subfigure}
 \begin{subfigure}[b]{0.3\textwidth}
  \includegraphics[width=\textwidth]{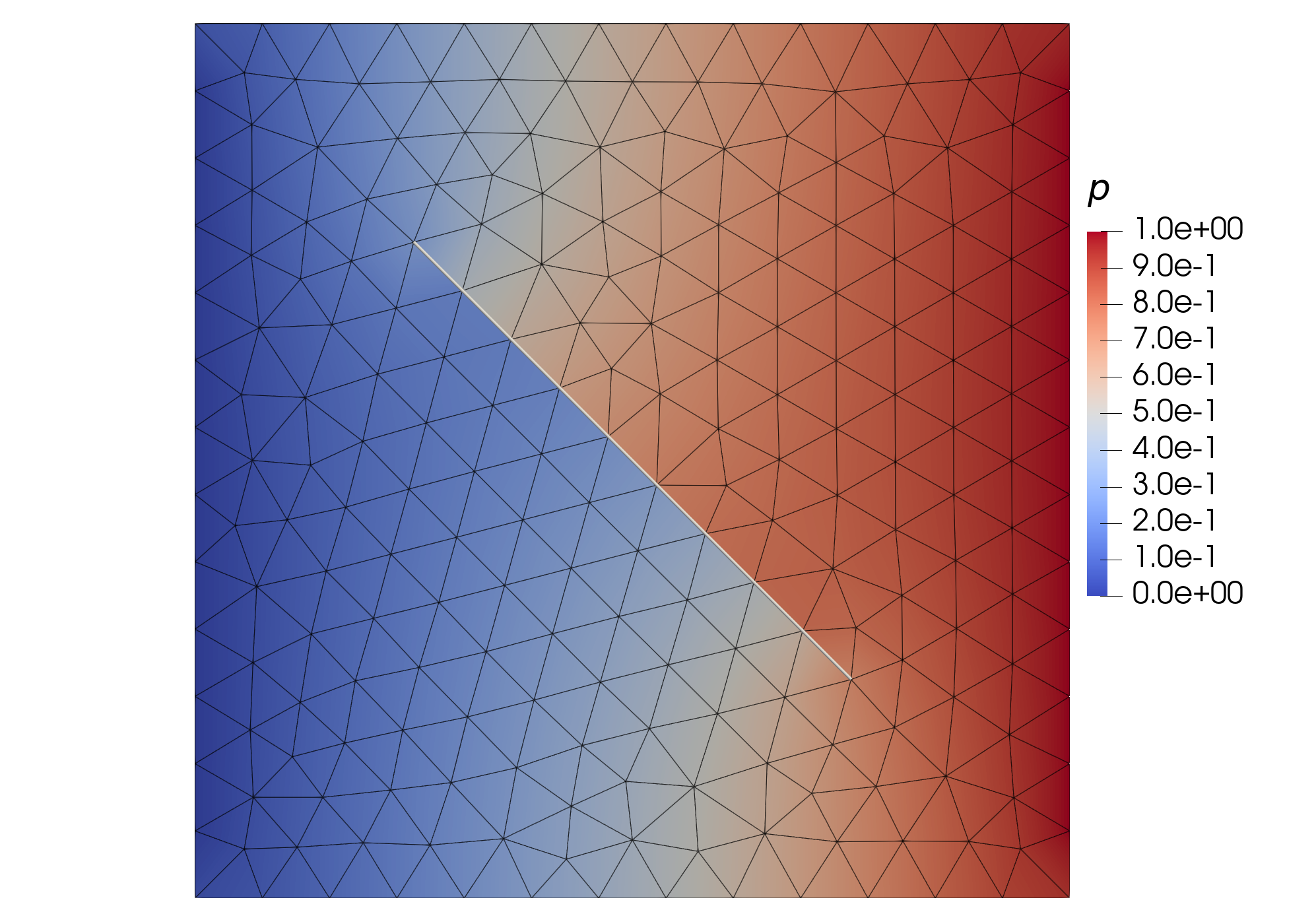}
  \caption{slanted barrier}
 \end{subfigure}

 \begin{subfigure}[b]{0.3\textwidth}
  \includegraphics[width=\textwidth]{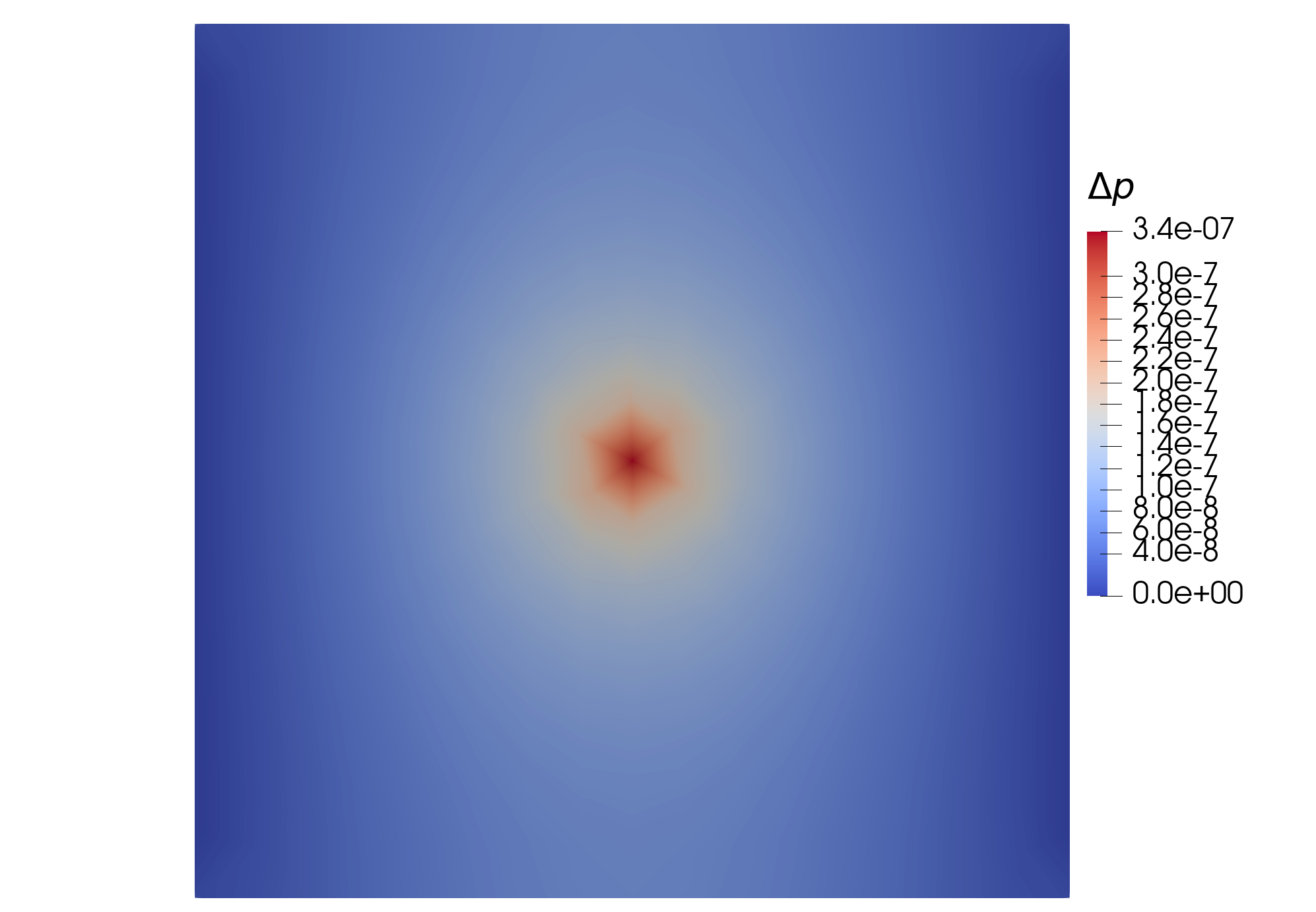}
  \caption{vertical barrier}
 \end{subfigure}
 \begin{subfigure}[b]{0.3\textwidth}
  \includegraphics[width=\textwidth]{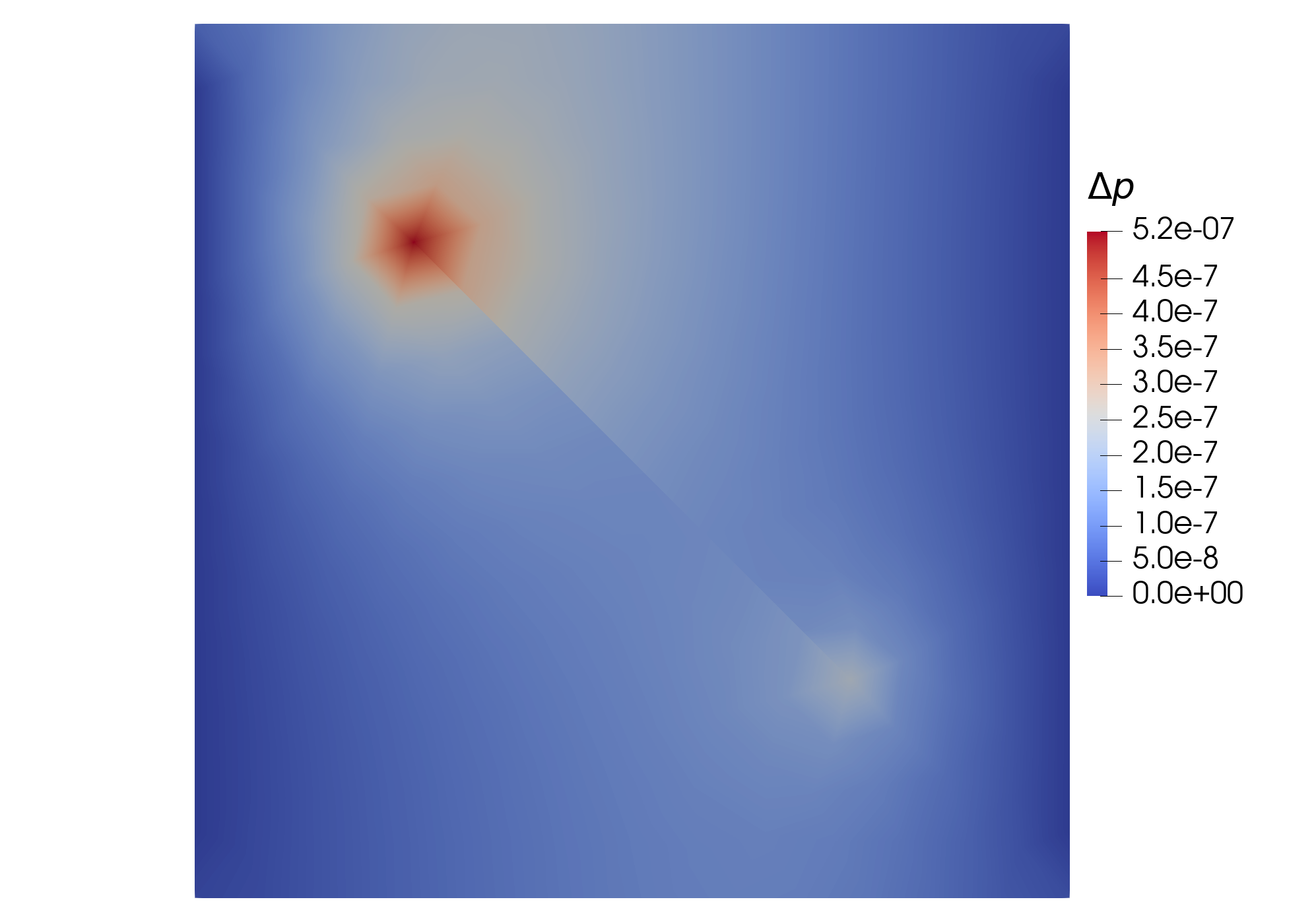}
  \caption{slanted barrier}
 \end{subfigure}
 \caption{\textbf{Example \ref{ex:single}: single barrier.}  Visualization of the solution $p$ (row $1$) obtained with the proposed method, and the difference $\Delta p$ (row $2$) to the solution obtained with the ebox-dfm in \cite{glaser2022comparison}. 
 In the first scenario (vertical barrier), the grid contains $253$ vertices and $450$ triangles; in the second scenario (slanted barrier), it contains $229$ vertices and $404$ triangles.}
 \label{fig:single_contour}
\end{figure}

\begin{figure}[!htbp]
 \centering
 \begin{subfigure}[b]{0.3\textwidth}
  \includegraphics[width=\textwidth]{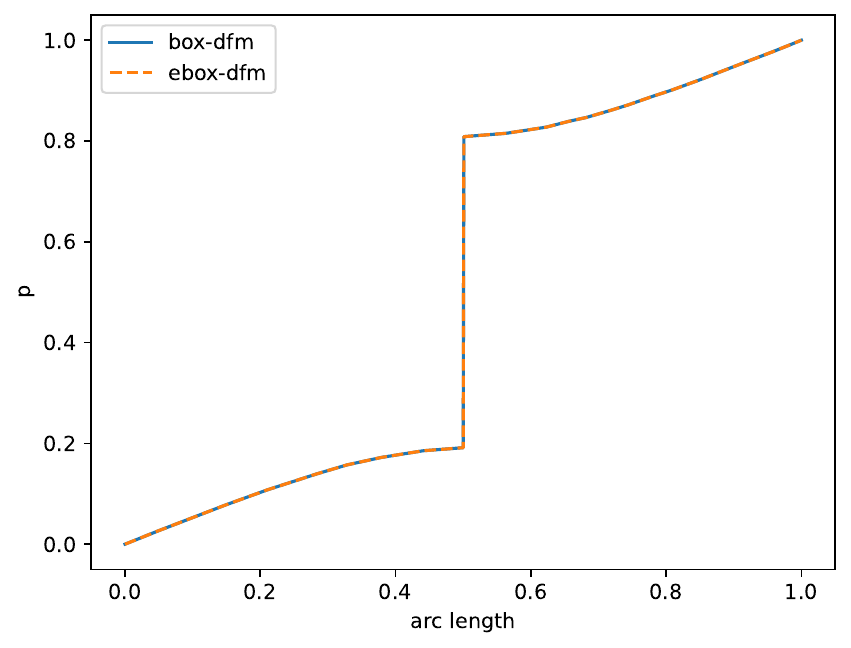}
  \caption{vertical barrier}
 \end{subfigure}
 \begin{subfigure}[b]{0.3\textwidth}
  \includegraphics[width=\textwidth]{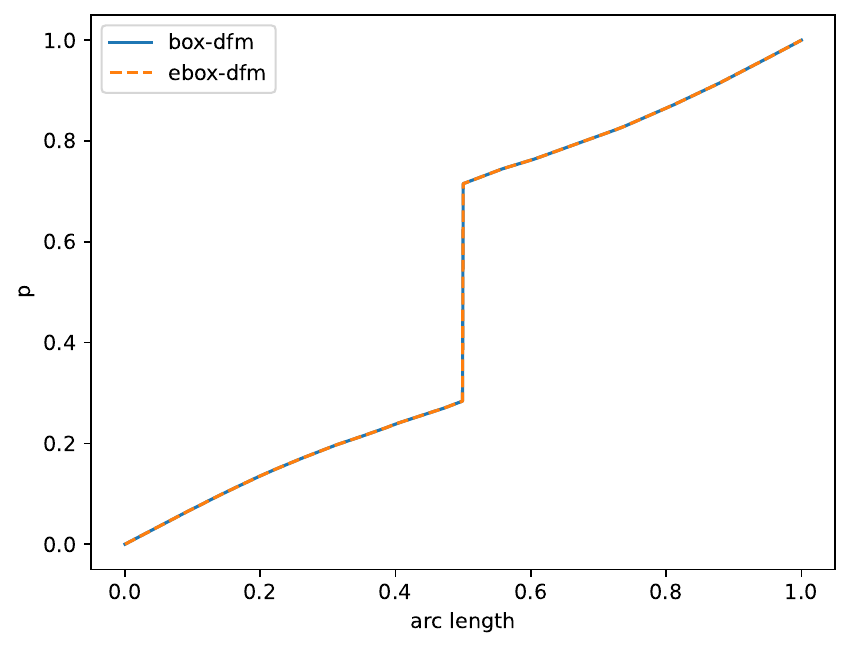}
  \caption{slanted barrier}
 \end{subfigure}
 \caption{\textbf{Example \ref{ex:single}: single barrier.} 
 Pressure profiles along the slice (0, 0.75) -- (1, 0.75) for the vertical barrier and along the slice (0, 0.5) -- (1, 0.5) for the slanted barrier.}
 \label{fig:single_plot}
\end{figure}
\end{exmp}

\begin{exmp}\label{ex:regular} 
{Regular barrier network}

In this experiment, we test the performance of our method on a regular barrier network.
The computational domain is defined as $\Omega=[0,1]^2$, and the permeability tensor of the porous matrix is given by $\mathbf{K}_m=\mathbf{I}$. 
The domain encompasses six barriers, with coordinates detailed in \cite{flemisch2018benchmarks}. 
Each barrier has a uniform aperture $a=10^{-4}$ and a permeability $k_b=10^{-4}$. 
The top and bottom boundaries are impermeable; the left boundary is set as a Neumann boundary with $g_N=1$; and the right boundary is defined by a Dirichlet condition with $g_D=1$.

We conduct computations using both our extended box-dfm and the ebox-dfm on three different grids, with a varying level of refinement. 
The pressure contours produced by our method, alongside the background mesh, and the differences in pressure between the two methods, are illustrated in Figure \ref{fig:regular_contour}. 
Furthermore, we present pressure slices along the line from (0, 0.1) to (0.9, 1) in Figure \ref{fig:regular_plot}, which are compared with the reference solution \cite{flemisch2018benchmarks} obtained using the mimetic finite difference (MFD) method with $1,136,456$ elements for the equi-dimensional model. 

Unlike the scenarios involving a single barrier, noticeable discrepancies in the solutions, especially on the coarse and medium refinement grids, between the two methods are observed. 
This is possibly due to their difference in handling the intersection of low-permeable barriers. 
In our method, tangential flows on barriers are completely neglected, thus no special 
condition at intersections is imposed.
{\color{black}
The neglected intersection volume in the ebox-dfm model may lead to an overestimation of the transfer fluxes across the barrier intersections in comparison to the equi-dimensional reference, in particular on coarse meshes. See the subsequent example for more investigations on the influence of differences in intersection handling.
}
From the comparison of slices in Figure \ref{fig:regular_plot}, our results show better agreement with the reference solution.

\begin{figure}[!htbp]
 \centering
 \begin{subfigure}[b]{0.3\textwidth}
  \includegraphics[width=\textwidth]{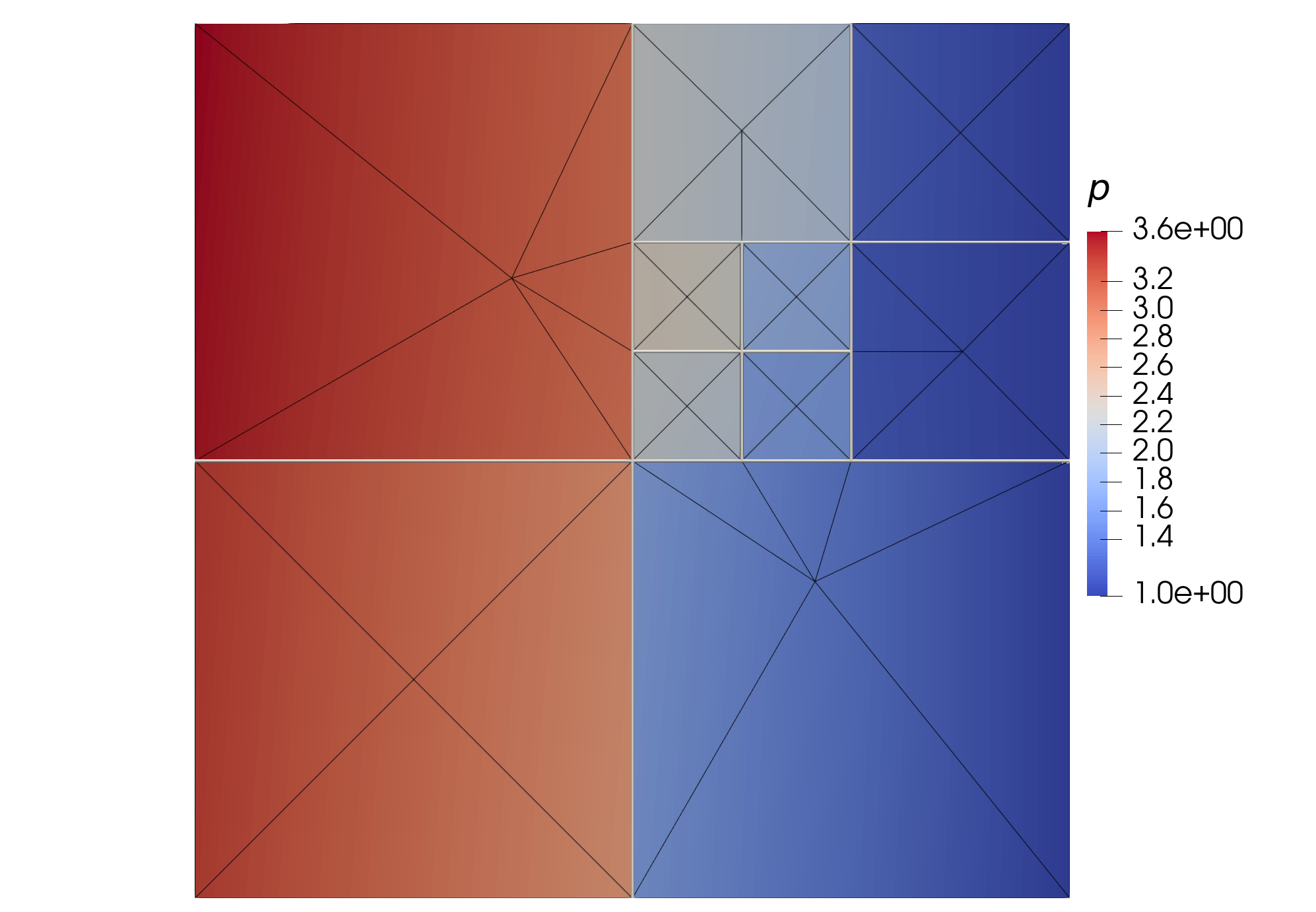}
  \caption{coarse grid}
 \end{subfigure}
 \begin{subfigure}[b]{0.3\textwidth}
  \includegraphics[width=\textwidth]{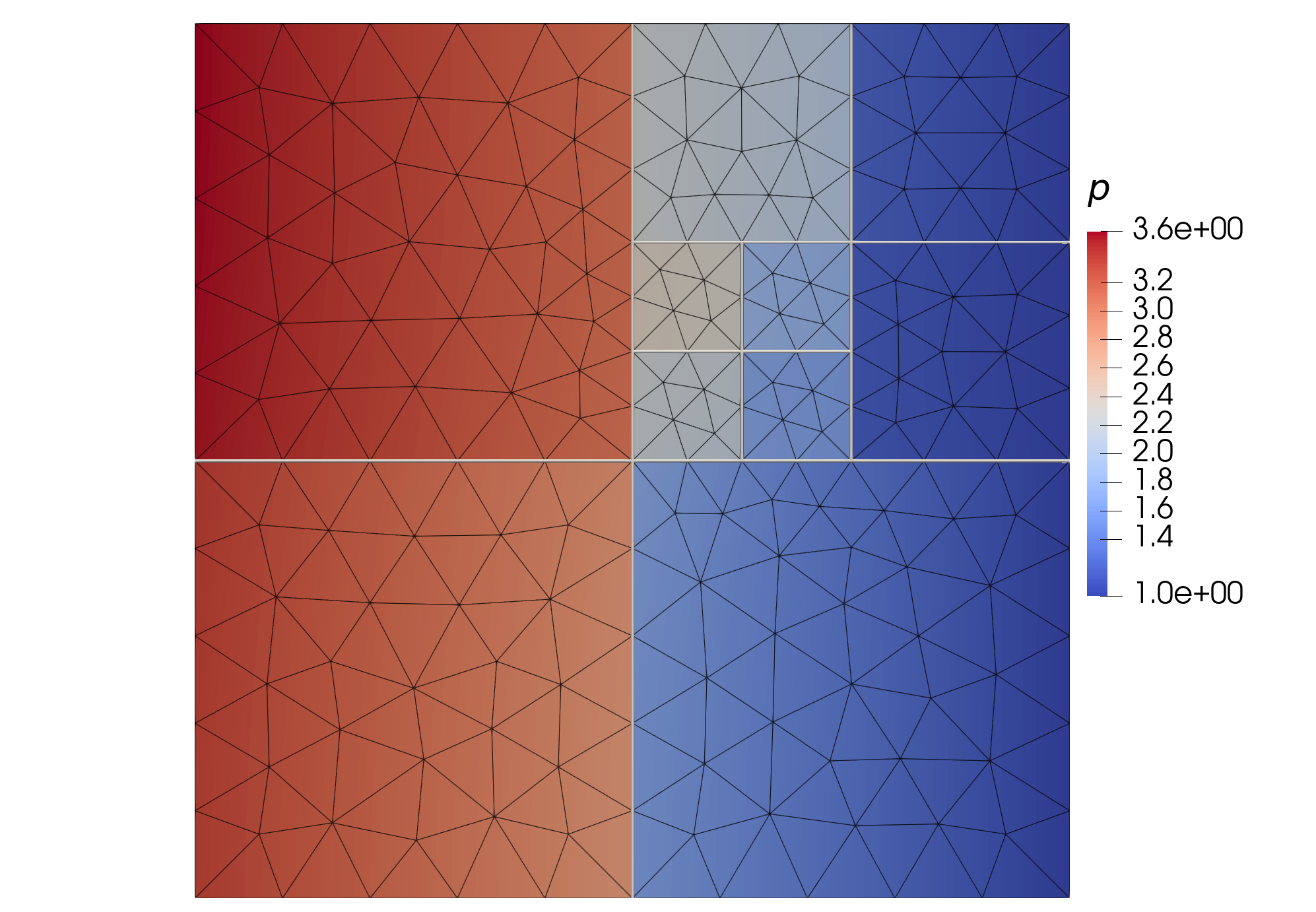}
  \caption{medium grid}
 \end{subfigure}
 \begin{subfigure}[b]{0.3\textwidth}
  \includegraphics[width=\textwidth]{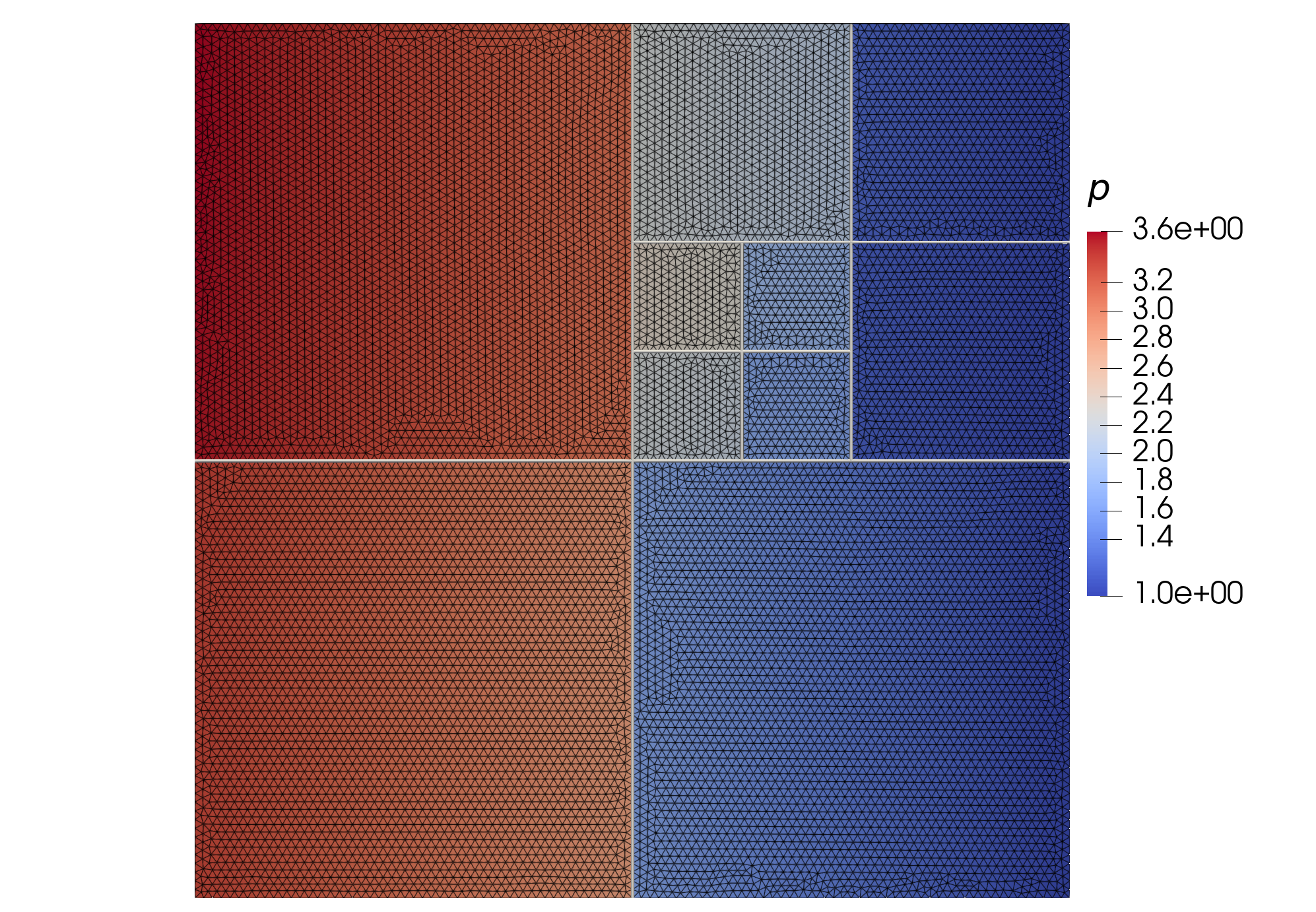}
  \caption{fine grid}
 \end{subfigure}

 \begin{subfigure}[b]{0.3\textwidth}
  \includegraphics[width=\textwidth]{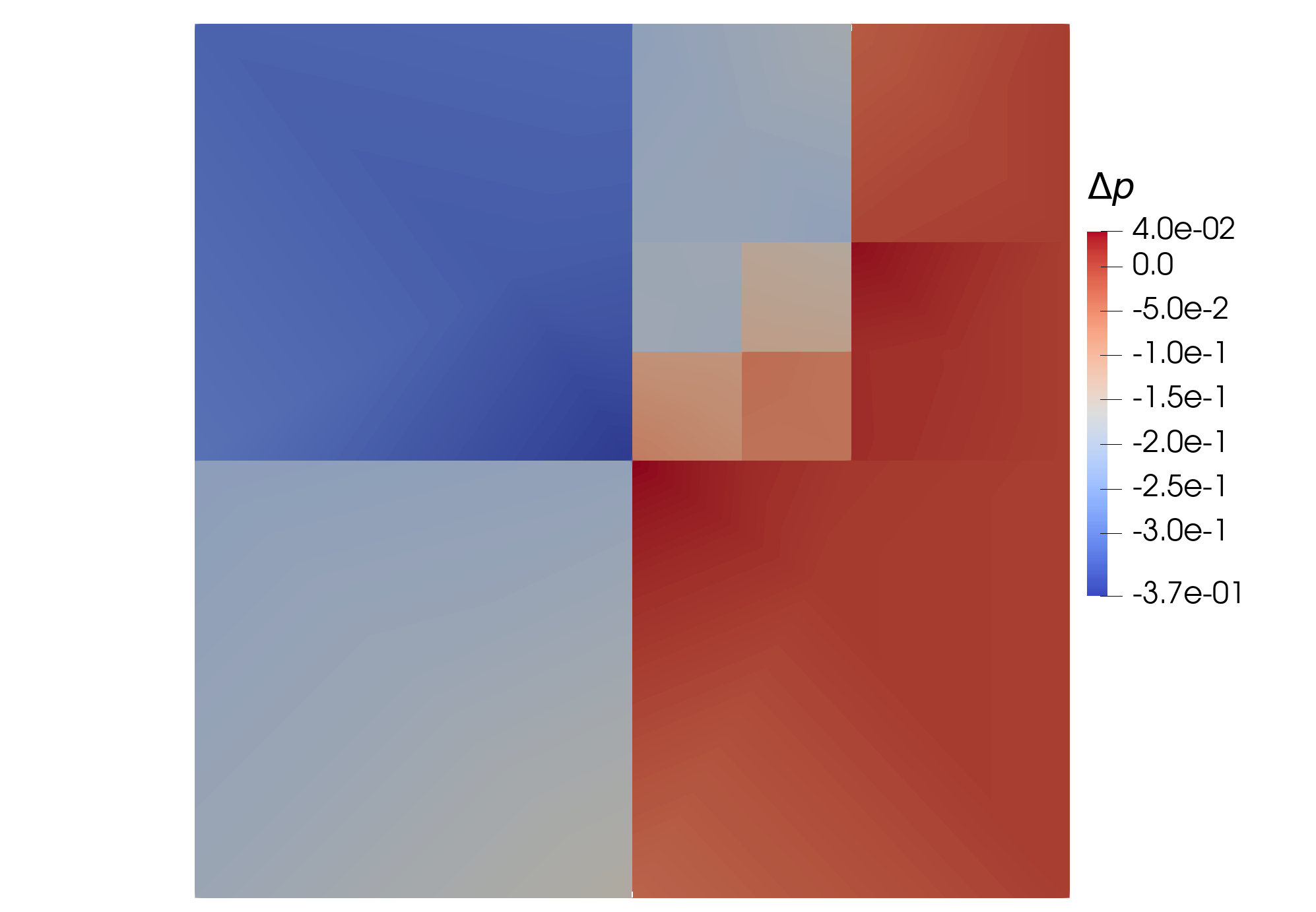}
  \caption{coarse grid}
 \end{subfigure}
 \begin{subfigure}[b]{0.3\textwidth}
  \includegraphics[width=\textwidth]{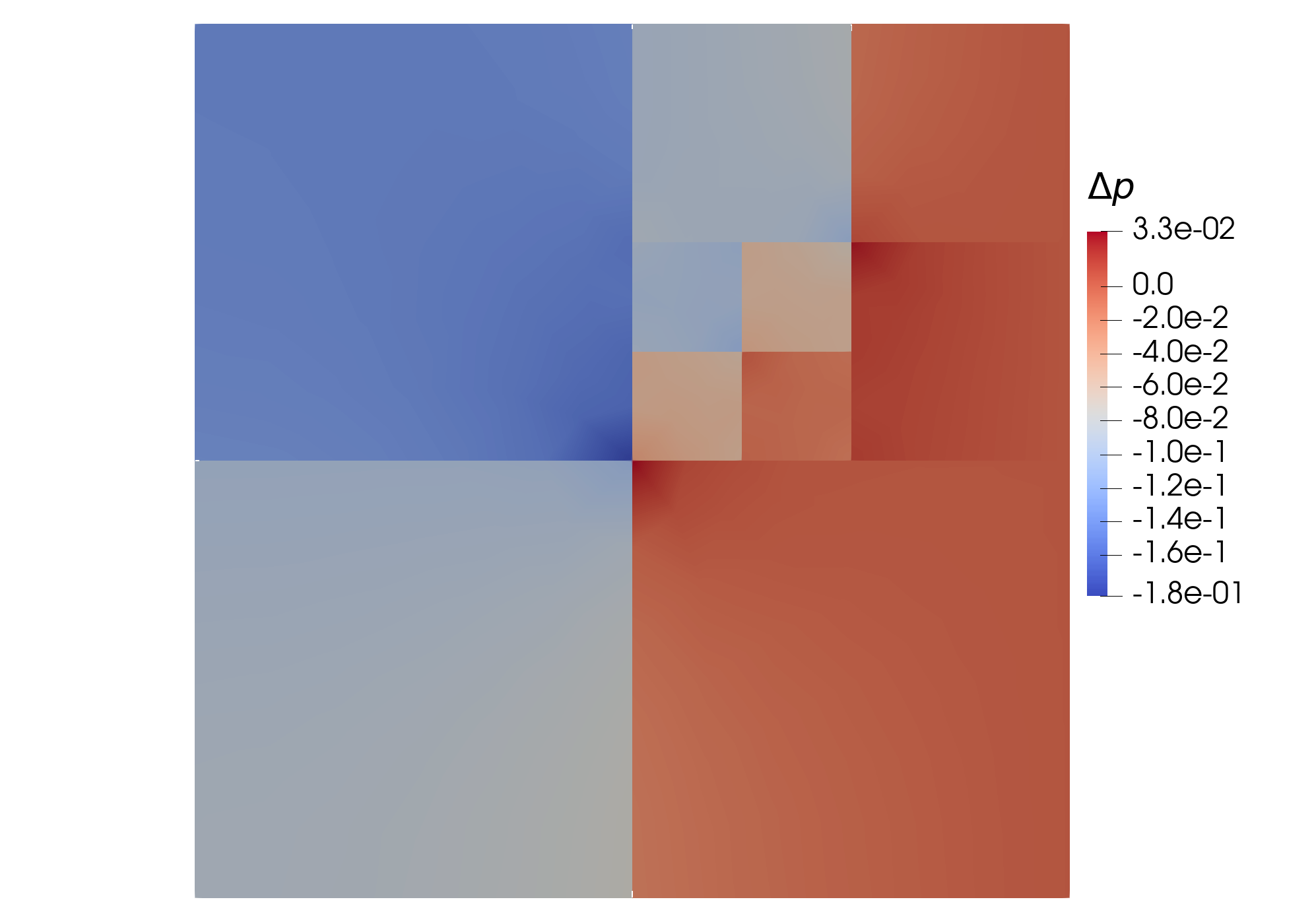}
  \caption{medium grid}
 \end{subfigure}
 \begin{subfigure}[b]{0.3\textwidth}
  \includegraphics[width=\textwidth]{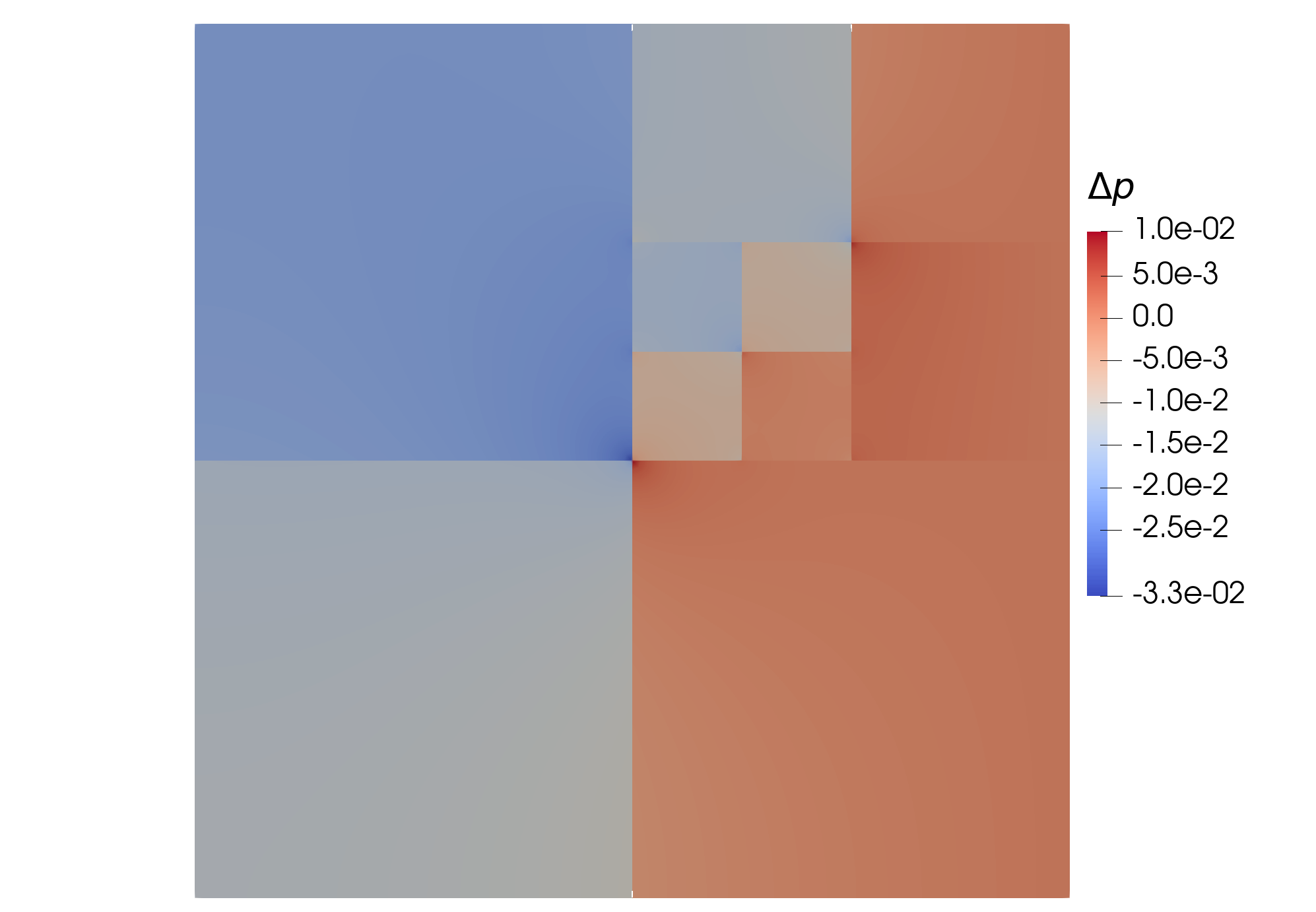}
  \caption{fine grid}
 \end{subfigure}
 \caption{\textbf{Example \ref{ex:regular}: regular barrier network.} 
 Visualization of the solution $p$ (row $1$) obtained with the proposed method, and the difference $\Delta p$ (row $2$) to the solution obtained with the ebox-dfm in \cite{glaser2022comparison}. 
 The coarse grid contains $29$ vertices and $46$ triangles; the medium grid contains $205$
 vertices and $366$ triangles; the fine grid contains $12, 112$ vertices and $23, 822$ triangles.}
 \label{fig:regular_contour}
\end{figure}

\begin{figure}[!htbp]
 \centering
 \begin{subfigure}[b]{0.3\textwidth}
  \includegraphics[width=\textwidth]{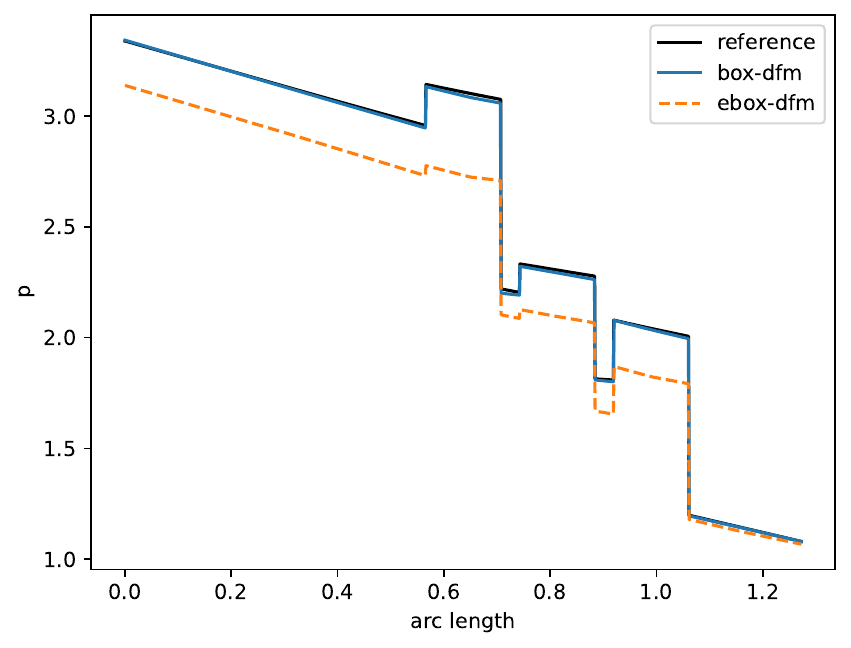}
  \caption{coarse grid}
 \end{subfigure}
 \begin{subfigure}[b]{0.3\textwidth}
  \includegraphics[width=\textwidth]{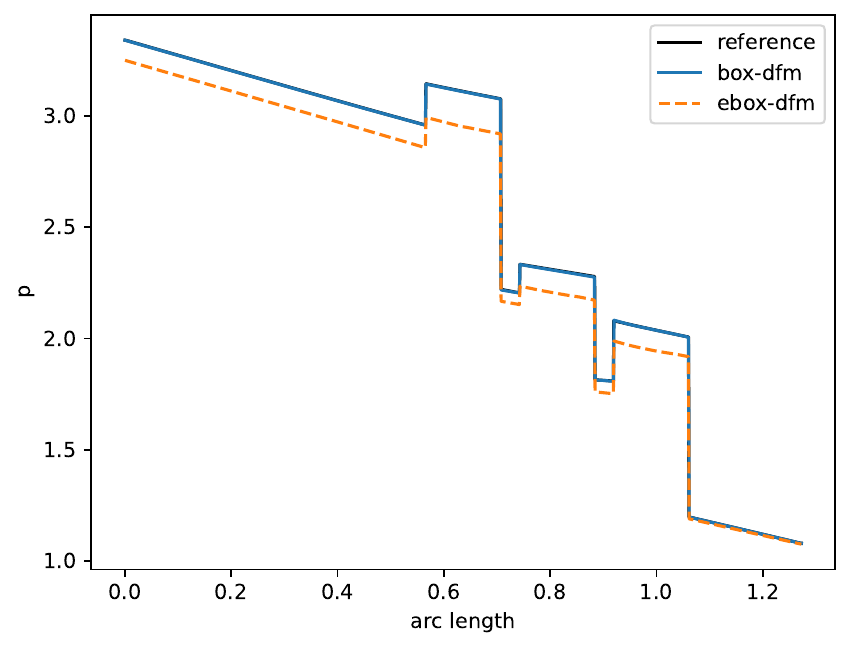}
  \caption{medium grid}
 \end{subfigure}
 \begin{subfigure}[b]{0.3\textwidth}
  \includegraphics[width=\textwidth]{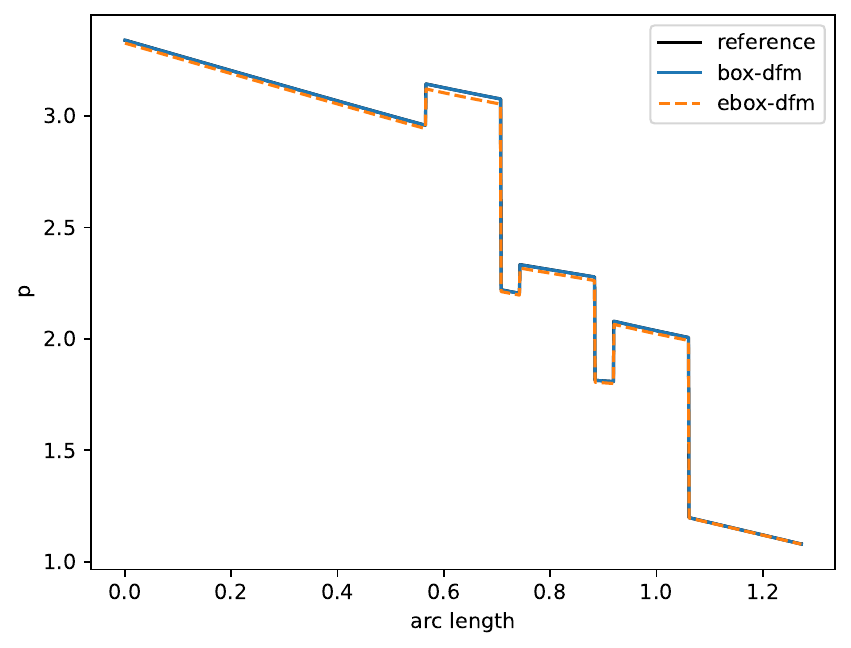}
  \caption{fine grid}
 \end{subfigure}
 \caption{\textbf{Example \ref{ex:regular}: regular barrier network.} 
 Pressure profiles along the slice (0, 0.1) -- (0.9, 1).
 The coarse grid contains $29$ vertices and $46$ triangles; the medium grid contains $205$ vertices and $366$ triangles; the fine grid contains $12, 112$ vertices and $23, 822$ triangles.
The reference, provided by the authors in \cite{flemisch2018benchmarks}, is obtained from the mimetic finite difference (MFD) method with $1,136,456$ elements for the equi-dimensional model.
}
 \label{fig:regular_plot}
\end{figure}

\end{exmp}

\begin{exmp}\label{ex:complex}
{Complex fracture-barrier network}

In this experiment, we test a complex fracture-barrier network as benchmark 4.3 in \cite{flemisch2018benchmarks}, where the high-permeable fractures and low-permeable barriers coexist and intersect.

The computational domain is $\Omega=[0,1]^2$, and the permeability tensor of the porous matrix is $\mathbf{K}_m=\mathbf{I}$. 
The domain contains 8 high-permeable fractures and 2 low-permeable barriers, with a uniform aperture $a=10^{-4}$.
All fractures have a uniform permeability $k_f=10^{4}$, and all barriers have a uniform permeability $k_b=10^{-4}$. 
For the exact coordinates of the fractures and barriers, we refer the readers to the Appendix C in \cite{flemisch2018benchmarks}.
Two sub-cases are considered.
In the sub-case (a), where the pressure gradient is predominantly vertical, the left and right boundaries are impermeable, i.e., $g_N=0$; and the top and bottom boundaries are Dirichlet with $g_D=4$ and $g_D=1$, respectively.
In the sub-case (b), where the pressure gradient is predominantly horizontal, the top and bottom boundaries are impermeable, i.e., $g_N=0$; and the left and right boundaries are Dirichlet with $g_D=4$ and $g_D=1$, respectively.

The treatment of intersections between high-permeability fractures and low-permeability barriers is a non-trivial issue. 
In \cite{flemisch2018benchmarks}, a harmonic average of the permeabilities of the fracture and barrier for the intersection cell is adopted in methods that explicitly account for the intersections.
{\color{black}In the ebox-dfm, pressure and flux continuity is enforced at intersections.}
In our proposed method, we prioritize the barrier at intersections, imposing discontinuous pressure at the intersection, {\color{black}but we will discuss an alternative approach in the sequel.}

We conduct the computation on a grid containing $1, 391$ vertices and $2, 680$ triangles using our extended box-dfm and the ebox-dfm for both cases. 
The pressure contours produced by our method, alongside the background mesh, and the differences in pressure between the extended box-dfm and the ebox-dfm, are presented in Figure \ref{fig:complex_contour}. 
Additionally, we present pressure slices along the line from (0, 0.5) to (1, 0.9) in Figure \ref{fig:complex_slice}, comparing them with the reference solution provided in \cite{flemisch2018benchmarks}.

{\color{black} In the reference solution, a harmonic average of the fracture and barrier permeabilities were assigned to the cells occupying the intersection region, thereby the intersection exhibits a more significant blocking feature.}
In reality, it is also possible for the intersection to show a more flow-preferable conducting feature. 
Therefore, we attempt an alternative approach to impose pressure continuity at the intersection and present the corresponding pressure slices in Figure \ref{fig:complex_slice_case_b}.
In these figures, we observe a larger deviation of our results from the reference solutions, but a smaller deviation from the results of the ebox-dfm is observed.
{\color{black}These findings are in line with the discussions in \cite{flemisch2018benchmarks} on the influence of not modelling the intersection region explicitly.}

\begin{figure}[!htbp]
 \centering
 \begin{subfigure}[b]{0.3\textwidth}
  \includegraphics[width=\textwidth]{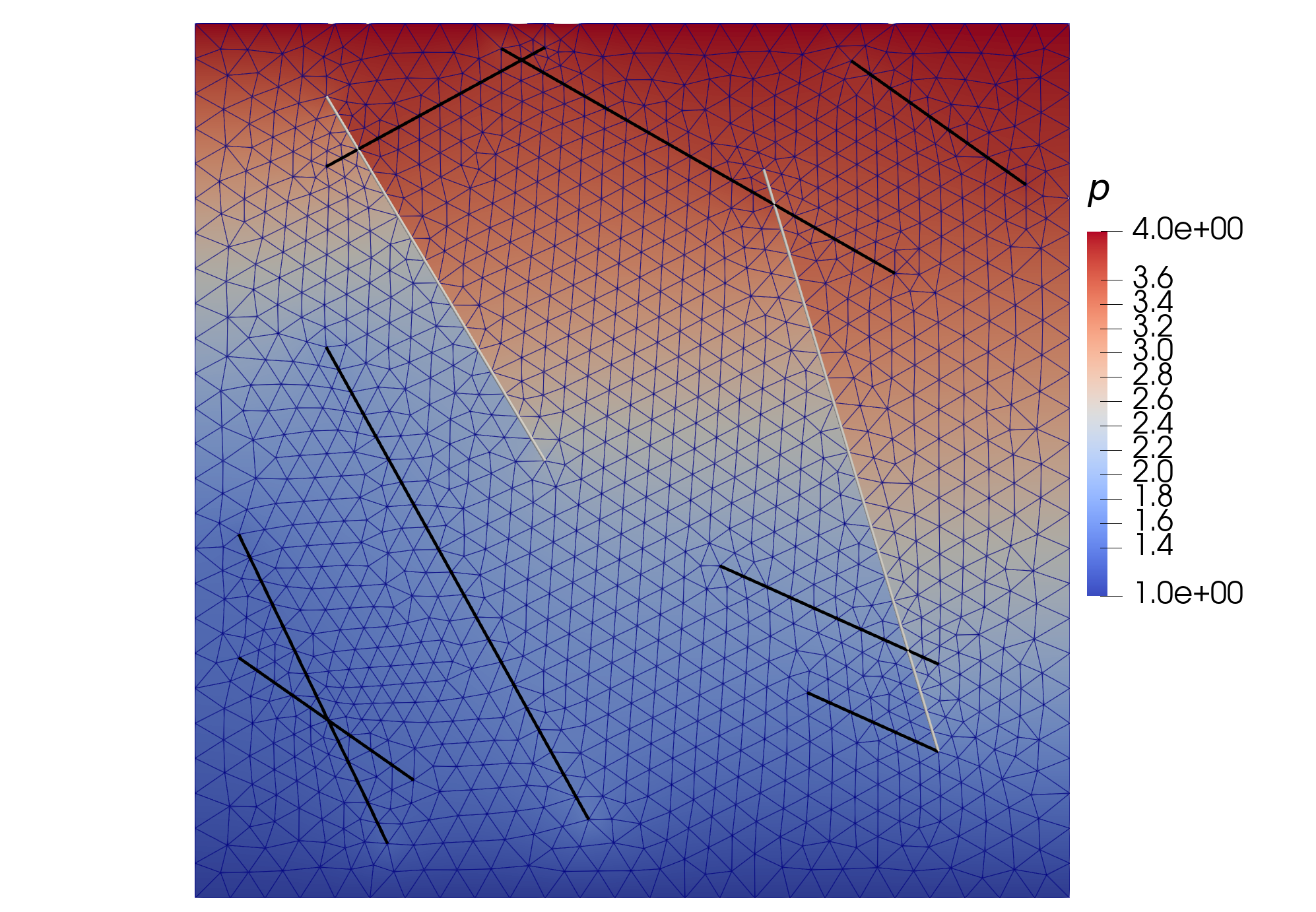}
  \caption{Flow from top to bottom}
 \end{subfigure}
 \begin{subfigure}[b]{0.3\textwidth}
  \includegraphics[width=\textwidth]{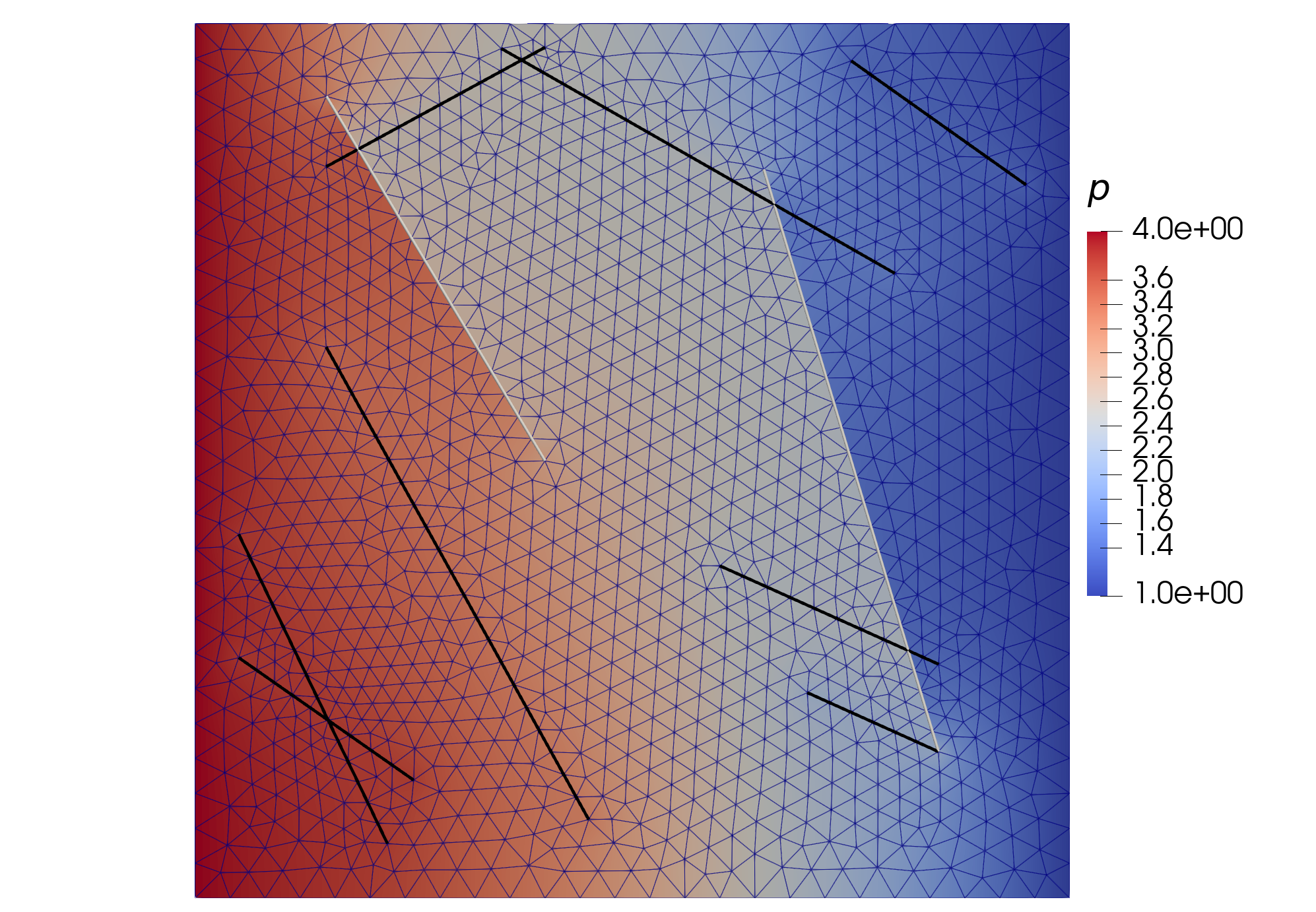}
  \caption{Flow from left to right}
 \end{subfigure}

 \begin{subfigure}[b]{0.3\textwidth}
  \includegraphics[width=\textwidth]{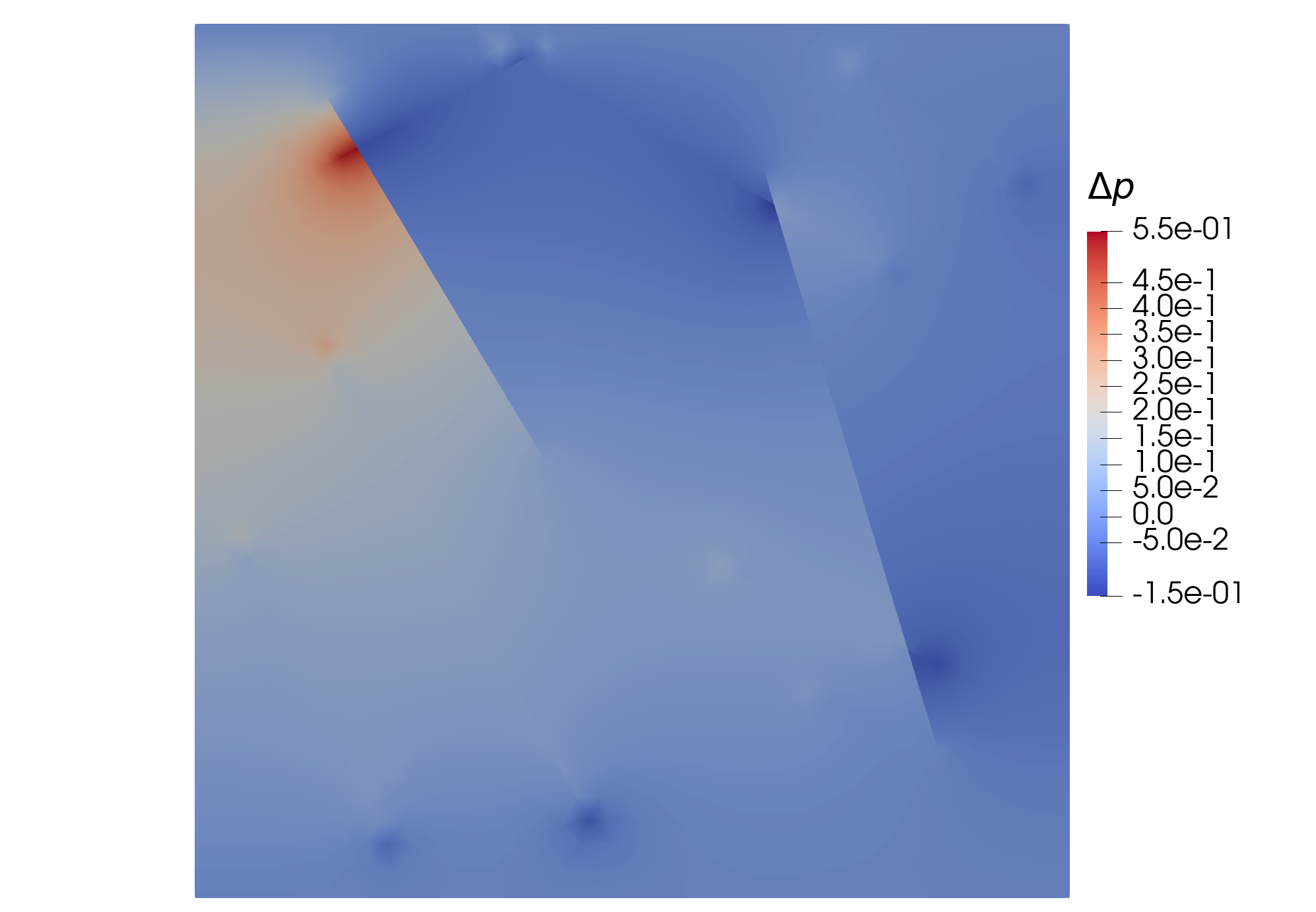}
  \caption{Flow from top to bottom}
 \end{subfigure}
 \begin{subfigure}[b]{0.3\textwidth}
  \includegraphics[width=\textwidth]{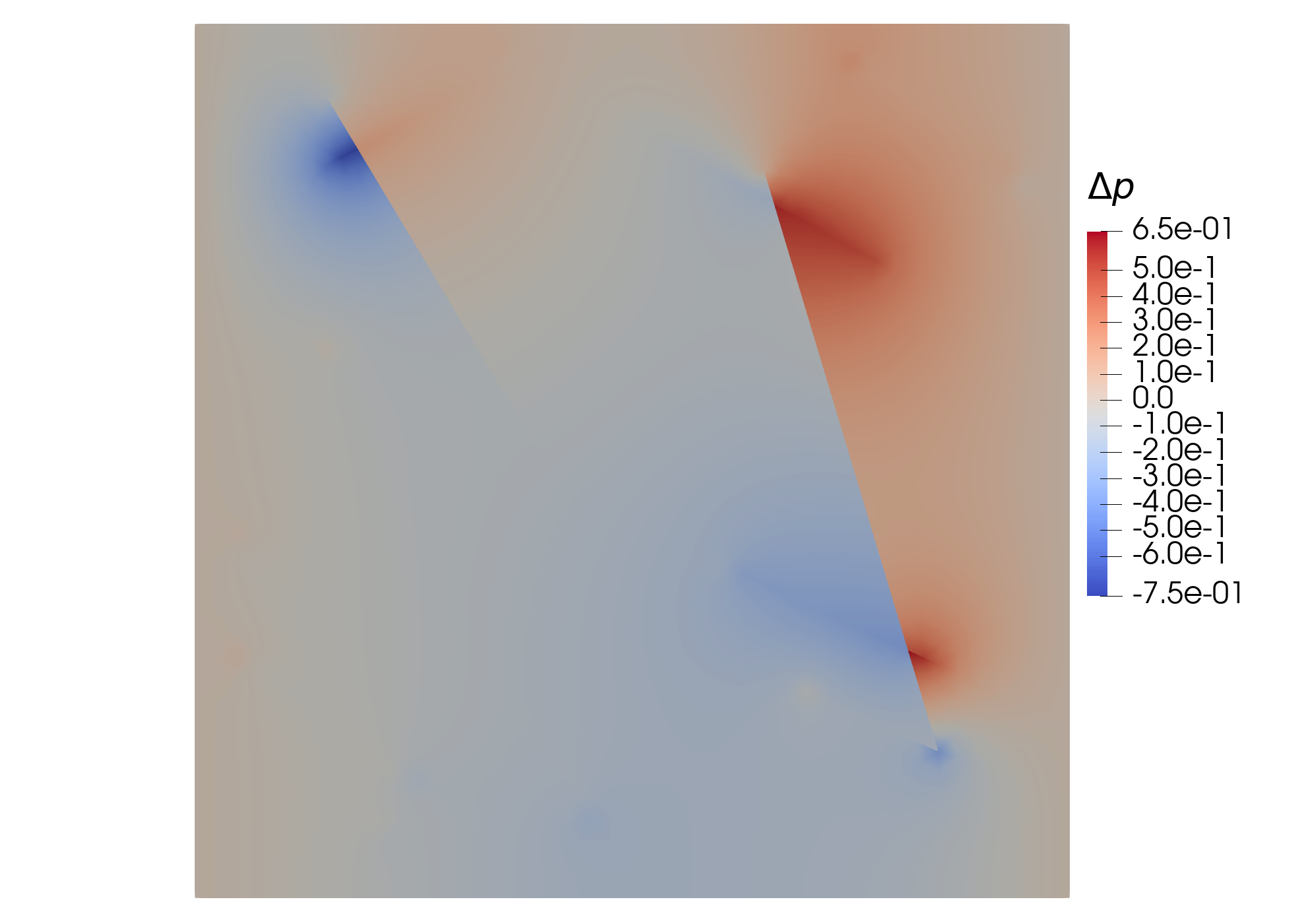}
  \caption{Flow from left to right}
 \end{subfigure}
 \caption{\textbf{Example \ref{ex:complex}: complex fracture-barrier network.}  Visualization of the solution $p$ (row $1$) obtained with the proposed method, and the difference $\Delta p$ (row $2$) to the solution obtained with the ebox-dfm in \cite{glaser2022comparison}. 
 The first scenario is the flow from top to bottom, and the second scenario is the flow from left to right. Both cases contain $1, 391$ vertices and $2, 680$ triangles.
 Pressure discontinuity is assumed at the intersections of fracture and barrier in the box-dfm.}
 \label{fig:complex_contour}
\end{figure}

\begin{figure}[!htbp]
 \centering
 \begin{subfigure}[b]{0.3\textwidth}
  \includegraphics[width=\textwidth]{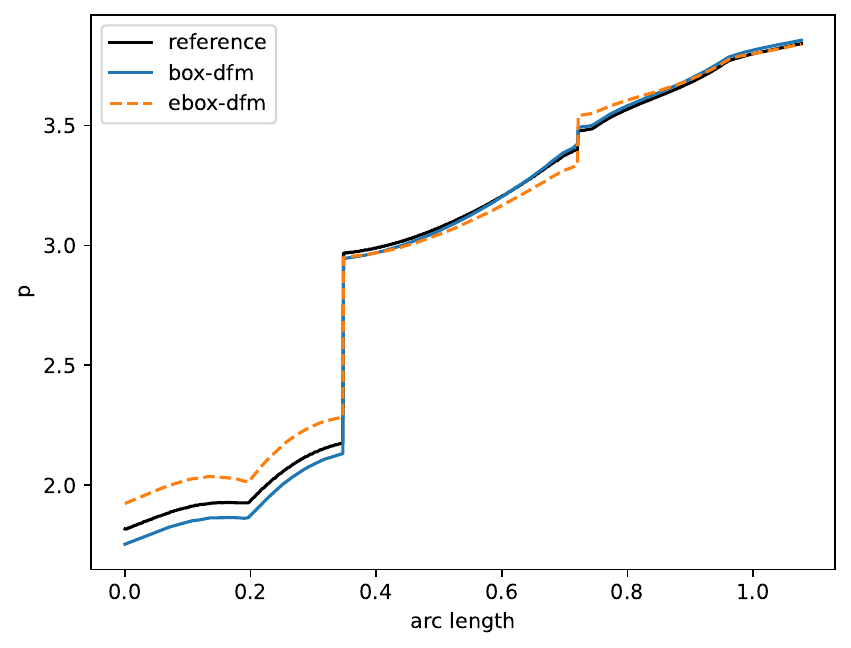}
  \caption{Flow from top to bottom}
 \end{subfigure}
 \begin{subfigure}[b]{0.3\textwidth}
  \includegraphics[width=\textwidth]{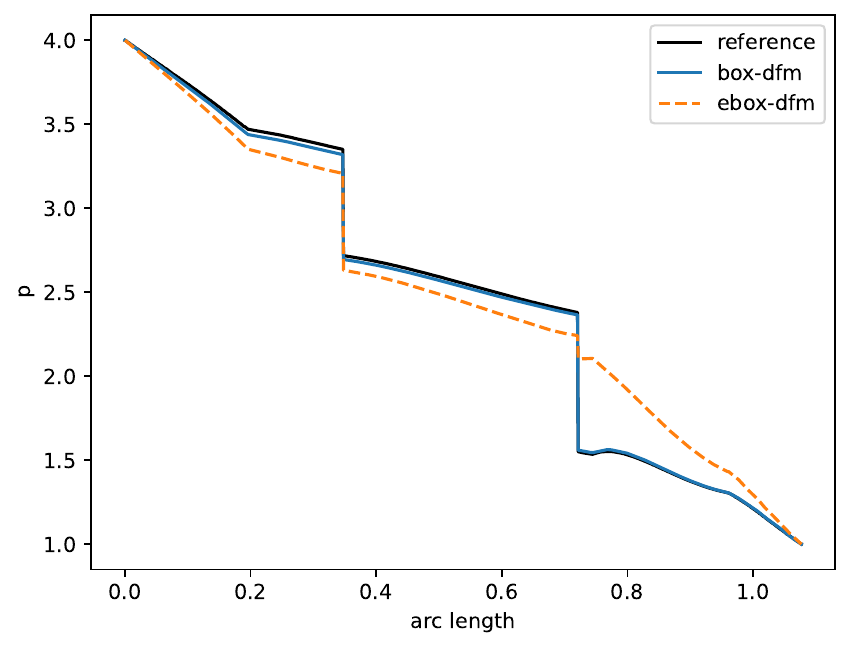}
  \caption{Flow from left to right}
 \end{subfigure}
 \caption{\textbf{Example \ref{ex:complex}: complex fracture-barrier network.} 
 Pressure profiles along the slice (0, 0.5) -- (1, 0.9).
 The reference, provided by the authors in \cite{flemisch2018benchmarks}, is obtained from the mimetic finite difference (MFD) method with $2,260,352$ elements for the equi-dimensional model.
 Pressure discontinuity is assumed at the intersections of fracture and barrier in the box-dfm.
 }
 \label{fig:complex_slice}
\end{figure}

\begin{figure}[!htbp]
 \centering
 \begin{subfigure}[b]{0.3\textwidth}
  \includegraphics[width=\textwidth]{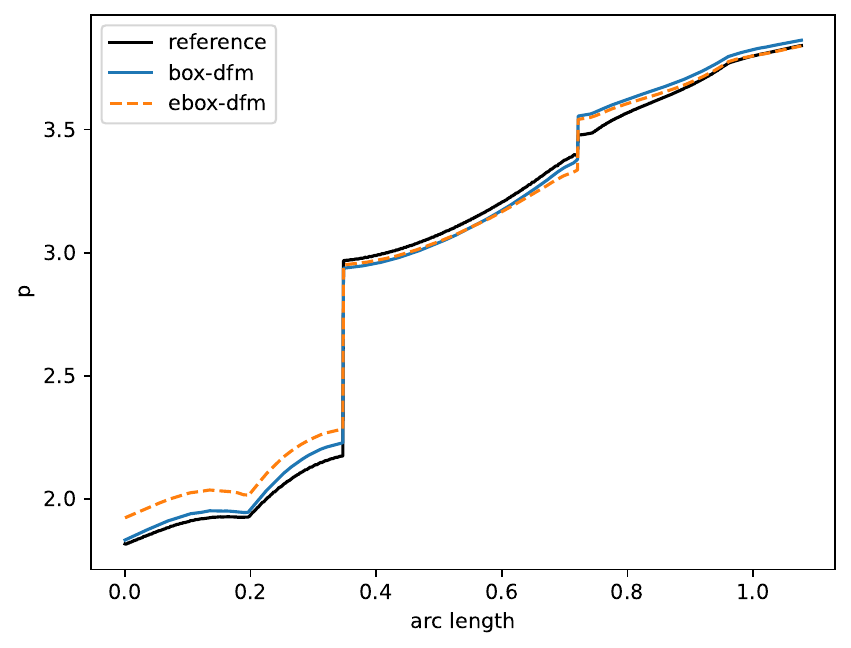}
  \caption{Flow from top to bottom}
 \end{subfigure}
 \begin{subfigure}[b]{0.3\textwidth}
  \includegraphics[width=\textwidth]{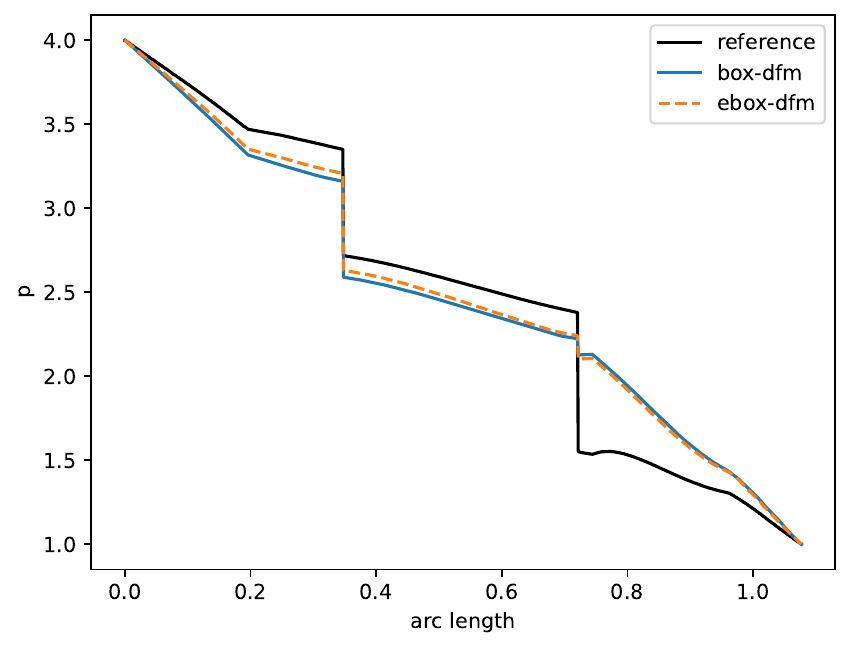}
  \caption{Flow from left to right}
 \end{subfigure}
 \caption{\textbf{Example \ref{ex:complex}: complex fracture-barrier network.} 
 Pressure profiles along the slice (0, 0.5) -- (1, 0.9).
 Pressure continuity is assumed at the intersections of fracture and barrier in the box-dfm.
 }
 \label{fig:complex_slice_case_b}
\end{figure}

\end{exmp}

\begin{exmp}\label{ex:real}
{Realistic barrier network}

In this example, we test the case of a realistic barrier network containing 64 low-permeable barriers whose distribution is based on an image of an outcrop near Bergen, Norway.
This test case is taken from \cite{glaser2022comparison}, where the authors transformed high-permeable fractures from the original test case presented in \cite{flemisch2018benchmarks} into a variant featuring low-permeable barriers to validate various discrete fracture models.

The computational domain is $\Omega=[0, 700\si{m}]\times [0, 600\si{m}]$, and a flow from left to right is driven by the Dirichlet boundary condition $g_D=1,013,250\si{m}$ at $x=0\si{m}$ and $g_D=0\si{m}$ at $x=700\si{m}$, along with impermeable boundary conditions at $y=0\si{m}$ and $y=600\si{m}$. 
All barriers in the domain have a uniform aperture of $a=10^{-2}\si{m}$ and permeability $k_b=10^{-18}\si{m/s}$. 
The permeability of the porous matrix is set to $k_m=10^{-14}\si{m/s}$. 
For detailed coordinates of the barriers, we refer the readers to \cite{flemisch2018benchmarks}.

We conduct the computation on three grids with different levels of refinement and compare the results of our extended box-dfm with those obtained from the ebox-dfm. The results of tpfa-dfm, mpfa-dfm, and ebox-dfm-mortar methods for the same test case are available in \cite{glaser2022comparison} for a more comprehensive comparison.
From the pressure contours and slices shown in Figure \ref{fig:real_contour} and Figure \ref{fig:real_plot}, respectively, one can observe the convergence between the results of the box-dfm and ebox-dfm with mesh refinement. 
The pressure profiles along (625, 0) -- (625, 600) show that our method's results on a coarse mesh exhibit better agreement with the reference solution, {\color{black}which is obtained using the ebox-dfm on a fine mesh consisting of $307, 829$ cells, while the finest mesh of the refinement series consists of $114, 721$ cells.}.

\begin{figure}[!htbp]
 \centering
 \begin{subfigure}[b]{0.3\textwidth}
  \includegraphics[width=\textwidth]{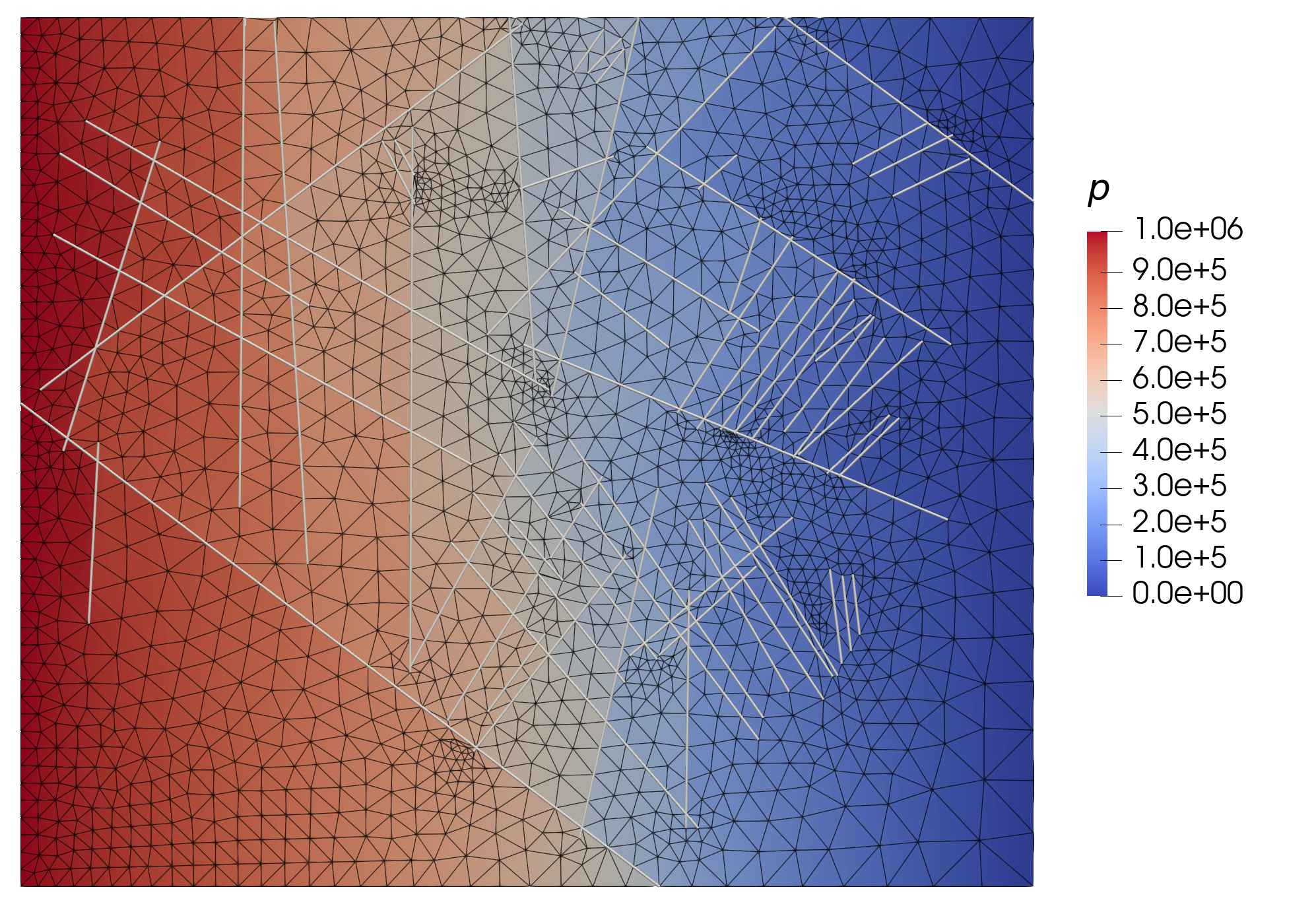}
  \caption{coarse grid}
 \end{subfigure}
 \begin{subfigure}[b]{0.3\textwidth}
  \includegraphics[width=\textwidth]{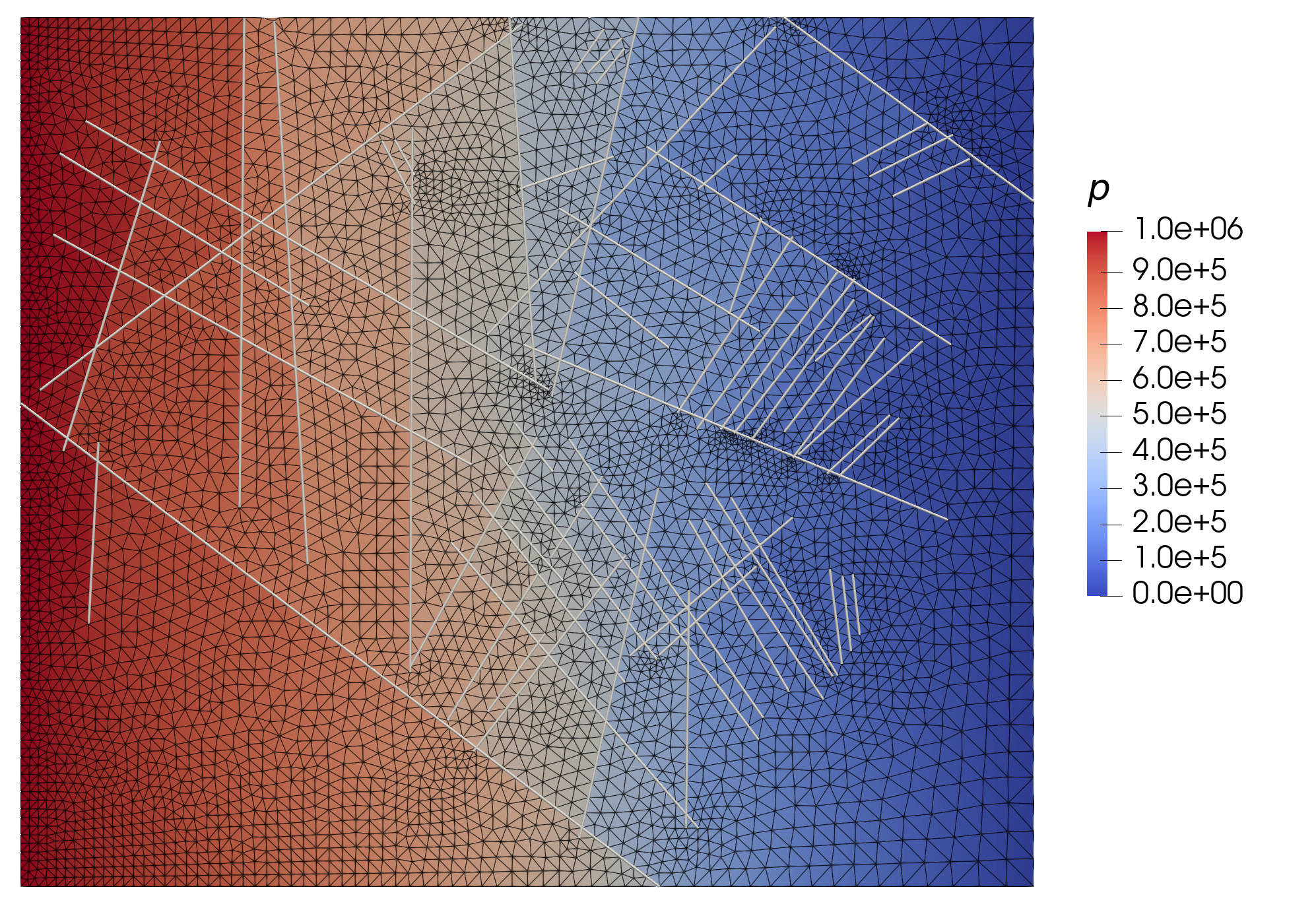}
  \caption{medium grid}
 \end{subfigure}
 \begin{subfigure}[b]{0.3\textwidth}
  \includegraphics[width=\textwidth]{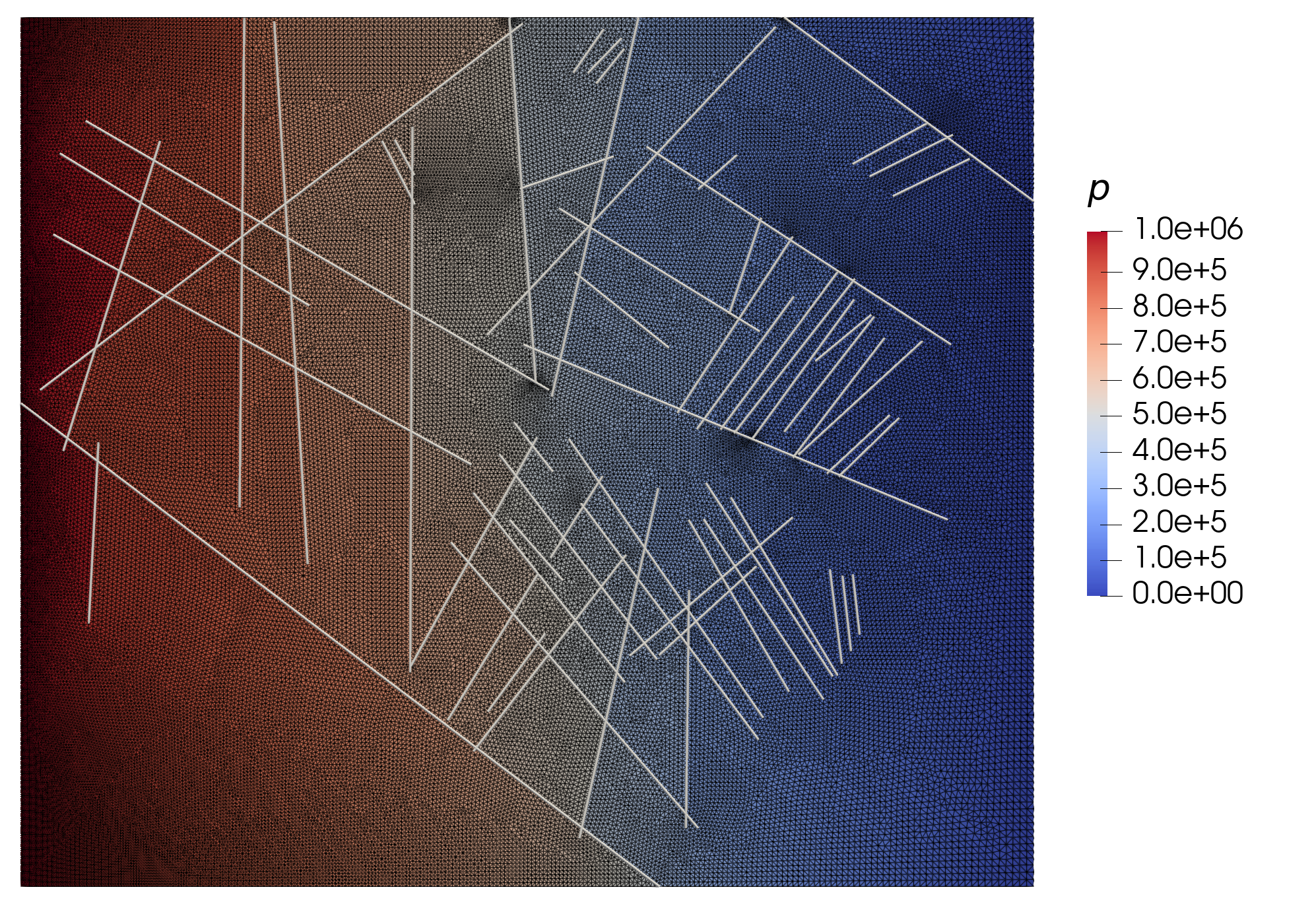}
  \caption{fine grid}
 \end{subfigure}

 \begin{subfigure}[b]{0.3\textwidth}
  \includegraphics[width=\textwidth]{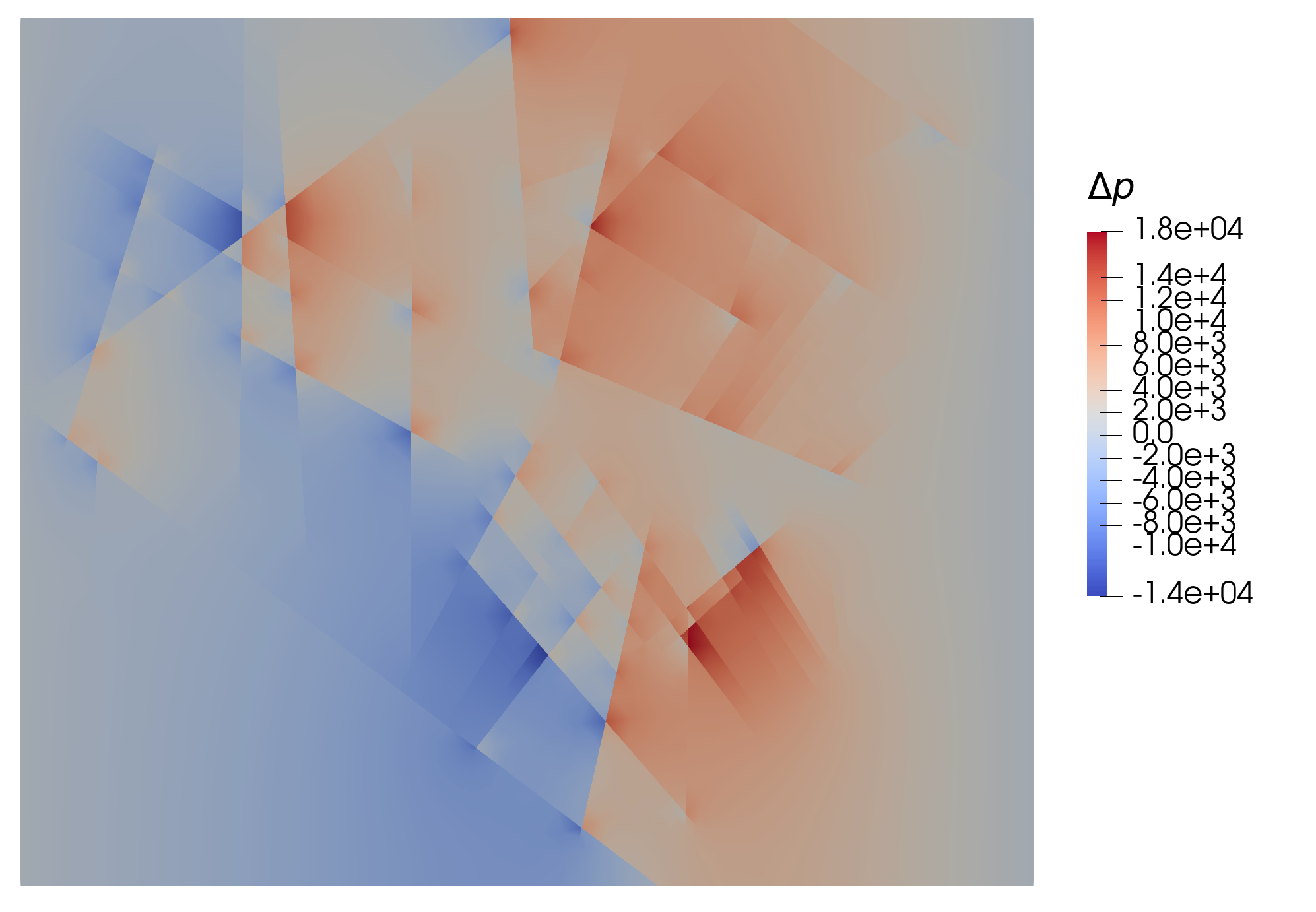}
  \caption{coarse grid}
 \end{subfigure}
 \begin{subfigure}[b]{0.3\textwidth}
  \includegraphics[width=\textwidth]{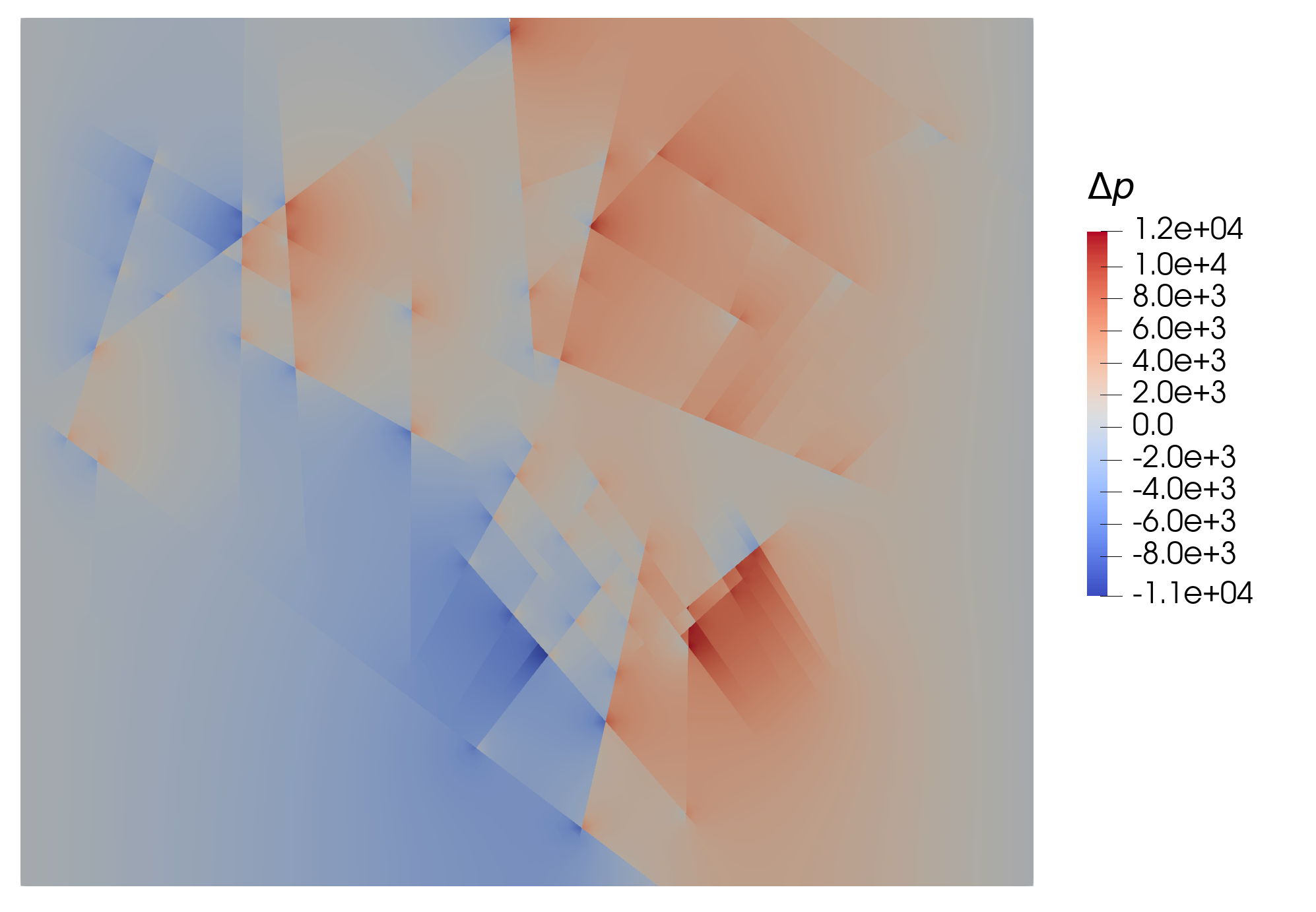}
  \caption{medium grid}
 \end{subfigure}
 \begin{subfigure}[b]{0.3\textwidth}
  \includegraphics[width=\textwidth]{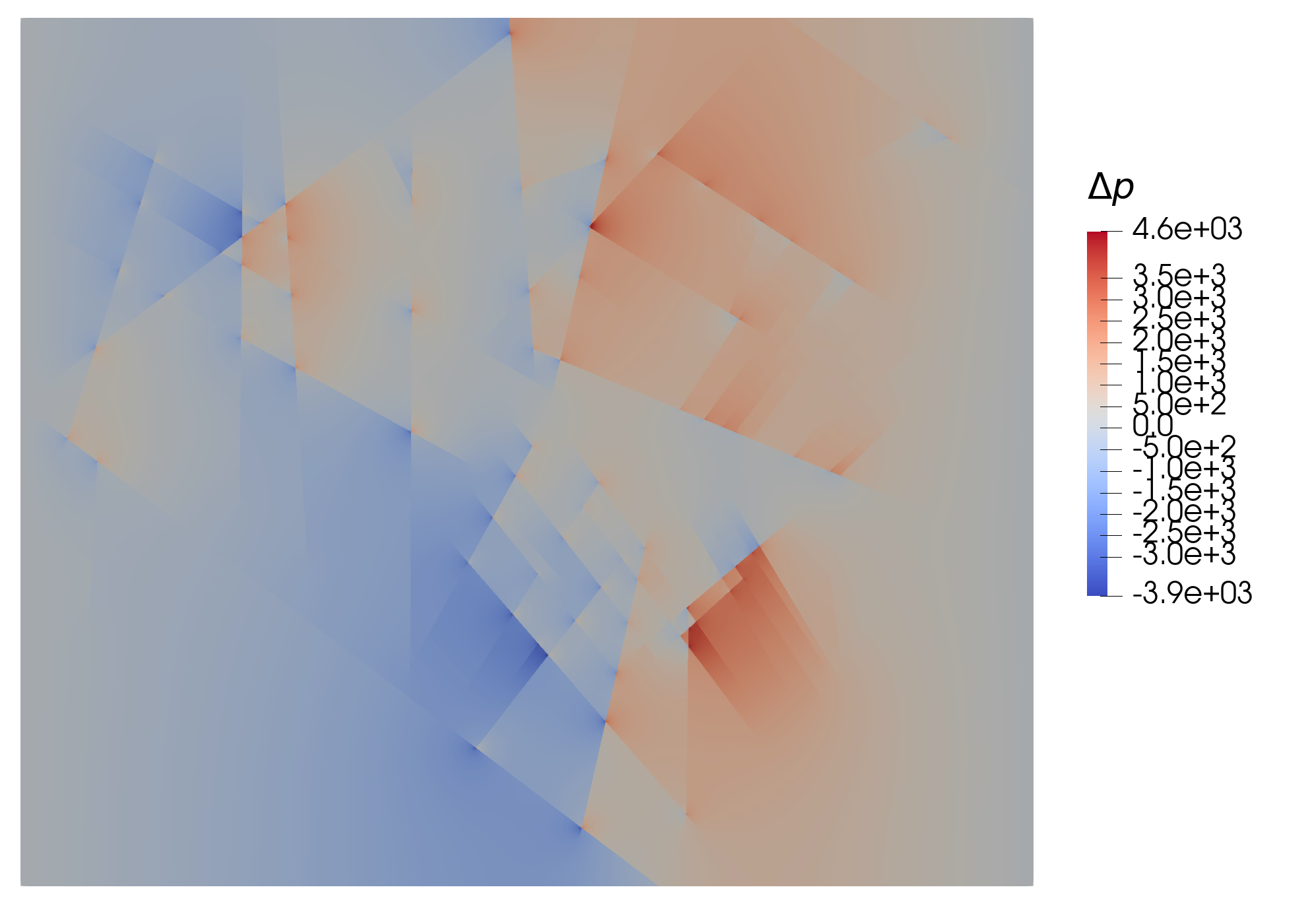}
  \caption{fine grid}
 \end{subfigure}
 \caption{\textbf{Example \ref{ex:real}: realistic barrier network.} 
 Visualization of the solution $p$ (row $1$) obtained from the proposed method, and the difference $\Delta p$ (row $2$) to the solution obtained with the ebox-dfm in \cite{glaser2022comparison}. The coarse grid contains $1, 884$ vertices and $3, 611$ triangles; the medium grid contains $4, 968$ vertices and $9, 646$ triangles; the fine grid contains $57, 904$ vertices and $114, 721$ triangles.}
 \label{fig:real_contour}
\end{figure}

\begin{figure}[!htbp]
 \centering
 \begin{subfigure}[b]{0.3\textwidth}
  \includegraphics[width=\textwidth]{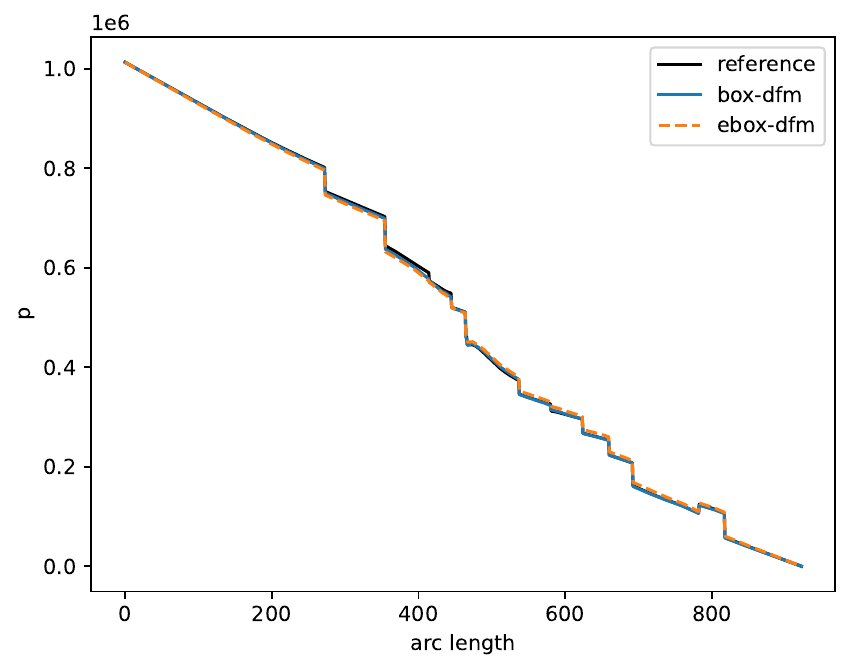}
  \caption{coarse grid}
 \end{subfigure}
 \begin{subfigure}[b]{0.3\textwidth}
  \includegraphics[width=\textwidth]{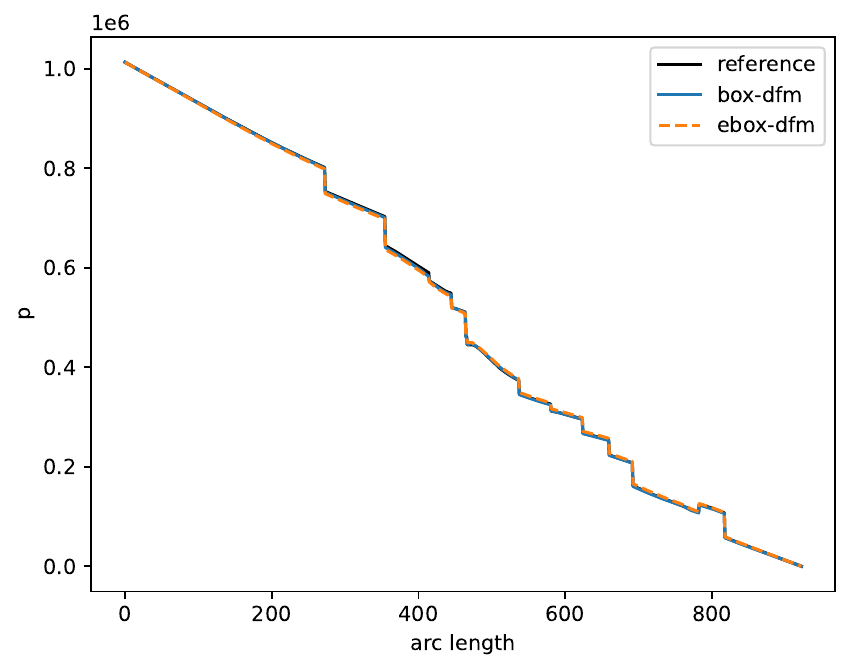}
  \caption{medium grid}
 \end{subfigure}
 \begin{subfigure}[b]{0.3\textwidth}
  \includegraphics[width=\textwidth]{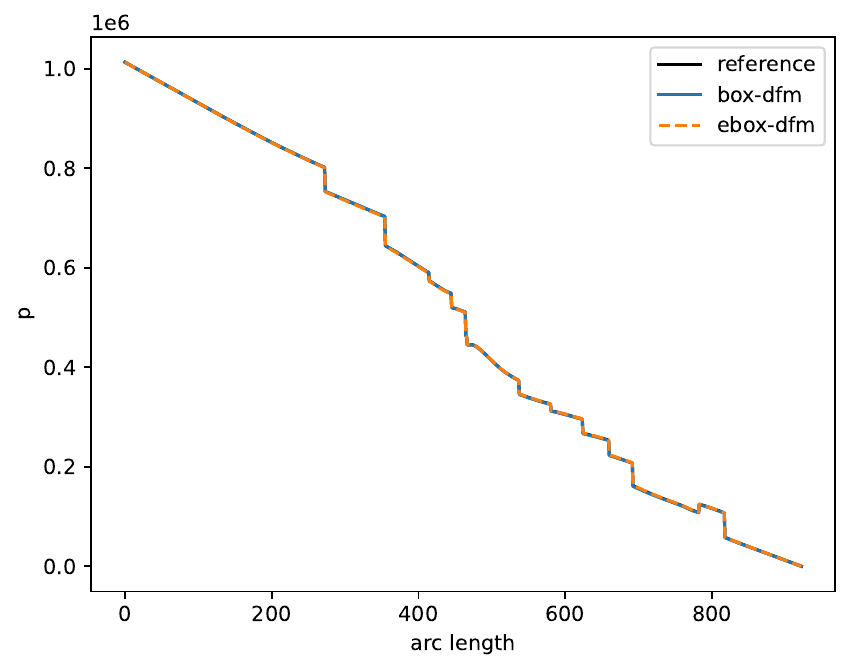}
  \caption{fine grid}
 \end{subfigure}
 
 \begin{subfigure}[b]{0.3\textwidth}
  \includegraphics[width=\textwidth]{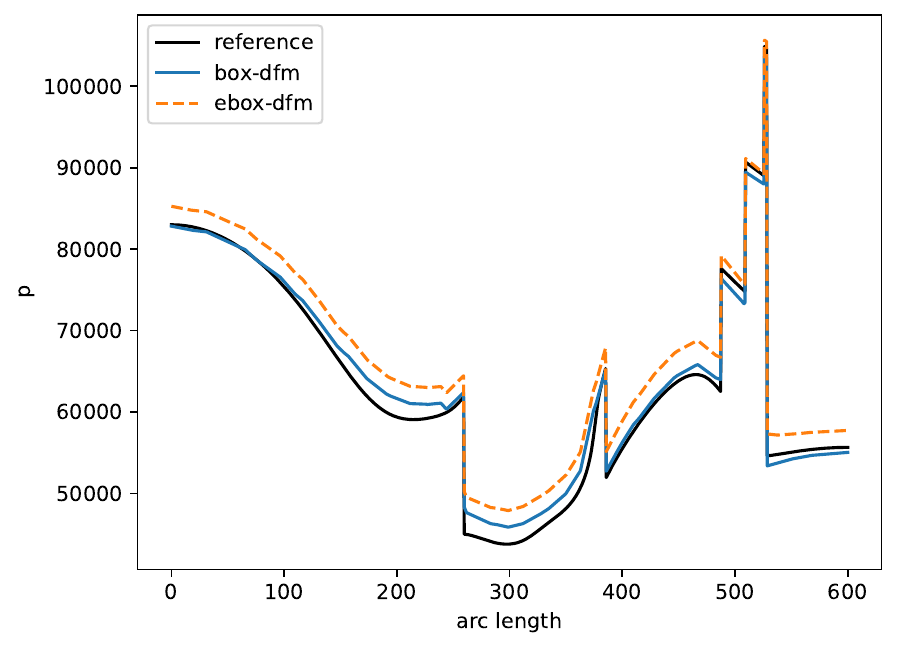}
  \caption{coarse grid}
 \end{subfigure}
 \begin{subfigure}[b]{0.3\textwidth}
  \includegraphics[width=\textwidth]{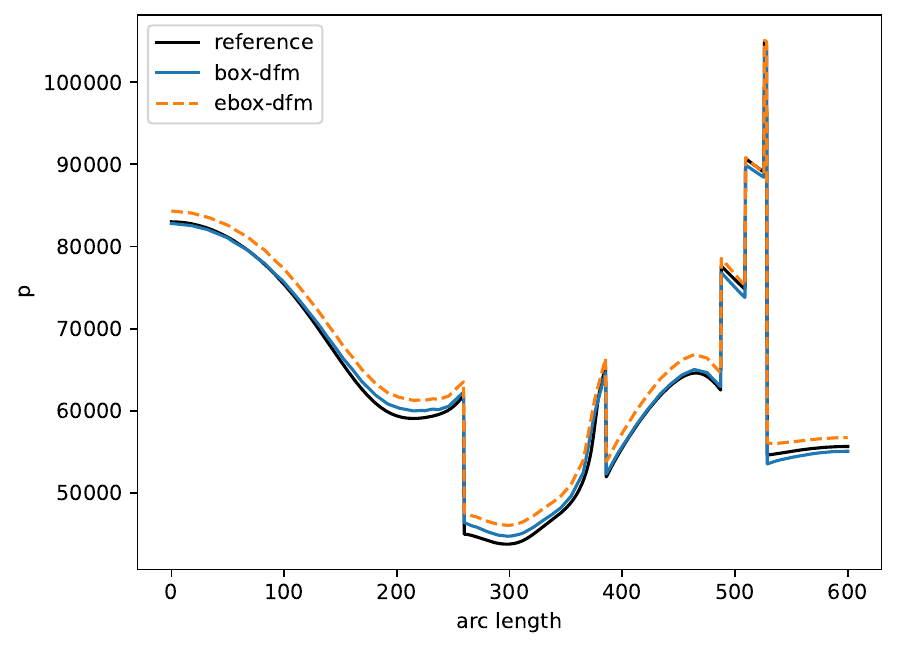}
  \caption{medium grid}
 \end{subfigure}
 \begin{subfigure}[b]{0.3\textwidth}
  \includegraphics[width=\textwidth]{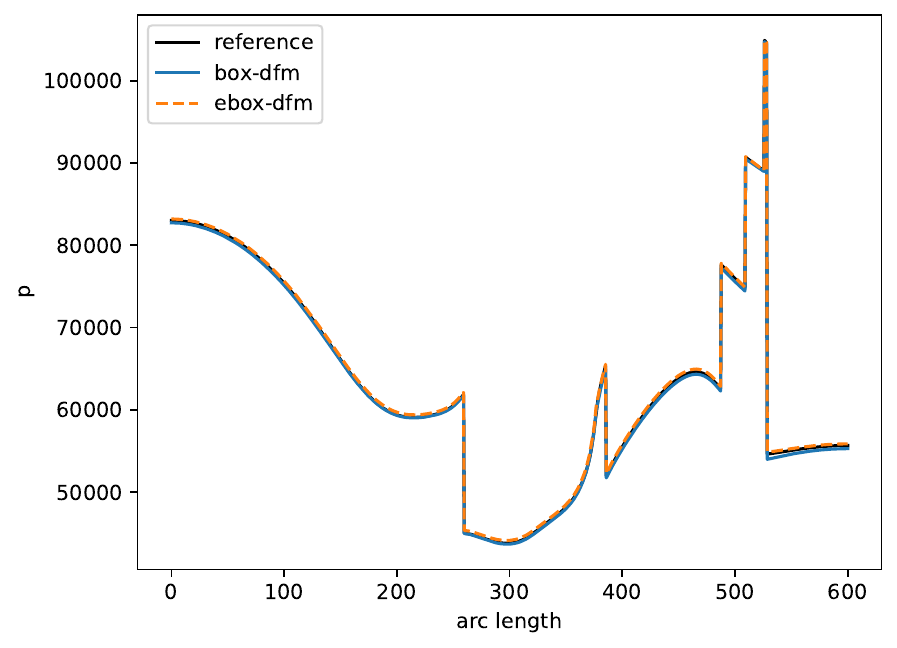}
  \caption{fine grid}
 \end{subfigure}
 \caption{\textbf{Example \ref{ex:real}: realistic barrier network.} Pressure profiles along the slice (0,0) -- (700, 600) (row $1$) and along the slice (625, 0) -- (625, 600) (row $2$). 
 The coarse grid contains $1, 884$ vertices and $3, 611$ triangles; the medium grid contains $4, 968$ vertices and $9, 646$ triangles; the fine grid contains $57, 904$ vertices and $114, 721$ triangles.
 The reference is obtained from the ebox-dfm with $154, 814$ vertices and $307, 829$ triangles.
 }
 \label{fig:real_plot}
\end{figure}

\end{exmp}

\begin{exmp}\label{ex:regular3D}
{Three-dimensional test case}

In this example, we test the flow in the unit cube $\Omega=[0, 1]^3$, containing $9$ regularly oriented barriers. 
One can find the exact coordinates of the barriers in case $2$ presented in \cite{berre2021verification}. 
The boundary $\partial\Omega$ is divided into three parts. 
The first part, $\partial\Omega_{D}=\{(x,y,z)\in\partial\Omega: x,y,z>0.875\}$, is imposed with the Dirichlet condition $g_D=1$; the second part, $\partial\Omega_{N}=\{(x,y,z)\in\partial\Omega: x,y,z<0.25\}$, is imposed with an inflow boundary condition $g_N=1$; the remaining boundaries are impermeable.
All barriers have a uniform aperture $a=10^{-4}$ and permeability $k_b=10^{-4}$.
The permeability of the porous matrix is different in two bulk regions.
In $\Omega_{1}=\{(x,y,z)\in\Omega: x>0.5, y<0.5\}\cup\{(x,y,z)\in\Omega: x>0.75, 0.5<y<0.75, z>0.5\}\cup\{(x,y,z)\in\Omega: 0.625<x<0.75, 0.5<y<0.625, 0.5<z<0.75\}$, we take $k_{m1}=0.1$; in $\Omega_2=\Omega\setminus\Omega_{1}$, we take $k_{m2}=1$.

We present the pressure contours and slices in Figure \ref{fig:regular3D_contour} and Figure \ref{fig:regular3D_plot}, respectively. 
As before, the convergence of box-dfm and ebox-dfm can be observed with grid refinement.
It can be seen from the slices that the ebox-dfm is closer to the reference solution on the coarse grid, while our proposed method is closer to the reference on the fine grid, {\color{black} where the reference solution is obtained from the data of the original benchmark study in \cite{berre2021verification}}.

\begin{figure}[!htbp]
 \centering
 \begin{subfigure}[b]{0.3\textwidth}
  \includegraphics[width=\textwidth]{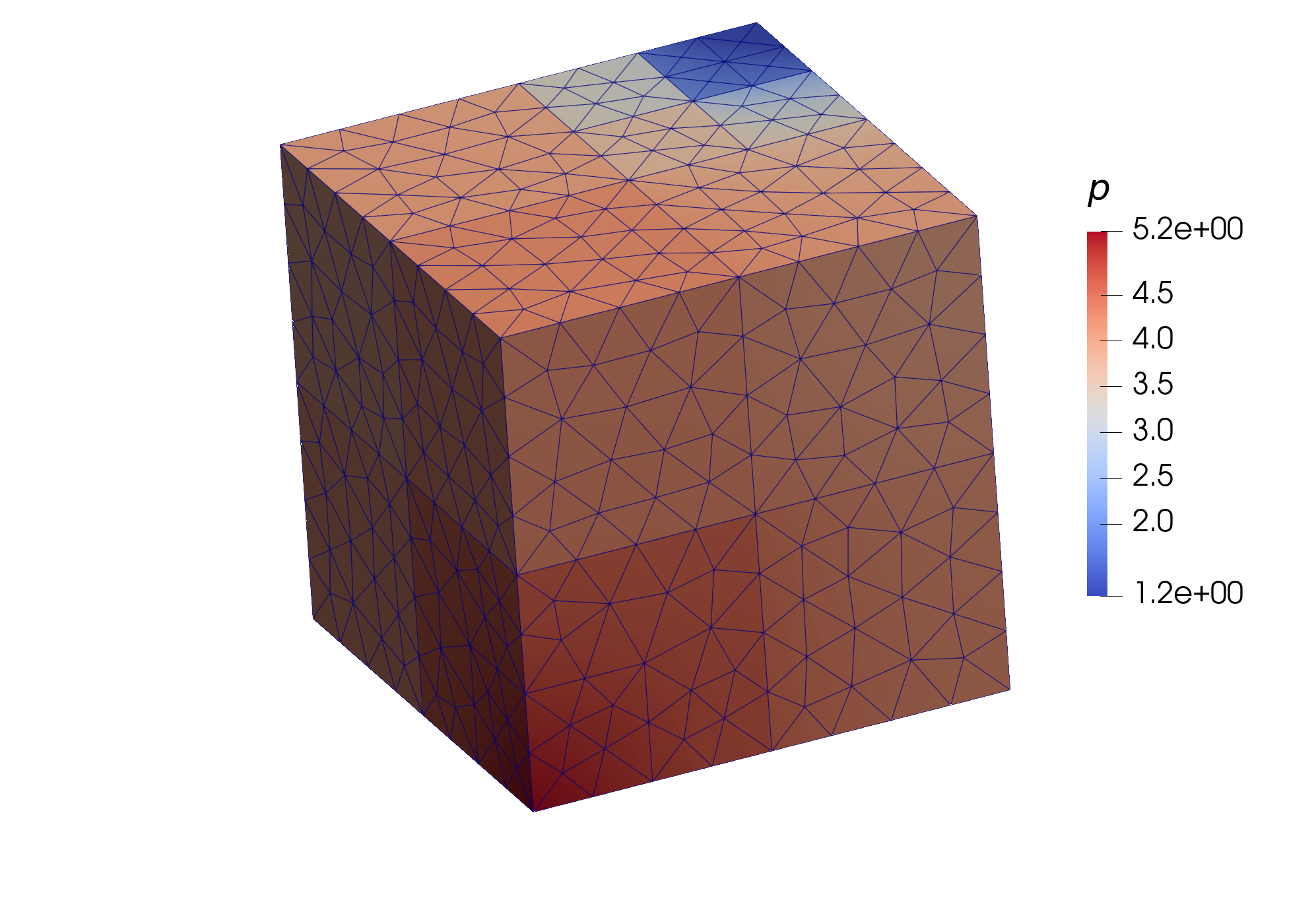}
  \caption{coarse grid}
 \end{subfigure}
 \begin{subfigure}[b]{0.3\textwidth}
  \includegraphics[width=\textwidth]{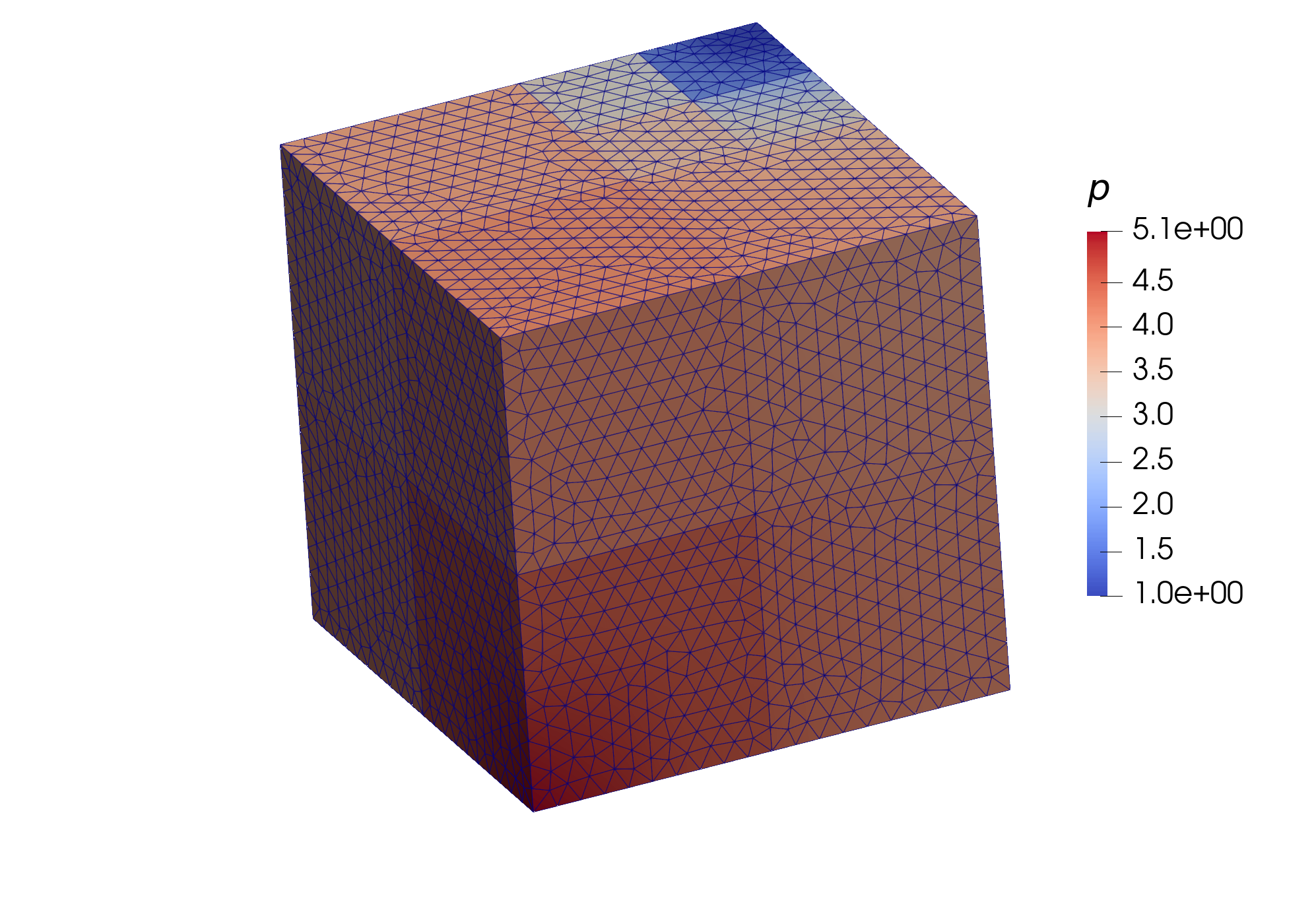}
  \caption{medium grid}
 \end{subfigure}
 \begin{subfigure}[b]{0.3\textwidth}
  \includegraphics[width=\textwidth]{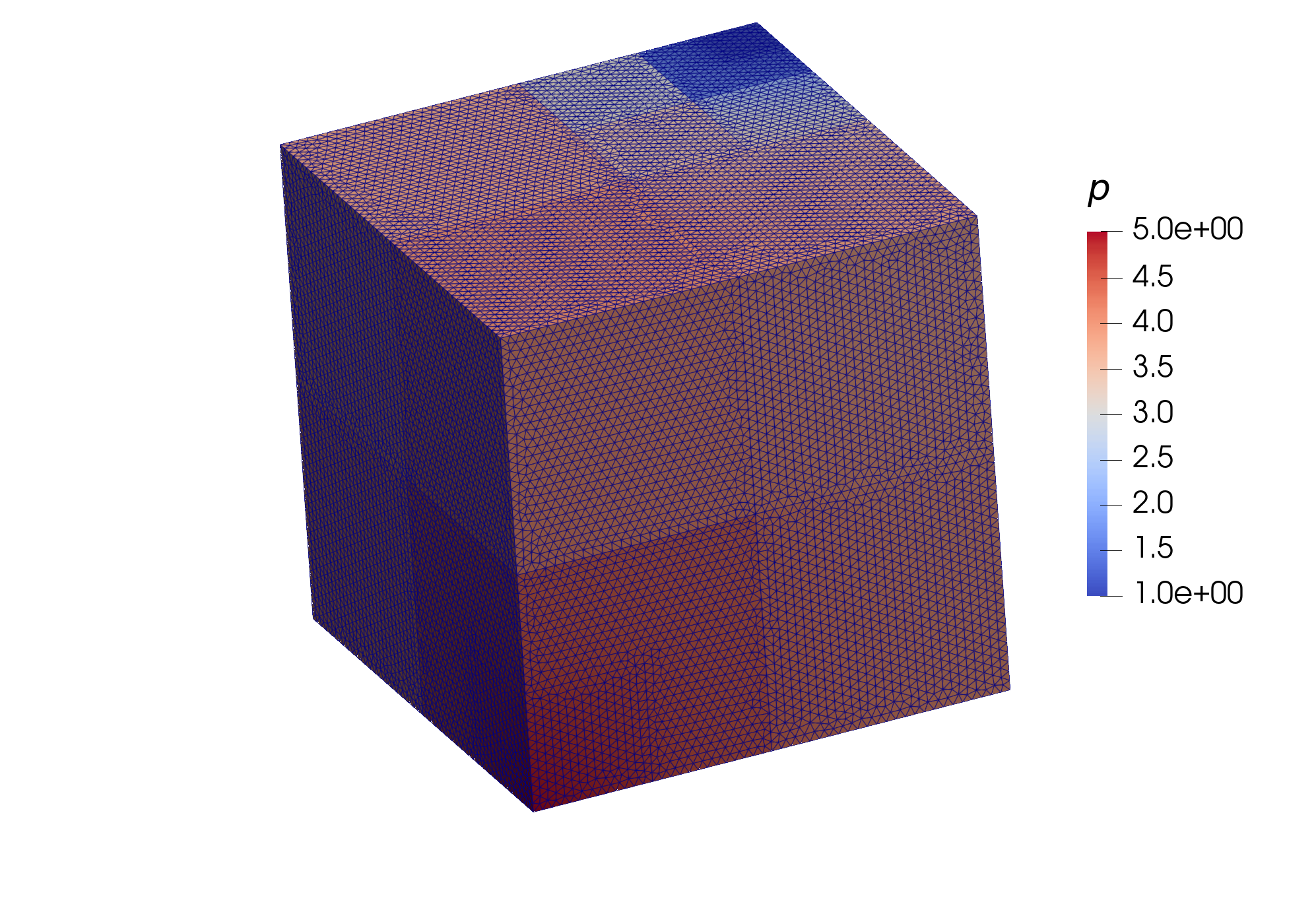}
  \caption{fine grid}
 \end{subfigure}

 \begin{subfigure}[b]{0.3\textwidth}
  \includegraphics[width=\textwidth]{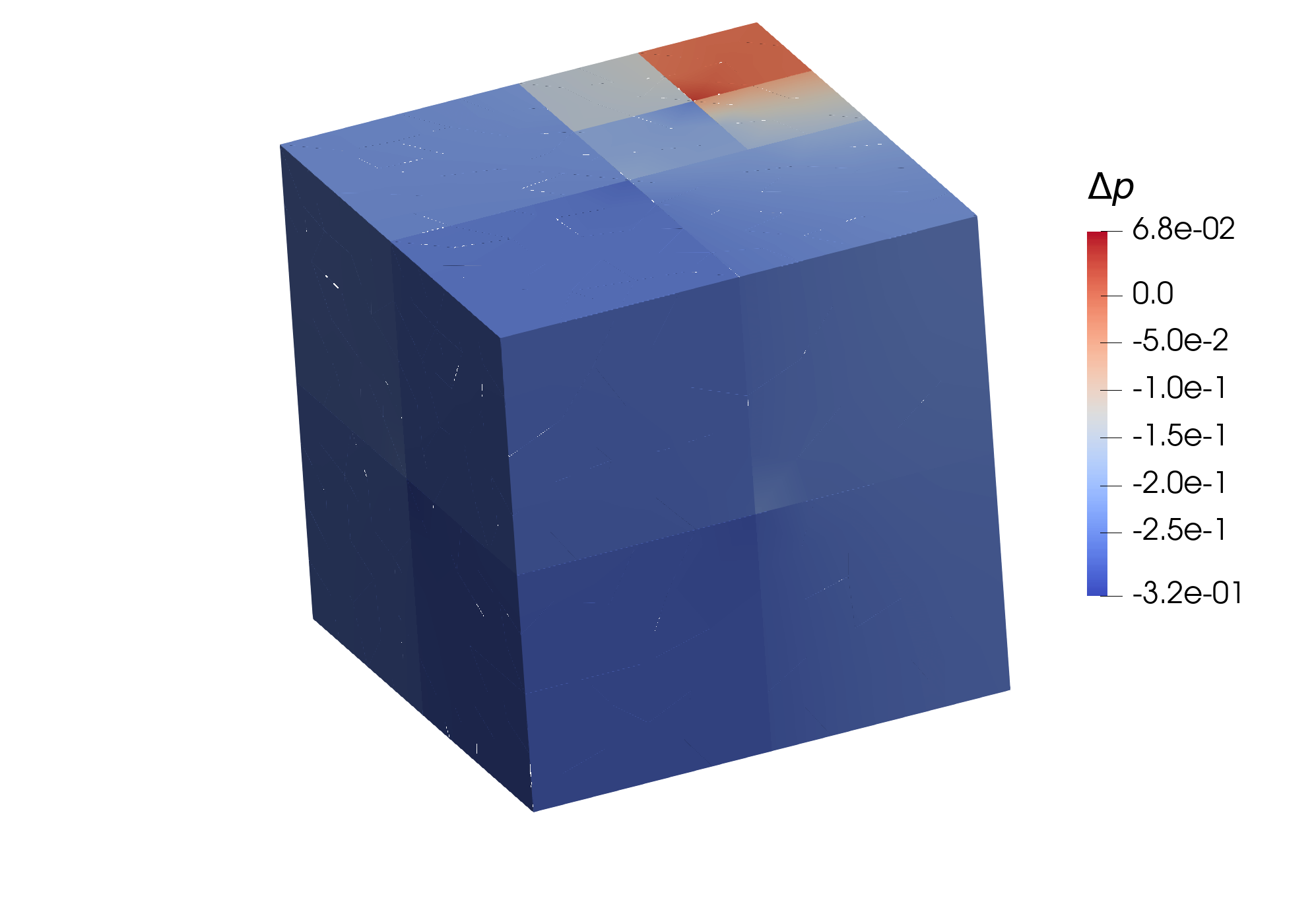}
  \caption{coarse grid}
 \end{subfigure}
 \begin{subfigure}[b]{0.3\textwidth}
  \includegraphics[width=\textwidth]{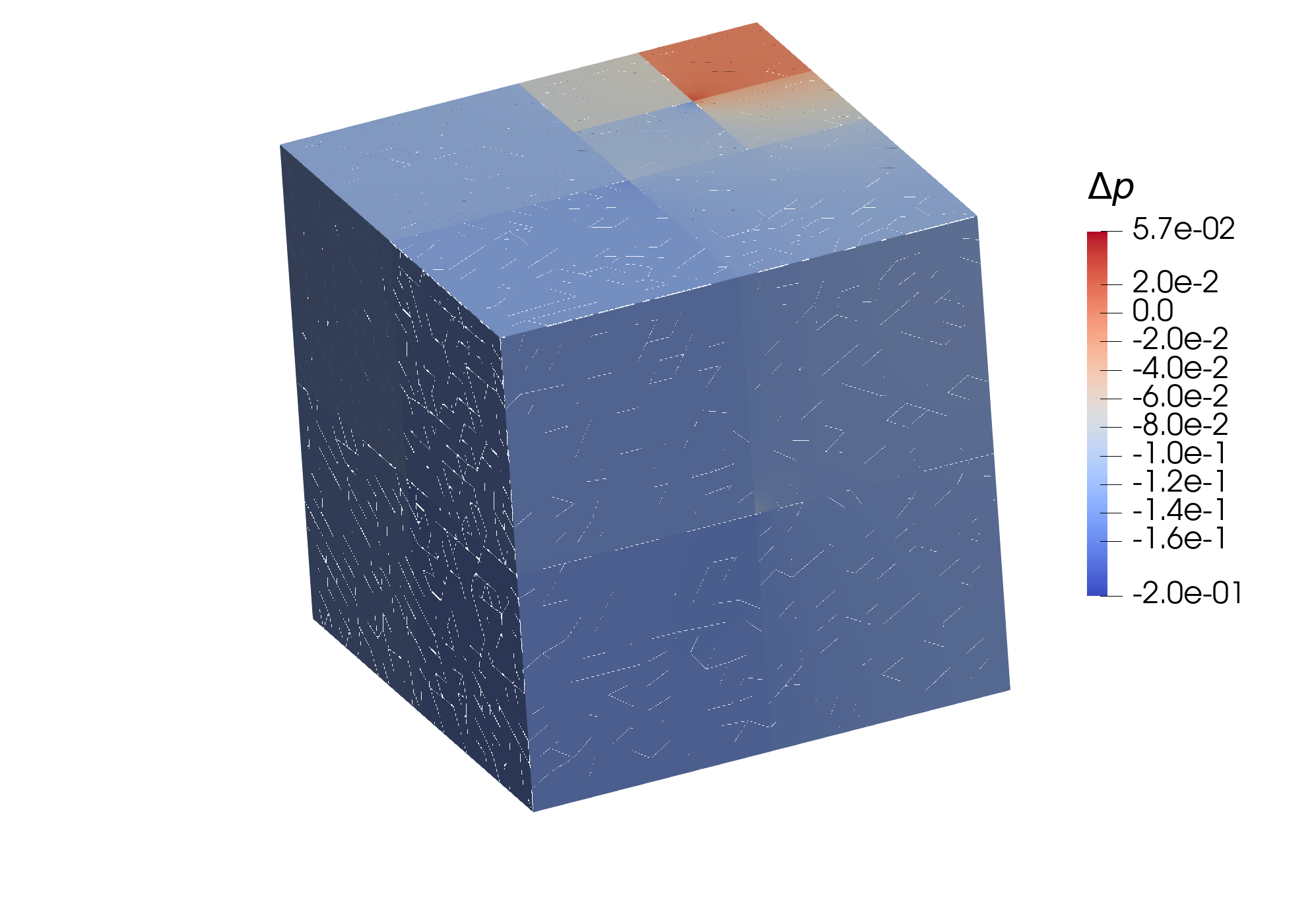}
  \caption{medium grid}
 \end{subfigure}
 \begin{subfigure}[b]{0.3\textwidth}
  \includegraphics[width=\textwidth]{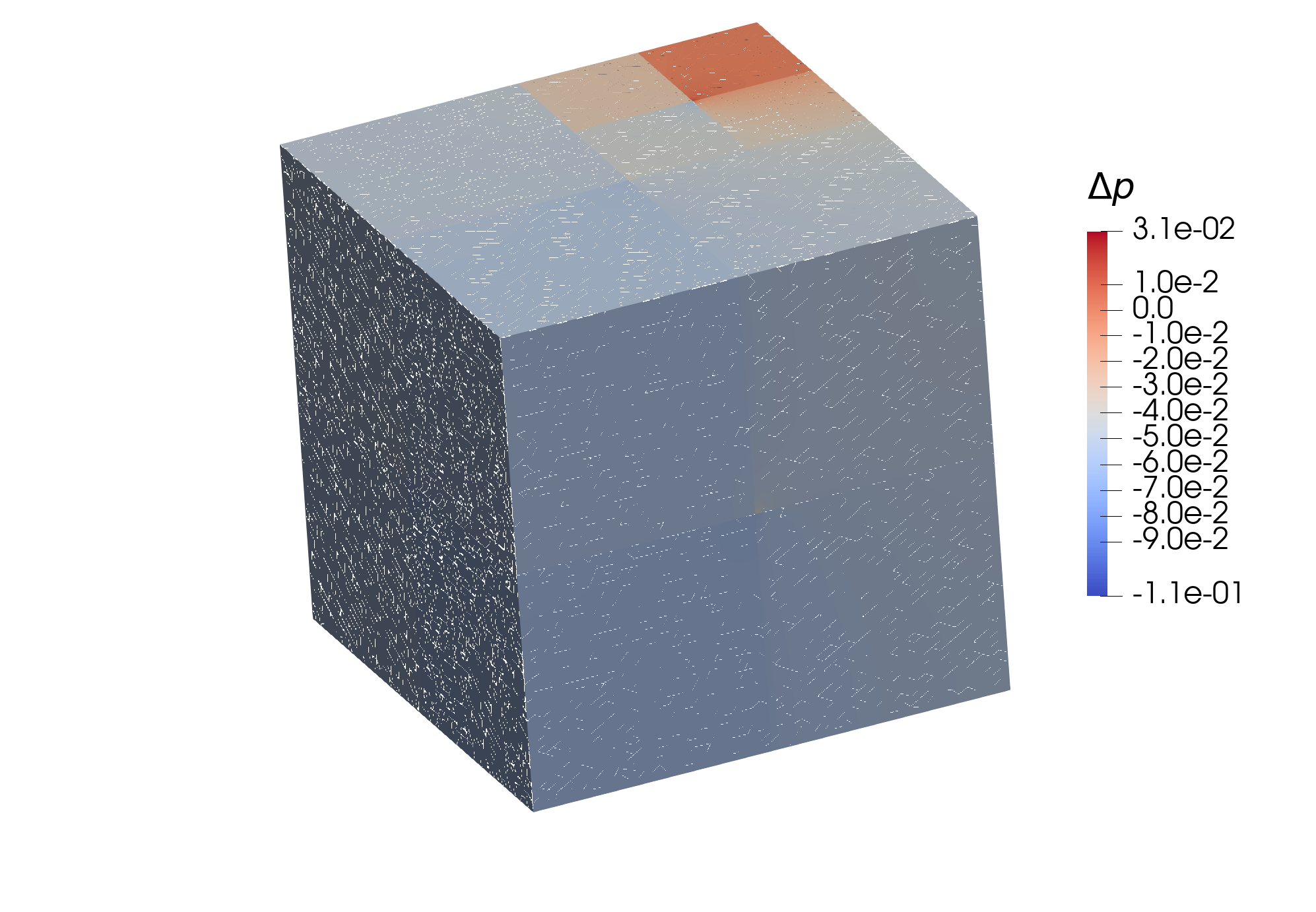}
  \caption{fine grid}
 \end{subfigure}
 \caption{\textbf{Example \ref{ex:regular3D}: three-dimensional test case.} Visualization of the solution $p$ (row $1$) obtained from the proposed method, and the difference $\Delta p$ (row $2$) to the solution obtained from the ebox-dfm in \cite{glaser2022comparison}.
 The coarse grid contains $948$ vertices and $3, 840$ tetrahedrons; the medium grid contains $8, 316$ vertices and $41, 414$ tetrahedrons; the fine grid contains $104, 343$ vertices and $593, 347$ tetrahedrons.
 }
 \label{fig:regular3D_contour}
\end{figure}

\begin{figure}[!htbp]
 \centering
 \begin{subfigure}[b]{0.3\textwidth}
  \includegraphics[width=\textwidth]{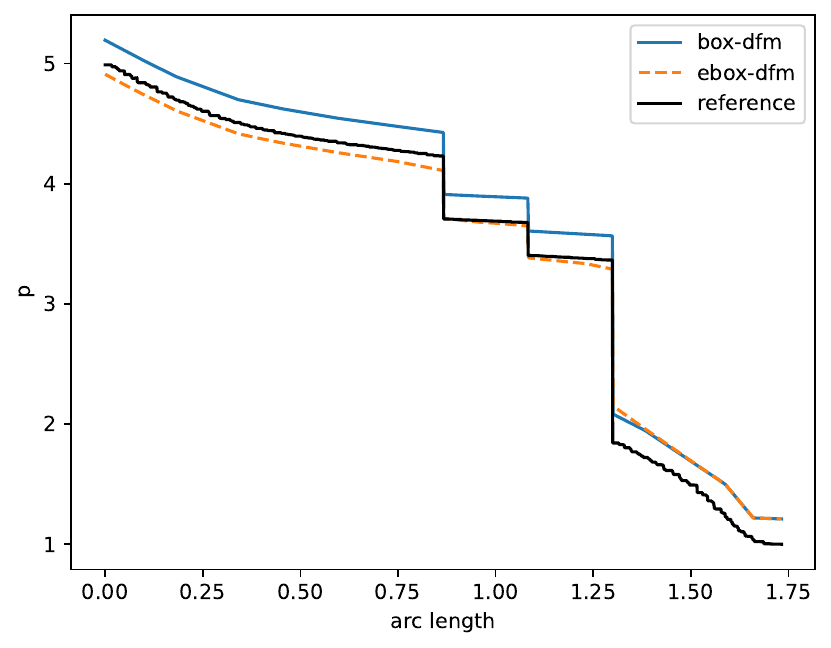}
  \caption{coarse grid}
 \end{subfigure}
 \begin{subfigure}[b]{0.3\textwidth}
  \includegraphics[width=\textwidth]{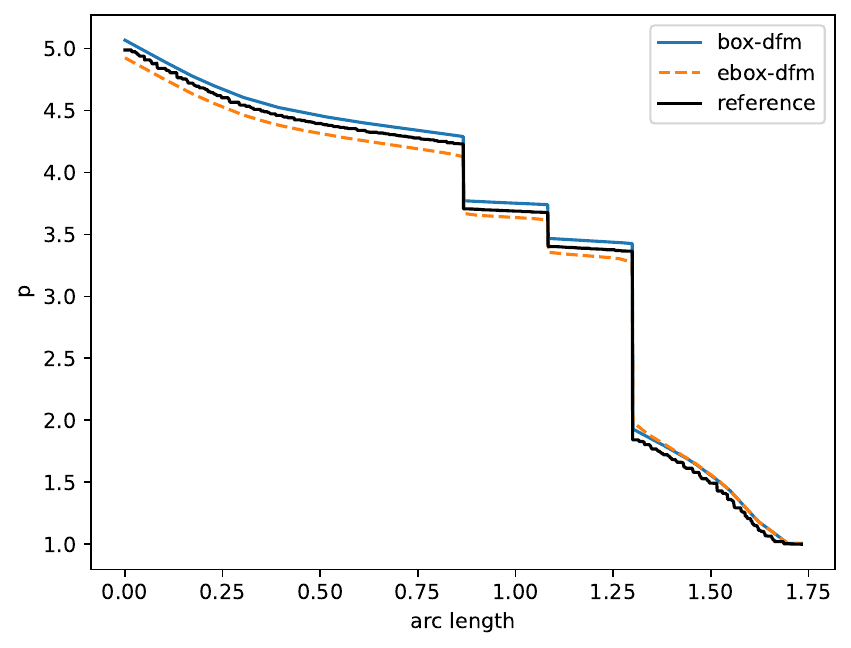}
  \caption{medium grid}
 \end{subfigure}
 \begin{subfigure}[b]{0.3\textwidth}
  \includegraphics[width=\textwidth]{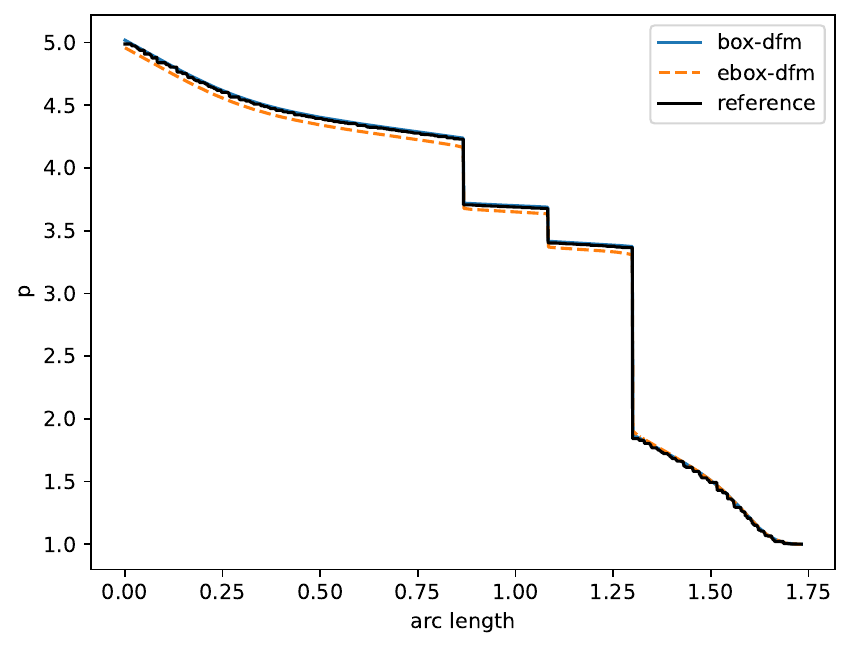}
  \caption{fine grid}
 \end{subfigure}
 \caption{\textbf{Example \ref{ex:regular3D}: three-dimensional test case.} 
 Pressure profiles along the line (0, 0, 0) -- (1, 1, 1).
 The coarse grid contains $948$ vertices and $3, 840$ tetrahedrons; the medium grid contains $8, 316$ vertices and $41, 414$ tetrahedrons; the fine grid contains $104, 343$ vertices and $593, 347$ tetrahedrons.
 The reference, provided by the authors in \cite{berre2021verification}, is obtained from the control volume multi-point flux approximation (MPFA) method with $1,046,566$ elements for the hybrid-dimensional model \cite{koch2021dumux}.
 }
 \label{fig:regular3D_plot}
\end{figure}

\end{exmp}

\begin{exmp}\label{ex:validity}
{Range of validity}

So far, we have tested only cases of barriers with low permeability in both the normal and tangential directions. In this experiment, we investigate the range of validity of the method for barriers with various tangential permeabilities \cite{angot2009asymptotic, zhao2023discrete}.

We consider the computational domain $\Omega=[0,1]^2$ with matrix permeability $\mathbf{K}_m=\mathbf{I}$.
The domain contains four barriers, $\gamma_1, \gamma_2, \gamma_3$, and $\gamma_4$, whose axes are given by: (0.3, 0.2) -- (1.0, 0.2), (0.0, 0.4) -- (0.7, 0.4), (0.3, 0.6) -- (1.0, 0.6), and (0.0, 0.8) -- (0.7 0.8), respectively. 
All barriers have a uniform aperture $a=10^{-3}$ and normal permeability $k_{n}=10^{-3}$.
We set the tangential permeability of $\gamma_2$ and $\gamma_4$ equal to their normal direction, i.e., $k_{\tau}=10^{-3}$.
For the barriers $\gamma_1$ and $\gamma_3$, their tangential permeabilities take various values: $k_{\tau}=10^{-3}, 1$ or $10^{3}$.

Two sub-cases are tested as in \cite{angot2009asymptotic, zhao2023discrete}.
In sub-case (a), the left and right boundaries are impermeable, i.e., $q_N=0$, and the top and bottom boundaries are set as Dirichlet with $g_D=1$ and $g_D=0$, respectively.
In sub-case (b), the boundary condition is purely Dirichlet with $g_D(x,y)=(2x-1)(3x-1)$.

We conduct the computation on a grid containing $11, 957$ vertices and $23, 512$ triangles using our extended box-dfm and ebox-dfm for all scenarios.
{\color{black}The pressure contours obtained from our proposed method for both cases are shown in Figure \ref{fig:validity_contour}.}
We present the pressure profiles along the line (0.65, 0) -- (0.65,1) in Figure \ref{fig:validity}, and compare them with the reference solution obtained from the box method for equi-dimensional model on an extremely refined mesh.

As can be seen from the figures, the model errors are small if the tangential permeability of the barriers is no greater than that of porous matrix, provided the aperture of the barriers is small.
However, for the scenarios where barriers have high tangential permeability, the treatment of neglecting tangential flow in the proposed model is not satisfactory. 
Therefore, this extended box-dfm method should not be used for barriers that have high tangential permeability, {\color{black} a situation that is expected to be of little relevance in practice}.

{\color{blue}
\begin{figure}[!htbp]
 \centering
 \begin{subfigure}[b]{0.49\textwidth}
  \includegraphics[width=\textwidth]{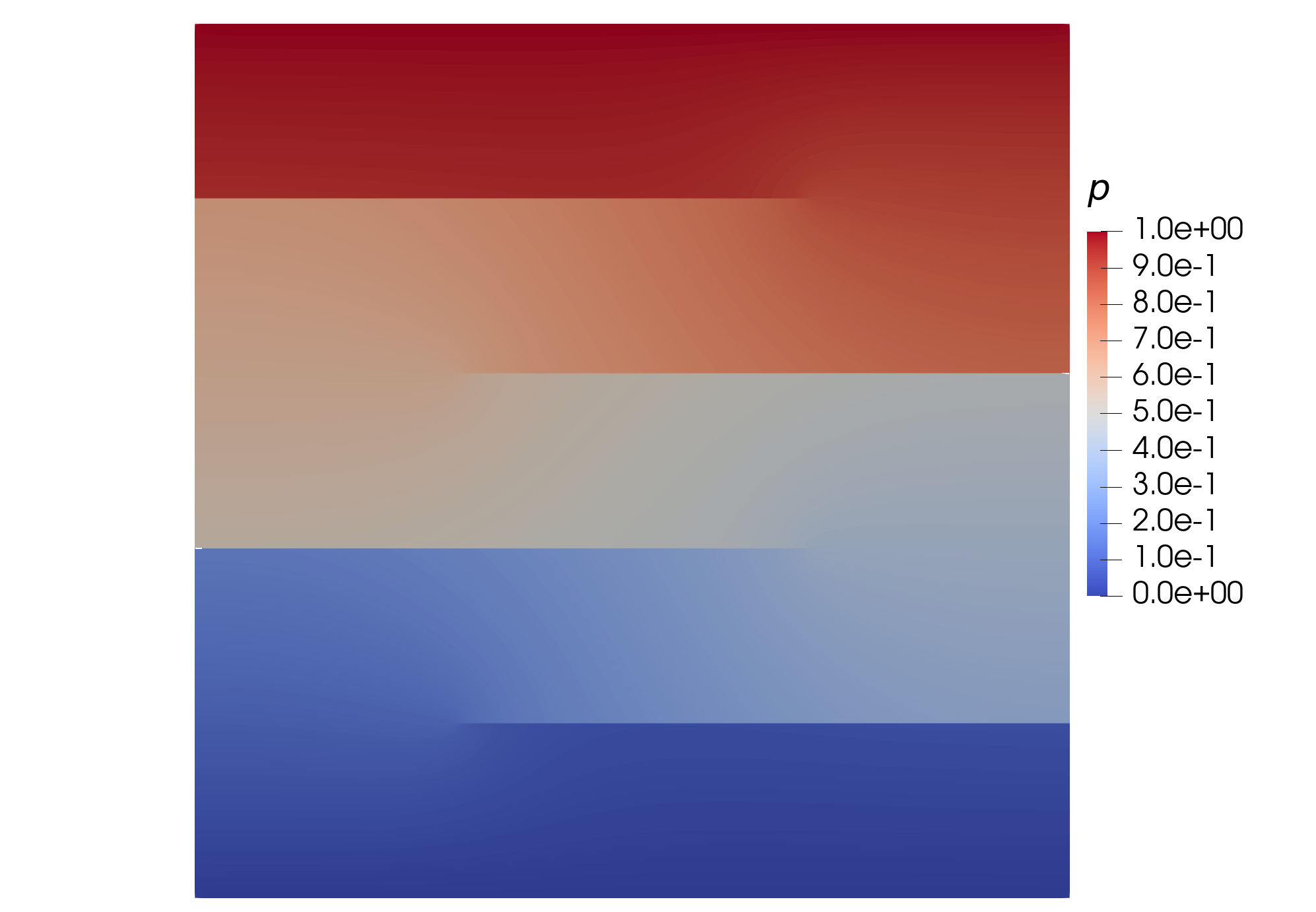}
  \caption{case (a)}
 \end{subfigure}
 \begin{subfigure}[b]{0.49\textwidth}
  \includegraphics[width=\textwidth]{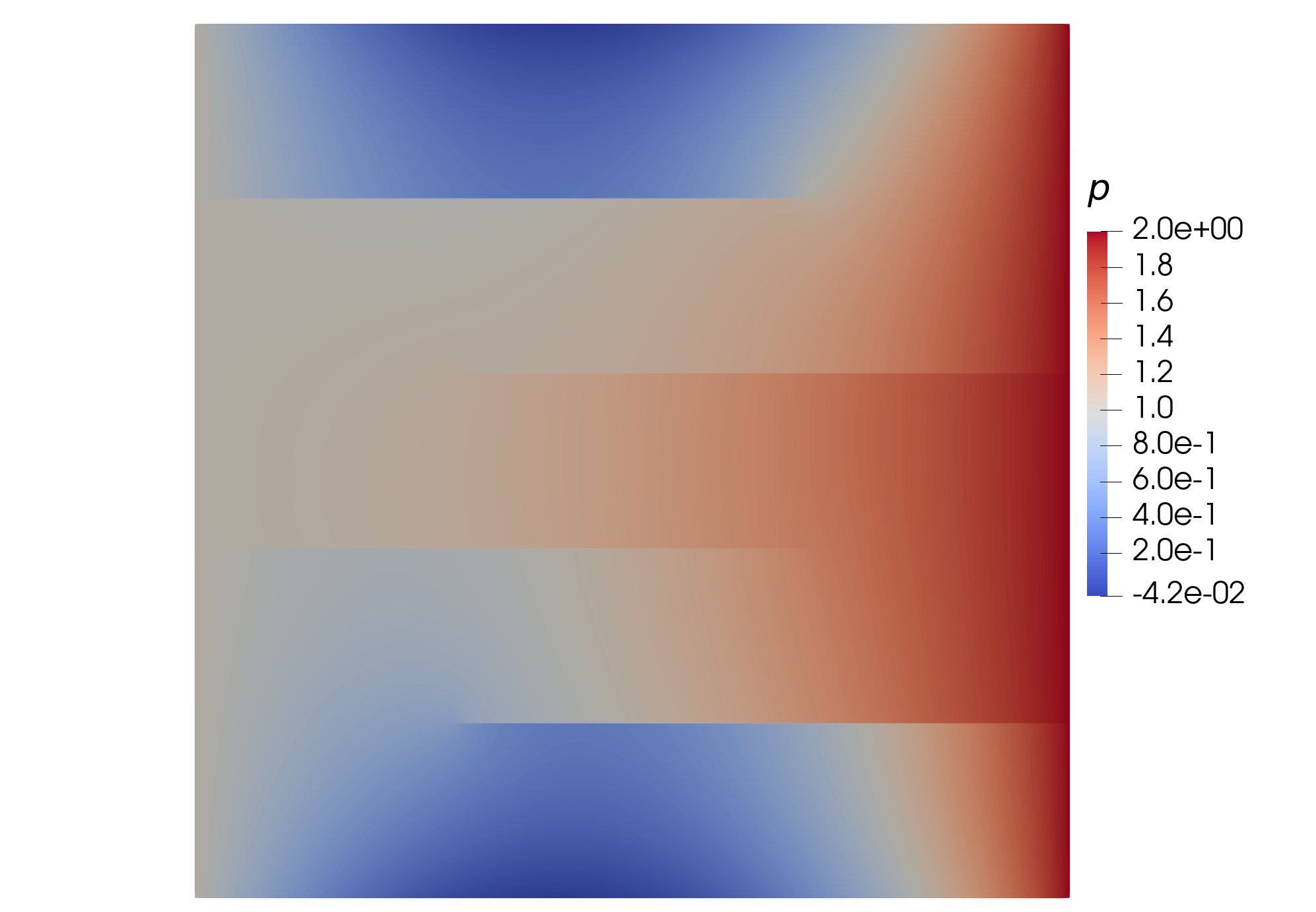}
  \caption{case (b)}
 \end{subfigure}
 \caption{\textbf{Example \ref{ex:validity}: range of validity.} 
 Visualization of the solution $p$ obtained with the proposed method.
 Both cases contain $11, 957$ vertices and $23, 512$ triangles.
 {\color{black}The tangential permeability $k_{\tau}=10^{-3}$ is adopted in both cases.}
}
 \label{fig:validity_contour}
\end{figure}}

\begin{figure}[!hbpt]
 \centering
 \begin{subfigure}[b]{0.3\textwidth}
  \includegraphics[width=\textwidth]{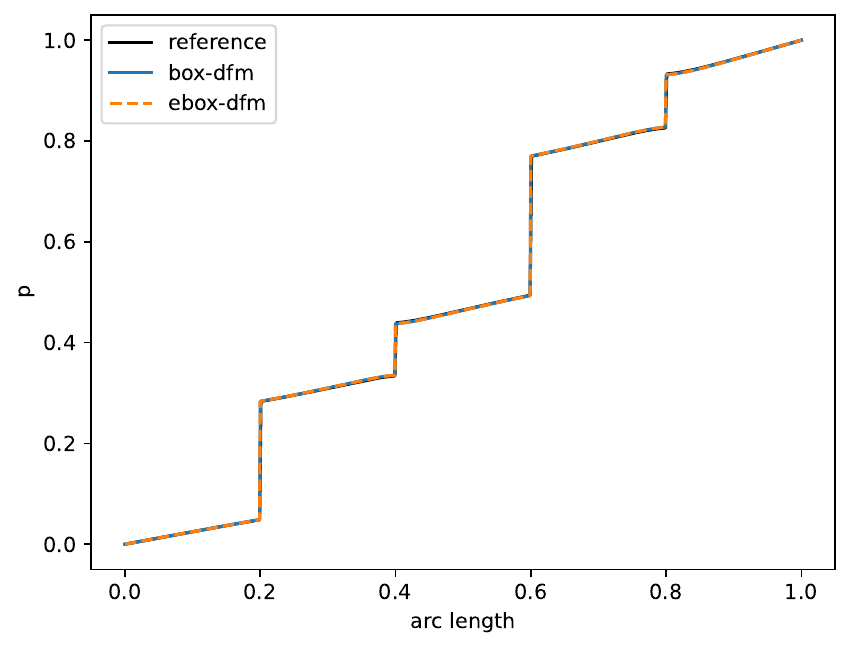}
  \caption{$k_\tau=10^{-3}$ in case (a)}
 \end{subfigure}
 \begin{subfigure}[b]{0.3\textwidth}
  \includegraphics[width=\textwidth]{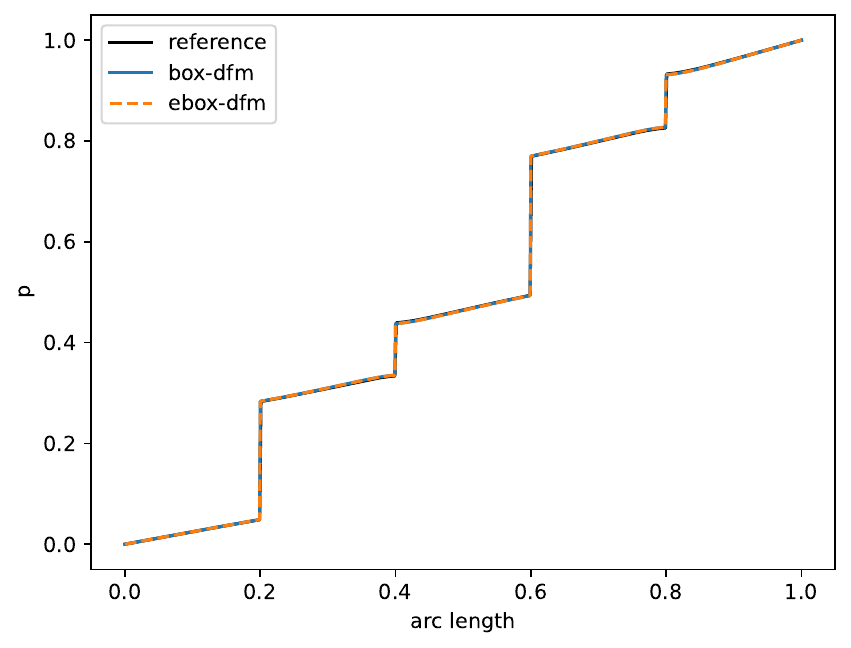}
  \caption{$k_{\tau}=1$ in case (a)}
 \end{subfigure}
\begin{subfigure}[b]{0.3\textwidth}
  \includegraphics[width=\textwidth]{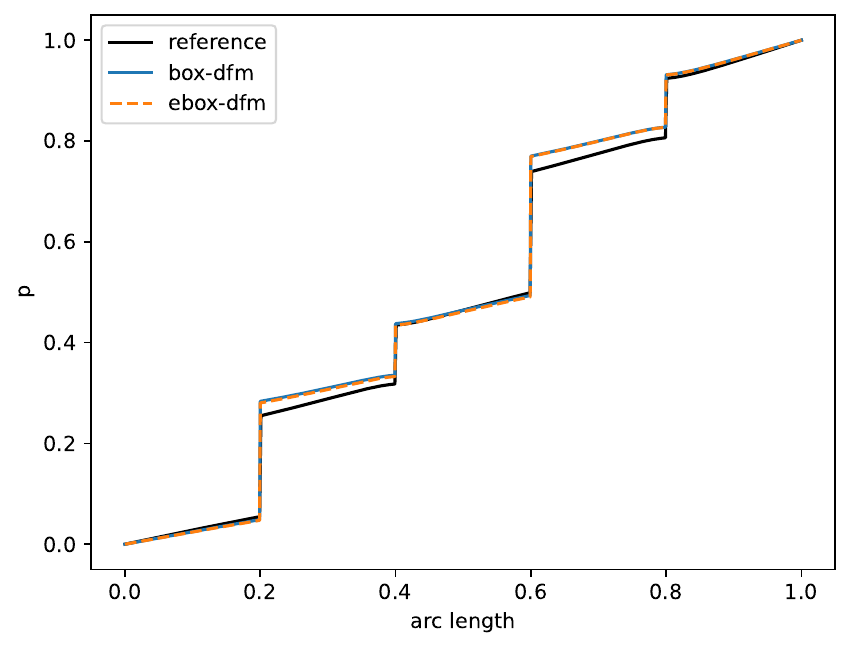}
  \caption{$k_{\tau}=10^{3}$ in case (a)}
 \end{subfigure}\\
  \begin{subfigure}[b]{0.3\textwidth}
  \includegraphics[width=\textwidth]{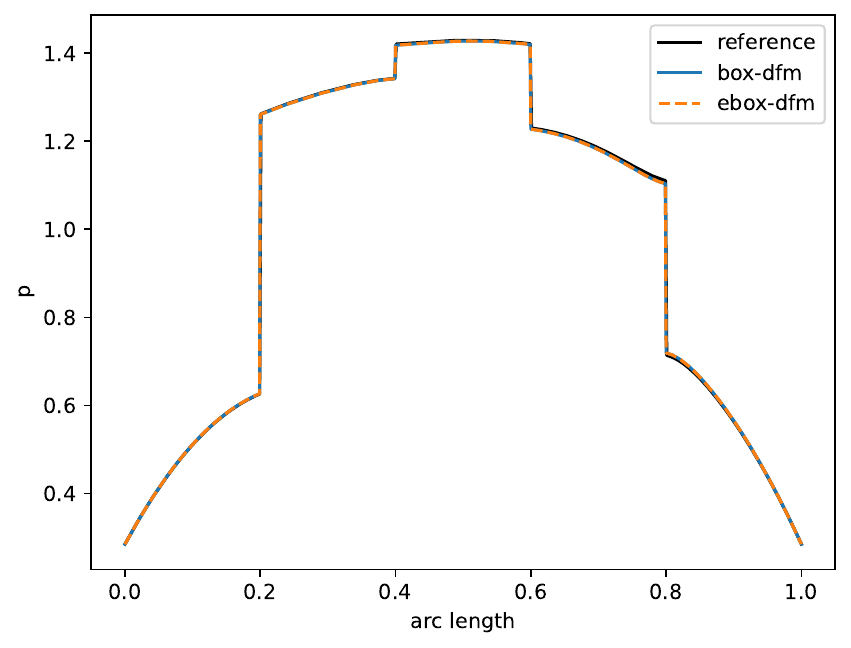}
  \caption{$k_{\tau}=10^{-3}$ in case (b)}
 \end{subfigure}
 \begin{subfigure}[b]{0.3\textwidth}
  \includegraphics[width=\textwidth]{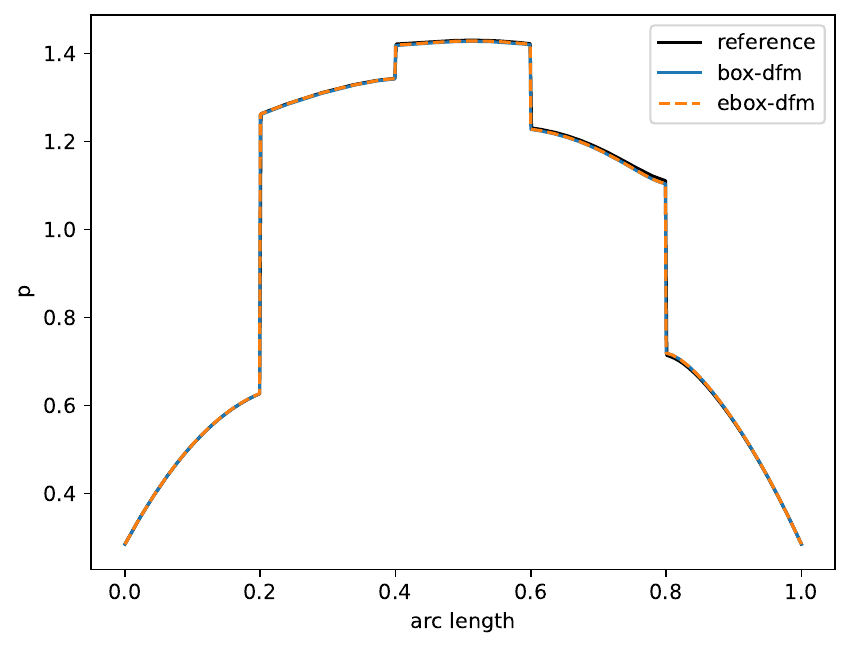}
  \caption{$k_{\tau}=1$ in case (b)}
 \end{subfigure}
\begin{subfigure}[b]{0.3\textwidth}
  \includegraphics[width=\textwidth]{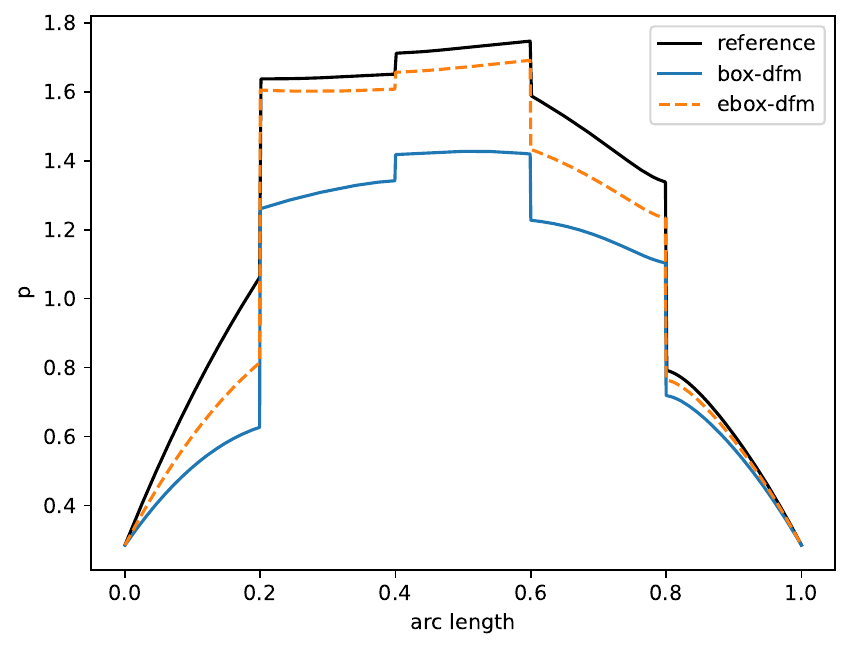}
  \caption{$k_{\tau}=10^{3}$ in case (b)}
 \end{subfigure} 
 \caption{\textbf{Example \ref{ex:validity}: range of validity} 
 Pressure profiles along the line (0.65, 0) -- (0.65,1).
 The box-dfm and ebox-dfm are computed on a grid containing $11, 957$ vertices and $23, 512$ triangles.
 The reference solution is obtained from the box method with $269, 152$ vertices and $538, 164$ triangles for the equi-dimensional model.}
 \label{fig:validity}
\end{figure}

\end{exmp}

\section{Summary}\label{Sect:summary}
{\color{black}
In this work, we have proposed a very simple extension of the Box-DFM to include blocking barriers, by adopting a broken linear Lagrange finite element space as the trial function space and using divided sub-boxes to apply the Gauss divergence theorem. The proposed method is identical to the original Box-DFM in the absence of blocking barriers and only requires minimal additional degrees of freedom, compared with other variants of Box-DFMs, to accommodate the pressure discontinuity across barriers. In addition, it inherits the local mass conservation, flexibility in the computational grid, and the symmetric positive-definiteness of the stiffness matrix from the original method.

Several widely practiced benchmark problems are tested. 
Numerical experiments exhibit good performance of our method when dealing with blocking barriers. 
It should be noted that our method is based on the interface model, which neglects the pressure jump across highly permeable fractures and the tangential flow along blocking barriers. 
Thus, the method is not suitable for handling cases of fractures with low permeability in the normal direction and high permeability in the tangential direction. 
This assertion is verified in the last numerical example. 
Moreover, it is not difficult to extend the treatment for barriers in two-phase flow. In future work, we will study the performance of such an extended Box-DFM in applications of multi-phase flows.

}

\begin{appendices}
\setcounter{figure}{0} 
\setcounter{equation}{0} 
\renewcommand{\theequation}{\thesection.\arabic{equation}} 
\renewcommand\thefigure{\Alph{section}\arabic{figure}}  
\section{The symmetry and positive definiteness of the stiffness matrix of \eqref{eq:BoxB}}\label{app:SPD}

In this appendix, we show the symmetry and positive definiteness of the stiffness matrix in \eqref{eq:BoxB}, provided that $\mathbf{K}_m$ is symmetric positive definite.
For simplicity, we assume $\mathbf{K}_m$ is a constant matrix and $\frac{k_b}{a}>0$ is a constant scalar. 
However, this condition can be relaxed to allow for piecewise constant $\mathbf{K}_m$ on cells and piecewise constant $\frac{k_b}{a}\geq0$ on edges of the primal mesh, provided that no area is enclosed by impermeable barriers.

\begin{lem}
Let $\phi_i\in\overline{V}_h$ the Lagrange basis and $B_{i}\in \overline{\mathcal{B}}$ the box associated with the node $i$, for $i=1,\ldots, \#\overline{\mathcal{B}}$.
Then we have 
\begin{equation}\label{eq:SPD_Eq}
\sum_{T\in\mathcal{T}}\int_{T} \mathbf{K}_m\nabla v\cdot\nabla\phi_i dxdy=-\int_{\partial B_i\setminus\Gamma} \mathbf{K}_m\nabla v\cdot\mathbf{n}ds,\quad \forall v\in\overline{V}_h,   
\end{equation}
where $\mathbf{n}$ denotes the unit outer normal of $\partial B_{i}$.  
\end{lem}

\begin{proof}
It is proved in \cite{hackbusch1989first} that,
\begin{equation}
\int_{T}\mathbf{K}_m\nabla v\cdot\nabla\phi_{i}dxdy=-\int_{\partial B_i\cap T^{o}}\mathbf{K}_m\nabla v\cdot\mathbf{n} ds,
\end{equation}
where $T^o$ is the interior of the cell $T\in\mathcal{T}$.
Summing over all $T\in\mathcal{T}$, one will get \eqref{eq:SPD_Eq}.
\end{proof}

Based on the above lemma, we can prove the following result.

\begin{thm}
Assume $\mathbf{K}_m$ is a symmetric positive-definite constant matrix and $\frac{k_b}{a}$ is a positive constant scalar.
The stiffness matrix in the method \eqref{eq:BoxB} is symmetric positive-definite.
\end{thm}

\begin{proof}

Let $\phi_i\in\overline{V}_h$ the Lagrange basis and $B_{i}\in \overline{\mathcal{B}}$ the box associated with the node $i$, for $i=1,\ldots, N$, where $N= \#\overline{\mathcal{B}}$.
One can calculate that, for the first term in \eqref{eq:BoxB},
\begin{equation*}
\begin{split}
-\int_{\partial B_i\setminus\Gamma}\mathbf{K}_m\nabla\phi_{j}\cdot\mathbf{n}ds = \sum_{T\in\mathcal{T}}\int_{T}\mathbf{K}_m\nabla\phi_{j}\cdot\nabla\phi_{i}dxdy = -\int_{\partial B_j\setminus\Gamma}\mathbf{K}_m\nabla\phi_{i}\cdot\mathbf{n}ds.
\end{split}
\end{equation*}
As for the second term in \eqref{eq:BoxB}, it is non-zero only if the nodes $i$ and $j$ are at the same vertex or on the same edge on $\Gamma$.
We discuss three possibilities ($i\neq j$):
\begin{enumerate}
\item If the nodes $i$ and $j$ are on the same vertex, then 
\begin{equation*}
    -\int_{\partial B_i\cap\Gamma} \frac{k_b}{a} (\phi_{j}^{+}-\phi_{j}^{-})ds = -\frac34\int_{\partial B_i\cap\partial B_{j}}\frac{k_b}{a}ds = -\int_{\partial B_j\cap\Gamma} \frac{k_b}{a} (\phi_{i}^{+}-\phi_{i}^{-})ds.
\end{equation*}
\item If the nodes $i$ and $j$ are two vertices of an edge $e$ and on the same side of $\Gamma$, then 
\begin{equation*}
    -\int_{\partial B_i\cap\Gamma} \frac{k_b}{a} (\phi_{j}^{+}-\phi_{j}^{-})ds = \frac{1}{8}\int_{e}\frac{k_b}{a}ds = -\int_{\partial B_j\cap\Gamma} \frac{k_b}{a} (\phi_{i}^{+}-\phi_{i}^{-})ds.
\end{equation*}
\item If the node $i$ and $j$ are two vertices of an edge $e$ and on the two sides of $\Gamma$, then 
\begin{equation*}
    -\int_{\partial B_i\cap\Gamma} \frac{k_b}{a} (\phi_{j}^{+}-\phi_{j}^{-})ds = -\frac18\int_{e}\frac{k_b}{a}ds = -\int_{\partial B_j\cap\Gamma} \frac{k_b}{a} (\phi_{i}^{+}-\phi_{i}^{-})ds.
\end{equation*}
\end{enumerate}

In all cases, we have 
\begin{equation*}
    -\int_{\partial B_i\cap\Gamma} \frac{k_b}{a} (\phi_{j}^{+}-\phi_{j}^{-})ds = -\int_{\partial B_j\cap\Gamma} \frac{k_b}{a} (\phi_{i}^{+}-\phi_{i}^{-})ds.
\end{equation*}
Therefore, the stiffness matrix of the method \eqref{eq:BoxB} is symmetric.

Take an arbitrary grid function $\{v_{i}\}_{i=1}^{N}$, and denote $v=\sum_{i=1}^{N}v_{i}\phi_{i}\in\overline{V}_h$.
One can calculate that, for the first term in \eqref{eq:BoxB}, 
\begin{equation*}
\begin{split}
-\sum_{i,j=1}^{N}v_{i}v_j\int_{\partial B_{i}\setminus\Gamma}\mathbf{K}_m\nabla\phi_{j}\cdot\mathbf{n}ds &= -\sum_{i=1}^{N}v_{i}\int_{\partial B_{i}\setminus\Gamma}\mathbf{K}_m\nabla v\cdot\mathbf{n}ds \\
&= \sum_{i=1}^{N}v_{i}\sum_{T\in\mathcal{T}}\int_{T}\mathbf{K}_m\nabla v\cdot\nabla\phi_{i}dxdy\\
&=\sum_{T\in\mathcal{T}}\int_{T}\mathbf{K}_m\nabla v\cdot\nabla vdxdy\geq 0,
\end{split}
\end{equation*}
where the equality holds only if $v$ is constant on every subdomain separated by $\Gamma$ (if no area is enclosed by $\Gamma$, then $v$ must be a constant on the whole domain.)

To facilitate the discussion on the second term in \eqref{eq:BoxB}, we define $\mathcal{E}^\Gamma=\{e:e=\partial T \cap \Gamma, \, \forall T \in \mathcal{T}_h\}$ the collection of edges on $\Gamma$, and let $\jl v \jr_{a} = v^{+}_a-v^{-}_{a}, \jl v \jr_b=v_{b}^{+}-v_{b}^{-}$ the jumps of $v$ at $a$ and $b$ on the two sides of $e\in\mathcal{E}^{\Gamma}$, where $a$ and $b$ are the two vertices of $e$ and $\pm$ are the limits specified for every $e$.
Then, one can calculate that, 
\begin{equation*}
\begin{split}
&-\sum_{i,j=1}^{N}v_{i}v_{j}\int_{\partial B_i\cap\Gamma}\frac{k_b}{a}(\phi_{j}^{+}-\phi_{j}^{-})ds = -\sum_{i=1}^{N}v_{i}\int_{\partial B_i\cap\Gamma}\frac{k_b}{a}(v^{+}-v^{-})ds\\
&=\frac12\sum_{e\in\mathcal{E}^{\Gamma}}\int_{e}\frac{k_b}{a}\left(\jl v \jr_a(\frac34\jl v\jr_a+\frac14\jl v\jr_b)+\jl v\jr_b(\frac14\jl v\jr_a+\frac34 \jl v\jr_b)\right)ds\\
&=\frac{1}{8}\sum_{e\in\mathcal{E}^{\Gamma}}\int_{e}\frac{k_b}{a}\left( 2\jl v\jr_a^2+2\jl v\jr_b^2+(\jl v\jr_a+\jl v\jr_b)^2 \right)ds\geq 0,
\end{split} 
\end{equation*}
where the equality holds only if the jump of $v$ vanish on $\Gamma$, i.e., $v$ is continuous across $\Gamma$.

Summing up the two inequalities, one can conclude that
\begin{equation*}
-\sum_{i,j=1}^{N}v_{i}v_j\left(\int_{\partial B_{i}\setminus\Gamma}\mathbf{K}_m\nabla\phi_{j}\cdot\mathbf{n}ds+ \int_{\partial B_i\cap\Gamma}\frac{k_b}{a}(\phi_{j}^{+}-\phi_{j}^{-})ds\right)\geq0,
\end{equation*}
where the equality holds only if $v$ is constant on the whole domain.
Therefore, the stiffness matrix of the method \eqref{eq:BoxB} is positive-definite, if the Dirichlet condition on a non-zero measure boundary is specified.
\end{proof}

\end{appendices}

\end{document}